\documentclass{amsart}
\usepackage{amsmath,amsthm,indentfirst}
\usepackage{amssymb}
\usepackage{amsfonts,dsfont}
\usepackage{algorithm}
\usepackage{algpseudocode}
\usepackage{pifont}
\usepackage{a4}
\usepackage{amscd}
\usepackage[latin1]{inputenc}
\usepackage{url}
\theoremstyle{plain}
\newtheorem{thm}{Theorem}
\newtheorem{cor}[thm]{Corollary} 

\newenvironment{cor'}
  {\addtocounter{thm}{-1}%
   \begin{cor}}
  {\end{cor}}

\newtheorem{lem}[thm]{Lemma}

\newenvironment{lem'}
  {\addtocounter{thm}{-1}%
   \begin{lem}}
  {\end{lem}}

\newtheorem{prop}[thm]{Proposition}

\newenvironment{prop'}
  {\addtocounter{thm}{-1}%
   \begin{prop}}
  {\end{prop}}

\newtheorem{remark}[thm]{Remark}

\def\bbz{\mathbb{Z}}
\def\bbq{\mathbb{Q}}

\def\bbr{\mathbb{R}}
\def\bba{\mathbb{A}}

\def\bbc{\mathbb{C}}

\def\bbe{\mathbb{E}}
\def\bbh{\mathbb{H}}

\def\bbg{\mathbb{G}}

\def\bbv{\mathbb{V}}
\def\bbu{\mathbb{U}}

\def\bbl{\mathbb{L}}

\def\bbw{\mathbb{W}}

\def\gcal{\mathcal{G}}

\def\ucal{\mathcal{U}}

\def\ocal{\mathcal{O}}
\def\ocal{\mathcal{O}}

\def\acal{\mathcal{A}}

\def\ccal{\mathcal{C}}

\def\cal{\mathcal{H}}
\def\lcal{\mathcal{L}}
\def\pcal{\mathcal{P}}

\def\hfr{\mathfrak{h}}
\def\afr{\mathfrak{a}}

\def\Pfr{\mathfrak{P}}

\def\pfr{\mathfrak{p}}

\def\Wfr{\mathfrak{W}}
\def\Vfr{\mathfrak{V}}

\def\ufr{\mathfrak{u}}
\def\gl{\mathfrak{gl}}
\def\gfr{\mathfrak{g}}

\def\zfr{\mathfrak{z}}
\def\f{\mathfrak{f}}

\def\ybf{\mathbf{y}}
\def\xbf{\mathbf{x}}

\def\wbf{\mathbf{w}}
\def\vbf{\mathbf{v}}
\def\ebf{\mathbf{e}}
\def\ibf{\mathbf{i}}
\def\fbf{\mathbf{f}}
\def\lbf{\mathbf{l}}

\DeclareMathOperator\Spec{Spec}

\DeclareMathOperator\SL{SL}
\DeclareMathOperator\GL{GL}

\DeclareMathOperator\Ad{Ad}
\DeclareMathOperator\ad{ad}
\DeclareMathOperator\Lie{Lie}

\DeclareMathOperator\rk{rank}

\DeclareMathOperator\diag{diag}
\DeclareMathOperator\Cay{Cay}

\DeclareMathOperator\supp{supp}

\DeclareMathOperator{\Hom}{Hom}

\newcommand{\lin}[2]{\Psi_{#1}^{#2}}
\newcommand{\linValue}[3]{\lin{#1}{#2}(\pi_{#2}(#3))}
\newcommand{\wt}[1]{\widetilde{#1}}
\newcommand{\wh}[1]{\widehat{#1}}
\def\h{\hspace{1mm}}

\def\vare{\varepsilon}
\def\be{\begin{equation}}
\def\ee{\end{equation}}
\def\one{\mathds{1}}

\def\ugfr{\underline{\gfr}}
\def\oPhi{\overline{\Phi}}
\def\acts{\curvearrowright}

\begin{document}
\title{Super-approximation, II:\\ the $p$-adic and bounded power of square-free integers cases.}

\author{Alireza Salehi Golsefidy}
\address{Mathematics Dept, University of California, San Diego, CA 92093-0112}
\email{golsefidy@ucsd.edu}
\thanks{A. S-G. was partially supported by the NSF grants DMS-1160472, DMS-1303121, DMS-1602137 and A. P. Sloan Research Fellowship.}
\subjclass{22E40}
\date{\today}
\begin{abstract}
Let $\Omega$ be a finite symmetric subset of $\GL_n(\bbz[1/q_0])$, $\Gamma:=\langle \Omega \rangle$, and let $\pi_m$ be the group homomorphism induced by the quotient map $\bbz[1/q_0]\rightarrow \bbz[1/q_0]/m\bbz[1/q_0]$. Then the family of Cayley graphs $\{{\rm Cay}(\pi_m(\Gamma),\pi_m(\Omega))\}_m$ is a family of expanders as $m$ ranges over fixed powers of square-free integers and powers of primes that are coprime to $q_0$ if and only if the connected component of the Zariski-closure of $\Gamma$ is perfect. Some of the immediate applications, e.g. orbit equivalence rigidity, {\em largeness} of certain $\ell$-adic Galois representations, are also discussed.  
\end{abstract}
\maketitle
\tableofcontents
\section{Introduction}
\subsection{Statement of the main results.}\label{ss:Statement}
Let $\Omega$ be a finite symmetric subset of $\GL_n(\bbq)$, and $\Gamma=\langle \Omega \rangle$. Since $\Omega$ is finite, for some $q_0\in \bbz^+$ we have $\Gamma\subseteq \GL_n(\bbz[1/q_0])$. {\em Strong approximation} implies that (under certain algebraic conditions) the closure of $\Gamma$ in $\prod_{p\nmid q_0}\GL_n(\bbz_p)$ is open in $\prod_{p\nmid q_0} (\bbg(\bbq_p)\cap \GL_n(\bbz_p))$, where $\bbg$ is the Zariski-closure of $\Gamma$ in $\mathbb{GL}_n$. In combinatorial language, this means that the Cayley graphs 
\[
{\rm Cay}\left(\pi_m(\bbg(\bbq)\cap \GL_n(\bbz[1/q_0])),\pi_m(\Omega)\right),
\]
 where $m$ and $q_0$ are coprime and $\pi_m$ is the group homomorphism induced by the quotient map $\pi_m:\bbz[1/q_0]\rightarrow \bbz[1/q_0]/m\bbz[1/q_0]$, have at most $C:=C(\Omega)$ many connected components. The point being that $C$ does not depend on $m$. {\em Super-approximation}\footnote{Following A. Kontorovich's suggestion, I call this phenomenon super-approximation. It is worth pointing out that this phenomenon has been called superstrong approximation~\cite{MSRI} by some authors.} tells us that these sparse graphs are {\em highly connected}, i.e. they form a family of {\em expanders}.  A family $\{X_i\}_{i\in I}$ of $d_0$-regular graphs is called {\em expander} if there is $c>0$ such that for any $i\in I$ and any subset $A$ of the set of vertices $V(X_i)$ of $X_i$ we have 
 \[
 c<\frac{|e(A,V(X_i)\setminus A)|}{\min(|A|,|V(X_i)\setminus A|)},
 \]
 where $e(A,V(X_i)\setminus A)$ consists of the edges that connect a vertex in $A$ to a vertex in $V(X_i)\setminus A$.
 
 Expanders are extremely useful in communication and theoretical computer science (e.g. see~\cite{HLW}). In the past decade they have been found useful in a wide range of pure math problems, e.g. affine sieve~\cite{BGS, SGS}, sieve in groups~\cite{LM}, variation of Galois representations~\cite{EHK}, etc. (see~\cite{MSRI} for a collection of surveys of related works and applications). 
 
 In this article we prove the {\em best possible} super-approximation results for two families of residue maps. Let $V_f(\bbq)$ be the set of primes in $\bbq$, and $\nu_p(q)$ be the $p$-valuation of $q$, i.e. the power of $p$ in $q$. 
 \begin{thm}\label{t:main}
 Let $\Omega$ be a finite symmetric subset of $\GL_{N_0}(\bbz[1/q_0])$, $\Gamma:=\langle \Omega \rangle$, and $M_0$ be a fixed positive integer. Suppose $\Gamma$ is an infinite group. Then the family of Cayley graphs $\{{\rm Cay}(\pi_m(\Gamma),\pi_m(\Omega))\}_m$ is a family of expanders as $m$ runs through either 
 \[
\{p^n|\h n\in \bbz^+, p\in V_f(\bbq), p\nmid q_0\}\h\h {\rm or}\h\h\{q|\h \forall p\in V_f(\bbq), \nu_p(q)\le M_0\}
\]
if and only if the connected component $\bbg^{\circ}$ of the Zariski-closure $\bbg$ of $\Gamma$ in $\mathbb{GL}_{N_0}$ is perfect, i.e. $[\bbg^{\circ},\bbg^{\circ}]=\bbg^{\circ}$. 
 \end{thm}

In the appendix, a quantitative open image for $p$-adic analytic maps is proved, which should be of independent interest. In a joint work with Zhang~\cite{SGZ}, this result is used to generalize a theorem of Burger and Sarnak~\cite{BuSa} to the setting of {\em super-approximation} (see Theorem~\ref{t:InducingSA} for the statement of this result.).   
 
 
\subsection{Comparing with the previous related results.} The importance of Theorem~\ref{t:main} lies on the fact that it is an {\em if-and-only-if } statement. In fact, by Proposition~\ref{p:Necessary}, we get that, if $\{\pi_{m_i}(\Gamma),\pi_{m_i}(\Omega)\}$ is a family of expanders for some increasing sequence $\{m_i\}$ of integers that are relatively prime to $q_0$, then $\{\pi_m(\Gamma),\pi_m(\Omega)\}_m$ is a family of expanders as $m$ runs through integers specified in Theorem~\ref{t:main}. 

Super-approximation for square-free numbers is the main theorem of~\cite{SGV}. Of course it is a special case of Theorem~\ref{t:main} for $m\in \{q|\h \forall p\in V_f(\bbq), \nu_p(q)\le M_0\}$ and $M_0=1$. It should, however, be noted that our proof relies on \cite{SGV}. Proof of this part of Theorem~\ref{t:main} is heavily influenced by the relevant part of \cite{BV}. But here, instead of working with the concrete structure of the congruence quotients of $\SL_n(\bbz)$ as it was done in \cite{BV}, one has to work with arbitrary perfect groups. So it is inevitable to make use of basic theorems from Bruhat-Tits theory to describe structure of a maximal compact subgroup of $\bbh(\bbq_p)$ where $\bbh$ is a semisimple $\bbq_p$-group, and to use {\em truncated} or {\em finite logarithmic maps}~\cite[Section 6]{Pink} or \cite[Section 2.9]{SGsemisimple}. In addition, one has to understand how such a maximal compact group acts on an open compact subgroup of $\bbu(\bbq_p)$ where $\bbu$ is a unipotent $\bbq_p$-group, which makes the use of the language of schemes necessary.  

Proof of Theorem~\ref{t:main} for powers of primes relies on \cite{SGsemisimple} where the semisimple case is proved. 

Prior to \cite{SGsemisimple}, the case of {\em Zariski-dense} subgroups of $\SL_n(\bbz)$ was studied by Bourgain-Gamburd~\cite{BG2,BG3} and Bourgain-Varj\'{u}~\cite{BV}. As it is pointed out in \cite[Section 1.3]{SGsemisimple}, the proofs by Bourgain and Gamburd~\cite{BG2,BG3} and Bourgain and Varj\'{u}~\cite{BV} rely heavily on Archimedean dynamics. In particular, those ideas cannot be implemented for a finitely generated subgroup $\Gamma$ of $\GL_n(\bbq)$ if $\Gamma$ is a bounded subgroup of $\GL_n(\bbr)$.
But it is worth pointing out that, if $\Gamma$ is a Zariski-dense subgroup of $\SL_n(\bbz)$, Bourgain and Varj\'{u}~\cite{BV} prove that $\{{\rm Cay}(\pi_m(\Gamma),\pi_m(\Omega))\}_m$ is a family of expanders with {\em no restriction} on $m$. 

Prior to the current work, {\em no} super-approximation result for powers of primes {\em beyond semisimple case} was known. To get such a result, one had to use new ideas and combine techniques from $p$-adic analytic analysis, $p$-adic analytic pro-$p$ groups, and non-commutative algebra.


\subsection{Applications}
\subsubsection{Random-walk and spectral gap.}
Let $\Omega$ be a finite symmetric subset of a compact group $G$. Let $\mu$ be the counting probability measure on $\Omega$ and $\Gamma:=\langle \Omega \rangle$. Let  
\[
T_{\mu}:L^2(\overline{\Gamma})\rightarrow L^2(\overline{\Gamma}),\h\h T_{\mu}(f):=\mu \ast f,
\]
where $(\mu \ast f)(g):=\sum_{g'\in \supp \mu}\mu(g')f(g'^{-1}g)$ and $\overline{\Gamma}$ is the closure of $\Gamma$ in $G$. It is easy to see that $T_{\mu}$ is a self-adjoint operator, the constant function $\one_{\overline{\Gamma}}$ is an eigenfunction with eigenvalue one, i.e. $T_{\mu}(\one_{\overline{\Gamma}})=\one_{\overline{\Gamma}}$, and $\|T_{\mu}\|=1$. Let
\[
\lambda(\mu;G):=\sup \{|\lambda| |\h\h \lambda \in {\rm spec}(T_{\mu}),\h \lambda<1\}.
\]
We get a fairly good understanding of the random-walk in $G$ with respect to $\mu$ if $\lambda(\mu;G)<1$, in which case it is said that either $\Gamma\acts \overline{\Gamma}$ or the random-walk with respect to $\mu$ has {\em spectral gap}. 

Let us recall a couple of well-known results which give us a connection between having spectral gap and explicit construction of expanders (see \cite[Chapter 4.3]{Lub}, \cite[Chapter 1.4]{LZ}, \cite[Remark 15]{SGsemisimple}).
\begin{remark}\label{r:ExpanderSpectralGap}
\begin{enumerate}
\item Let $\Gamma$ be the group generated by a finite symmetric set $\Omega$. Suppose $\{N_i\}_{i\in I}$ is a family of normal, finite-index subgroups of $\Gamma$. Suppose, for any $i_1,i_2\in I$, there is $i_3\in I$ such that $N_{i_3} \subseteq N_{i_1}\cap N_{i_2}$. Then, by Peter-Weyl theorem, one has 
		\[
		\lambda(\pcal_{\Omega}; \varprojlim \Gamma/N_i)=\sup_i \lambda(\pcal_{\iota_{N_i}(\Omega)};\Gamma/N_i),
		\]     	
	where $\iota_{N_i}:\Gamma\rightarrow \Gamma/N_i$ is the natural quotient map. 
\item Suppose $\{N_i\}_{i\in I}$ is a family of normal, finite-index subgroups of $\Gamma$. Then the family of Cayley graphs $\{\Cay(\Gamma/N_i,\iota_{N_i}(\Omega))\}_{i\in I}$ is a family of expanders if and only if \[\sup_i \lambda(\pcal_{\iota_{N_i}(\Omega)};\Gamma/N_i)<1.\]  
\end{enumerate}
	
\end{remark}
Based on these results, we get that Theorem~\ref{t:main} implies the following. 
\begin{cor}\label{c:SpectralGap}
Let $\Omega$ be a finite symmetric subset of $\GL_{N_0}(\bbz[1/q_0])$, and $\pcal_{\Omega}$ be the counting probability measure on $\Omega$. Suppose the connected component $\bbg^{\circ}$ of the Zariski-closure $\bbg$ of $\Gamma:=\langle \Omega \rangle$ in $\mathbb{GL}_{N_0}$ is perfect. Then for any positive integer $M_0$ we have
\[
\lambda(\pcal_{\Omega}; \prod_{p\nmid q_0}\GL_{N_0}(\bbz/p^{M_0}\bbz))<1,
\]
and
\[
\sup_{p\nmid q_0} \lambda(\pcal_{\Omega}; \GL_{N_0}(\bbz_p))<1.
\]
\end{cor} 

Since a stronger form of the {\em only-if} part of Theorem~\ref{t:main} (see Proposition~\ref{p:Necessary}) holds, having spectral gap for a single place implies a uniform spectral gap for all the places.
\begin{cor}\label{c:OneForAll}
	Let $\Omega$ be a finite symmetric subset of $\GL_{N_0}(\bbz[1/q_0])$, and $\pcal_{\Omega}$ be the counting probability measure on $\Omega$. Then for some $p\nmid q_0$ we have $\lambda(\pcal_{\Omega}; \GL_{N_0}(\bbz_p))<1$ if and only if 
	\[
\sup_{p\nmid q_0} \lambda(\pcal_{\Omega}; \GL_{N_0}(\bbz_p))<1.
\]
\end{cor}

Spectral gap has a well-known {\em weighted equidistribution} consequence. 
\begin{cor}\label{c:WeightedEquidistribution}
In the setting of Corollary~\ref{c:SpectralGap}, there is $\lambda<1$  such that for any prime $p$ which does not divide $q_0$ and $\bbg(\bbq_p)\cap \GL_{N_0}(\bbz_p)$-finite function $f_0$ on $\bbg(\bbq_p)$ (which means that $f_0$ is invariant under an open subgroup of $\bbg(\bbq_p)$) we have
\[
|\sum_{\gamma\in \Gamma} \pcal_{\Omega}^{(l)}(\gamma) f_0(\gamma)-\int_{\Gamma_p} f_0(g) dg|\le \| f_0-\int_{\Gamma_p} f_0(g) dg \|_2\h \sqrt{|\Gamma_p\cdot f_0|} \lambda^l,
\]
where $\Gamma_p$ is the closure of $\Gamma$ in $\GL_{N_0}(\bbz_p)$ and it acts on functions on $\bbg(\bbq_p)$ via the right translation, i.e. $(g\cdot f)(g'):=f(g'g)$, $\pcal_{\Omega}^{(l)}$ is the $l$-th convolution power of $\pcal_{\Omega}$, $dg$ is the probability Haar measure on $\Gamma_p$, and $\|f\|_2:=(\int_{g\in \Gamma_p} f(g)^2 dg)^{1/2}$.
\end{cor}
\begin{proof}
Since $\Omega$ is a symmetric set, we have that 
\be\label{e:PointwiseValue}
\sum_{\gamma\in\Gamma} \pcal_{\Omega}^{(l)}(\gamma)f_0(\gamma)-\int_{\Gamma_p} f_0(g) dg=T_{\mu}^l(f_0)(I)-\langle \one_{\Gamma_p}, f_0\rangle,
\ee
where $I$ is the identity matrix and $\langle,\rangle$ is the dot product in $L^2(\Gamma_p)$. 

On the other hand, by Corollary~\ref{c:SpectralGap}, we have that
\be\label{e:L2Bound}
\|T_{\mu}^l(f_0)-\langle \one_{\Gamma_p},f_0\rangle\|_2 = \|T_{\mu}^l(f_0-\langle \one_{\Gamma_p},f_0\rangle \one_{\Gamma_p})\|_2\le \lambda^l \|f_0-\langle \one_{\Gamma_p},f_0\rangle \one_{\Gamma_p})\|_2,
\ee	
where $\lambda:=\sup_{p\nmid q_0}\lambda(\pcal_{\Omega};\GL_{N_0}(\bbz_p))$. For any function $f$, let $\Gamma_{p,f}:=\{g\in \Gamma_p|\h g\cdot f=f\}$. Since for any $g\in \Gamma_p$ and $f\in L^2(\Gamma_p)$ we have $T_{\mu}(g\cdot f)=g\cdot T_{\mu}(f)$, we get that $\Gamma_{p,f}\subseteq \Gamma_{p,T_{\mu}(f)}$. So we have
\begin{align}
\notag \|T_{\mu}^l(f_0)-\langle \one_{\Gamma_p},f_0\rangle\|_2^2 & 
= \int_{\Gamma_p/\Gamma_{p,f_0}} \int_{\Gamma_{p,f_0}} |T_{\mu}^l(f_0)(gg')-\langle \one_{\Gamma_p},f_0\rangle|^2 dg' dg\\
\notag&= \int_{\Gamma_p/\Gamma_{p,f_0}} \int_{\Gamma_{p,f_0}} |(g'\cdot T_{\mu}^l(f_0))(g)-\langle \one_{\Gamma_p},f_0\rangle|^2 dg' dg\\
\notag &= \int_{\Gamma_p/\Gamma_{p,f_0}} |T_{\mu}^l(f_0)(g)-\langle \one_{\Gamma_p},f_0\rangle|^2 \int_{\Gamma_{p,f_0}}  dg' dg\\
\notag &= \frac{1}{[\Gamma_p:\Gamma_{p,f_0}]} \sum_{g\Gamma_{p,f_0}\in \Gamma_p/\Gamma_{p,f_0}} |T_{\mu}^l(f_0)(g)-\langle \one_{\Gamma_p},f_0\rangle|^2\\
\label{e:L2Norm} &\ge \frac{|T_{\mu}^l(f_0)(I)-\langle \one_{\Gamma_p},f_0\rangle|^2}{|\Gamma_p\cdot f_0|}.
\end{align}
Equations (\ref{e:PointwiseValue}), (\ref{e:L2Bound}), and (\ref{e:L2Norm}) imply the claim.
\end{proof}
Corollary~\ref{c:WeightedEquidistribution} is another indication that {\em super-approximation} is a suitable name for such a phenomenon as it implies a quantitative way to approximate points. 


\subsubsection{Orbit equivalence.}  Suppose $\Gamma\subseteq G$ and $\Lambda\subseteq H$ are dense subgroups of compact groups $G$ and $H$. We say the left translation actions $\Gamma\acts (G,m_G)$ and $\Lambda\acts (H,m_H)$ are {\em orbit equivalent} if there exists a measure class preserving Borel isomorphism $\theta:G\rightarrow H$ such that $\theta(\Gamma g)=\Lambda \theta(g)$, for $m_G$-almost every $g\in G$. Surprisingly, if $\Gamma$ and $\Lambda$ are amenable, the mentioned actions are orbit equivalent~\cite{OW,CFW}. In the past decade there have been a lot of progress on this subject, and as a result now it is known that one gets orbit equivalence rigidity under spectral gap assumption~\cite{Ioa1,Ioa2}.
\begin{cor}\label{c:OrbitEquivalence}
Let $\Omega$ be a finite symmetric subset of $\GL_{N_0}(\bbz[1/q_0])$, and $\pcal_{\Omega}$ be the counting probability measure on $\Omega$. Suppose the connected component $\bbg^{\circ}$ of the Zariski-closure $\bbg$ of $\Gamma:=\langle \Omega \rangle$ in $\mathbb{GL}_{N_0}$ is perfect. Let $\Gamma_p$ be the closure of $\Gamma$ in $\GL_{N_0}(\bbz_p)$ where $p\nmid q_0$.

Let $\Lambda$ be a countable dense subgroup of a profinite group $H$. Then $\Gamma\acts (\Gamma_p,m_{\Gamma_p})$ and $\Lambda\acts (H,m_H)$ are orbit equivalent if and only if there are open subgroups $G_0\subseteq \Gamma_p$ and $H_0\subseteq H$ and a continuous isomorphism $\delta:G_0\rightarrow H_0$ such that $[\Gamma_p:G_0]=[H:H_0]$ and $\delta(G_0\cap \Gamma)=\Lambda \cap H_0$. In particular, $H$ is $p$-adic analytic.  
\end{cor}
\begin{proof}
This is a direct consequence of Corollary~\ref{c:SpectralGap} and \cite[Theorem A]{Ioa1}.
\end{proof} 


\subsubsection{Variations of Galois representations in one-parameter families of abelian varieties.} Since eight years ago, because of a surge of works by various people specially Cadoret, Tamagawa~\cite{Cad,CT1,CT2}, Hui, and Larsen~\cite{Hui,HL}, we have got a much better understanding of the image of $\ell$-adic and adelic Galois representations induced by Tate modules of an abelian scheme. On the other hand, Ellenberg, Hall, and Kowalski in~\cite{EHK} gave a beautiful connection between variations of Galois representations and certain spectral gap property\footnote{Instead of being {\em expanders}, they only need a weaker assumption on the relevant Schreier graphs, called {\em esperantist graphs}.}. Here we make an observation that Ellenberg-Hall-Kowalski~\cite{EHK} machinery combined with Theorem~\ref{t:main} gives an alternative approach towards \cite{CT1,CT2}; in particular, we do not get any new result on this topic. 

In this section, let $k$ be a finitely generated characteristic zero field, $\overline{k}$ be the algebraic closure of $k$, and $U$ be a smooth algebraic curve over $k$ such that $U\times_k \overline{k}$ is connected. Let $\acal \rightarrow U$ be an abelian scheme of dimension $g\ge 1$, defined over $k$. Take an embedding of $k$ into $\bbc$, and let 
\[
\rho_0:\pi_1(U(\bbc),y_0)\rightarrow {\rm Aut}_U(\acal)\subseteq GL_{2g}(\bbz)
\]
be the monodromy representation. Let $\Gamma:=\rho_0(\pi_1(U(\bbc),y_0))$, and $\bbg$ be the Zariski-closure of $\Gamma$ in $(\GL_{2g})_{\bbq}$. Let $\Omega$ be a finite symmetric generating set of $\Gamma$.

For any $x\in U(\overline{k})$, the fiber $\acal_x$ over $x$ is an abelian variety defined over the residue field $k(x)$ at $x$. For any prime $l$, let $T_{l,x}$ be the Tate module of $\acal_x$, i.e. 
\[
T_{l,x}:=\varprojlim \acal_x[l^m](\overline{k}),
\]
where $\acal_x[l^m](\overline{k})$ is the $l^m$-th torsion elements of $\acal_x(\overline{k})$. So $T_{l,x}\simeq \bbz_l^{2g}$ and we get the $l$-adic Galois representation
\[
\rho_{l,x}:{\rm Gal}(\overline{k}/k(x))\rightarrow {\rm Aut}_{\bbz_l}(T_{l,x})\simeq \GL_{2g}(\bbz_l).
\]

\begin{lem}\label{l:Perfect}
In the above setting, the connected component $\bbg^{\circ}$ of $\bbg
$ is perfect.
\end{lem}
\begin{proof}
By \cite[Proposition 16]{EHK}, $\Gamma$ has a finite index subgroup $\Lambda$ such that $\pi_l(\Lambda)$ is a perfect group and generated by its order $l$ elements for large enough prime $l$. In particular, the index of any proper subgroup of $\pi_l(\Lambda)$ is at least $l$.

Let $\Gamma^{\circ}:=\Gamma\cap \bbg^{\circ}$. So $\Gamma^{\circ}$ is Zariski-dense in $\bbg^{\circ}$. Since $[\pi_l(\Lambda):\pi_l(\Lambda\cap \Gamma^{\circ})]\le [\Gamma:\Gamma^{\circ}]$, for large enough prime $l$ we have that $\pi_l(\Lambda)=\pi_l(\Lambda\cap \Gamma^{\circ})$. 

On the other hand, the quotient map $\bbg^{\circ} \xrightarrow{f} (\bbg^{\circ})^{\rm ab}:= \bbg^{\circ}/[\bbg^{\circ},\bbg^{\circ}]$ is defined over $\bbq$. Hence after realizing $(\bbg^{\circ})^{\rm ab}$ as a subgroup of $(\mathbb{GL}_{n'})_{\bbq}$, for large enough $l$ we have that $\pi_l(f(\Gamma^{\circ}))$ is a homomorphic image of $\pi_l(\Gamma^{\circ})^{{\rm ab}}:=\pi_l(\Gamma^{\circ})/[\pi_l(\Gamma^{\circ}),\pi_l(\Gamma^{\circ})]$. Suppose to the contrary that $\bbg^{\circ}$ is not perfect. Then $f(\Gamma^{\circ})$ is a finitely generated, infinite, abelian group as it is a Zariski-dense subgroup of $(\bbg^{\circ})^{\rm ab}$. Hence $|\pi_l(f(\Gamma^{\circ}))| \rightarrow \infty$ as $l$ goes to infinity, which implies that $|\pi_l(f(\Gamma^{\circ}\cap \Lambda))|\rightarrow \infty$. Since $\pi_l(f(\Gamma^{\circ}\cap \Lambda))$ is abelian and a homomorphic image of $\pi_l(\Gamma^{\circ}\cap \Lambda)$, we get a contradiction.
\end{proof}
Now we can give an alternative proof of the main result of \cite{CT1,CT2}.

\begin{prop}\label{p:l-adicGaloisRepresentation}
  In the above setting, for any integer $d$ and any prime $\ell$ there is an integer $B:=B(d,\ell,\rho_0)$ such that  
  \[
  U_{d,\ell}:=\{x\in U(\overline{k})|\h [k(x):k]\le d, [\Gamma_{\ell}:{\rm Im}(\rho_{\ell,x})\cap \Gamma_{\ell}]>B\},
  \]
  is finite.
\end{prop}
\begin{proof}
Suppose to the contrary that there is a sequence $\{x_i\}$ of points in $U(\overline{k})$ such that 
\begin{enumerate}
\item $x_i\neq x_j$ if $i\neq j$,
\item $[k(x_i):k]\le d$,
\item $[\Gamma_{\ell}:{\rm Im}(\rho_{\ell,x_i})\cap \Gamma_{\ell}]=N_i$ where $N_i\in \bbz^+\cup\{\infty\}$, and $N_i\rightarrow \infty$.
\end{enumerate}
By Lemma~\ref{l:Frattini}, there is an increasing sequence $\{m_i\}_i$ of positive integers and a sequence of open subgroups $\{H_i\}_i$ of $\Gamma_{\ell}$ such that 
\[
\lim_{i\rightarrow\infty} [\Gamma_{\ell}:H_i]=\infty,\h 
\Gamma_{\ell}\cap {\rm Im}(\rho_{\ell,x_i}) \subseteq H_i,
{\rm and}\h\h \Gamma_{\ell}[\ell^{m_i}]\subseteq H_i.
\]
Going to a subsequence, if necessary, we can and will assume that $\pi_{\ell^{m_i}}(H_i)=\pi_{\ell^{m_i}}(H_j)$ for any $j\ge i$. Therefore we can and will assume that 
$
H_1\supsetneq H_2\supsetneq \cdots.
$

For each $i$, we get an open subgroup of $\pi_1^{et}(U\times_k \overline{k})$, and so we get an \'{e}tale cover $U_i\xrightarrow{\phi_i} U$ of $U$ that is defined over $k$. Since ${\rm Im}(\rho_{\ell,x_i})\subseteq H_i$, there is $\wt{x}_i\in U_i(k(x_i))$ which is  in the fiber over $x_i\in U(k(x_i))$. Moreover, since $\{H_i\}_i$ is a decreasing sequence of open subgroups, we get $k$-\'{e}tale covering maps $U_j\xrightarrow{\phi_{ij}} U_i$ for any $i\le j$ that satisfy $\phi_i\circ \phi_{ij}=\phi_j$. In particular, we have that 
\be\label{e:InfinitelyManyRationalPoints}
\{\phi_{ij}(\wt{x}_j)\}_j \subseteq \bigcup_{[k':k]\le d} U_i(k')
\ee
is an infinite subset for any $i$. On the other hand, since $\pi_1^{et}(U_i\times_k \overline{k})$ is the profinite closure of $\pi_1(U_i(\bbc),x_0)$, the natural embedding induces an isomorphism between the Schreier graphs
\be\label{e:Sch}
{\rm Sch}(\pi_1(U(\bbc),x_0)/\pi_1(U_i(\bbc),\wt{x}_{0,i}); \Omega) \simeq {\rm Sch}(\pi_{\ell^{m_i}}(\Gamma_{\ell})/\pi_{\ell^{m_i}}(H_i); \pi_{\ell^{m_i}}(\Omega)),
\ee
where $\wt{x}_{0,i}$ is a point over $x_0$. On the other hand, by Lemma~\ref{l:Perfect} and Theorem~\ref{t:main}, we have that 
\be\label{e:Cay}
\{{\rm Cay}(\pi_{\ell^{m_i}}(\Gamma); \pi_{\ell^{m_i}}(\Omega)\}_i
\ee
is a family of expanders. Since the Schreier graphs in (\ref{e:Sch}) are quotients of the Cayley graphs in (\ref{e:Cay}) and their size goes to infinity, they form a family of expanders. Therefore by the main result of Ellenberg-Hall-Kowalski~\cite[Theorem 8]{EHK} we have that the (geometric) gonality $\gamma(U_i)$ of $U_i$ goes to infinity. Hence by a corollary~\cite[Theorem 2.1]{CT2} of Falting's theorem~\cite{Fal} on Lang-Mordell conjecture (see \cite{Fer,Mc,Maz}) we have that 
\[
\bigcup_{[k':k]\le d} U_i(k')
\]
is finite for large enough $i$. This contradicts (\ref{e:InfinitelyManyRationalPoints}).
\end{proof}

\subsubsection{Inducing super-approximation} As it was mentioned earlier, in the appendix a quantitative open function theorem (see Lemma~\ref{l:OpenFunction'}) is proved, which is of independent interest. This result plays an important role in my joint work with Zhang~\cite{SGZ}, where we prove that one can induce super-approximation from a subgroup under a mild (and needed) algebraic condition. 
\begin{thm}\label{t:InducingSA}
Let $\Omega_1$ be a finite symmetric subset of $\GL_{N_0}(\bbz[1/q_0])$, and $\Gamma_1:=\langle \Omega_1\rangle$. Let $\Omega_2$ be a finite symmetric subset of $\Gamma_1$, and $\Gamma_2:=\langle\Omega_2\rangle$. Let $\bbg_i^{\circ}$ be the connected component of the Zariski-closure of $\Gamma_i$ in $\mathbb{GL}_{N_0}$ for $i=1,2$. Suppose the smallest normal subgroup of $\bbg_1^{\circ}$ which contains $\bbg_2^{\circ}$ is $\bbg_1^{\circ}$. Then, if $\Gamma_2$ has super-approximation, then $\Gamma_1$ has, too; that is equivalent to say   
\[
\lambda(\pcal_{\Omega_2};\prod_{p\nmid q_0}\bbz_p)<1 \Longrightarrow \lambda(\pcal_{\Omega_1};\prod_{p\nmid q_0}\bbz_p)<1.
\]
\end{thm}
Such a result for arithmetic groups was proved by Burger and Sarnak~\cite{BuSa}.

\section*{Acknowledgments.} I would like to thank Adrian Ioana for pointing out the application of spectral gap to orbit rigidity. I am thankful to Cristian Popescu for  useful conversations concerning $\ell$-adic Galois representations. Special thanks are due to Efim Zelmanov for his support and encouragement. My sincere appreciation goes to the referees for all the corrections and suggestions, which made this article much easier to read, and for helping me to spot some inconsistencies in the first version of this work.
\section{Preliminary results and notation.}\label{s:prelim}
\subsection{Notation.} In this note, for two real valued functions $f$ and $g$, and a set of parameters $X$, we write $f=O_X(g)$ if  there are positive functions $C_1(X)$ and $C_2(X)$ of the set of parameters such that for $t\ge C_1(X)$ we have $0\le f\le C_2(X) g$; notice that this is slightly different from the usual Landau symbol as $O_X(g)$ is assumed to be {\em non-negative}. For two real valued functions $f$ and $g$ of a real variable $t$ and a set of parameters $X$, we write $f=\Theta_X(g)$ if there are positive functions $C_1(X), C_2(X)$, and $C_3(X)$ such that $C_1(X) g(t)\le f(t)\le C_2(X) g(t)$ for any $t\ge C_3(X)$.  For two real valued functions $f$ and $g$ of a real variable $t$ and a set of parameters $X$, we write $f\ll_X g$ if there are positive functions $C_1(X)$ and $C_2(X)$ such that for $t\ge C_2(X)$ we have $f(t)\le C_1(X) g(t)$. So for two non-negative functions $f$ and $g$ we have $f=O_X(g)$ if and only if $f\ll_X g$; and $f=\Theta_X(g)$ if and only if $g\ll_Xf\ll_X g$.

For a (commutative unital) ring $R$ and $a\in R$, $\pi_a:R\rightarrow R/aR$ is the quotient map. We denote also the induced group homomorphism $\GL_n(R)\rightarrow \GL_n(R/aR)$ and all of its restrictions by $\pi_a$. For a subgroup $\Gamma$ of $\GL_n(R)$, $\Gamma[a]$ denotes $\{\gamma\in \Gamma|\h \pi_a(\gamma)=1\}$; in particular $\Gamma[1]=\Gamma$.

For a subset $A$ of a group $G$, we write either $\prod_CA$ or $\underbrace{A\cdot A \cdot \cdots \cdot A}_{C \text{ times}}$ for the subset $\{a_1\cdots a_c|\h a_1,\ldots,a_C\in A\}$; $A^{-1}$ denotes the subset $\{a^{-1}|\h a\in A\}$; and $A$ is called symmetric if $A=A^{-1}$. For an additive group $G$, $\sum_C A$ denotes the subset $\{a_1+\cdots+a_C|\h a_1,\ldots,a_C\in A\}$.  

For a measure $\mu$ on a group $G$, we let $\supp{\mu}$ be the support of $\mu$. For a measure $\mu$ on a group $G$ and $g\in G$, we let $\mu(g):=\mu(\{g\})$. For a measure $\mu$ with finite support on a group $G$ and a complex valued function $f$ on $G$, we let $\mu\ast f$ be the convolution of $\mu$ and $f$, that means $(\mu\ast f)(g)=\sum_{s\in \supp(\mu)}\mu(s)f(s^{-1}g)$. For a measure $\mu$ with finite support on a group $G$, we let $\mu^{(l)}$ be the $l$-fold convolution of $\mu$, that means for any $g\in G$, we have $\mu^{(l)}(g)=(\underbrace{\mu\ast\cdots\ast\mu}_{l \text{ times}})(g)=\sum_{g_1\cdots g_l=g}\mu(g_1)\cdots\mu(g_l)$. We say a measure $\mu$ with finite support on a group $G$ is symmetric if $\mu(g^{-1})=\mu(g)$ for any $g\in G$. For a finite subset $X$ of a group $G$, we let $\pcal_X$ be the probability counting measure on $X$.

The set of (rational) primes is denoted by $V_f(\bbq)$. For any prime $p$, $\f_p$ denotes the finite field with $p$ elements; $\bbq_p$ is the field of $p$-adic numbers; $\bbz_p$ is the ring of $p$-adic integers; and $v_p:\bbq_p\setminus\{0\}\rightarrow \bbz$ (and its restriction to $\bbq\setminus\{0\}$) is the $p$-adic valuation, that means $v_p(x)=n$ if and only if $x\in p^n\bbz_p\setminus p^{n+1}\bbz_p$. For any prime $p$ and $x\in \bbq_p$, we let $|x|_p:=(1/p)^{v_p(x)}$ if $x\neq 0$, and $|0|_p:=0$.

\subsection{Necessity.}\label{ss:Necessity} In this section, we will prove that getting a family of expanders modulo an infinite sequence of integers implies that the connected component of the Zariski-closure is perfect (see Proposition~\ref{p:Necessary}). In particular, we get the necessary part of Theorem~\ref{t:main}. 

Let us remark that proof of~\cite[Section 5.1]{SGV} can be adjusted to give the necessary part of Theorem~\ref{t:main} for any fixed prime $p$. The main point being that the proof in \cite{SGV} makes use of the fact that the congruence kernels modulo square-free numbers define a topology. 

The proof here is rather straightforward and has three parts: 
\begin{enumerate}
	\item[(a)] Changing $\Gamma$ to $\Gamma^{\circ}:=\Gamma\cap \bbg^{\circ}$, we can and will assume $\bbg$ is Zariski-connected; 
	\item[(b)] If $\bbg$ has infinite abelianization $\bbg^{\rm ab}:=\bbg/[\bbg,\bbg]$, then the order of the abelianization $\pi_{q_i}(\Gamma)^{\rm ab}$ of $\pi_{q_i}(\Gamma)$ gets arbitrarily large as $q_i$ goes to infinity. 
	\item[(c)] If $S$ is a finite symmetric generating set of an abelian group  $A$, then the order of $A$ is bounded by a function of $|S|$ and $\lambda(\pcal_{S};A)$ (see \cite[Corollary 3.3]{LubWei}).  
\end{enumerate}
\begin{prop}~\label{p:Necessary}
Let $\Omega$ be a finite symmetric subset of $\GL_n(\bbz[1/q_0])$ and  $\Gamma=\langle \Omega\rangle\subseteq \GL_n(\bbz[1/q_0])$. Let $\bbg^{\circ}$ be the Zariski-connected component of the Zariski-closure $\bbg$ of $\Gamma$ in $(\GL_n)_{\bbq}$. If $\{\Cay(\pi_{q_i}(\Gamma),\pi_{q_i}(\Omega))\}_i$ is a family of expanders for integers $q_1<q_2< \cdots$ coprime with $q_0$, then $\bbg^{\circ}$ is perfect, i.e. $[\bbg^{\circ},\bbg^{\circ}]=\bbg^{\circ}$. 
\end{prop} 
\begin{proof}
Let $\Gamma^{\circ}:=\Gamma\cap \bbg^{\circ}$. Then, by \cite[Corollary 17]{SGsemisimple}, $\Gamma^{\circ}$ has a finite symmetric generating set $\Omega^{\circ}$ and 
\[
\{\Cay(\pi_{q_i}(\Gamma^{\circ}),\pi_{q_i}(\Omega^{\circ}))\}_i
\]
is a family of expanders. Therefore for some positive integer $C$ we have $|\pi_{q_i}(\Gamma^{\circ})/[\pi_{q_i}(\Gamma^{\circ}),\pi_{q_i}(\Gamma^{\circ})]|<C$ for any $i$ (see \cite[Corollary 3.3]{LubWei}). 

On the other hand, $\bbg^{\circ}$ and $[\bbg^{\circ},\bbg^{\circ}]$ are $\bbq$-groups, and the quotient map $\bbg^{\circ}\xrightarrow{f} \bbg^{\circ}/[\bbg^{\circ},\bbg^{\circ}]$ is a $\bbq$-homomorphism. Therefore $\bbg^{\circ}/[\bbg^{\circ},\bbg^{\circ}]$ can be viewed as a $\bbq$-subgroup of $(\GL_{n'})_{\bbq}$ such that first the quotient map $f$ induces a $\bbq$-homomorphism from $\bbg$ to $(\GL_{n'})_{\bbq}$ and second $f(\Gamma^{\circ})\subseteq \GL_{n'}(\bbz[1/q_0])$. So there is an integer $q'$ such that for any prime $p$ we have $\|f(\gamma)-I\|_p\le |q'|_p^{-1} \|\gamma-I\|_p$. Therefore $f$ induces an epimorphism from $\pi_{q_i}(\Gamma^{\circ})/[\pi_{q_i}(\Gamma^{\circ}),\pi_{q_i}(\Gamma^{\circ})]$ onto $\pi_{q_i/\gcd(q_i,q')}(f(\Gamma^{\circ}))$. Hence for any $i$ we have
\[
|\pi_{q_i/\gcd(q_i,q')}(f(\Gamma^{\circ}))|<C.
\]
Since $f(\Gamma^{\circ})$ is a Zariski-dense subset of $\bbg^{\circ}/[\bbg^{\circ},\bbg^{\circ}]$ and modulo arbitrarily large integers it has at most $C$ elements, we have that $\bbg^{\circ}/[\bbg^{\circ},\bbg^{\circ}]$ is finite. So by the connectedness we have that $\bbg^{\circ}$ is perfect.
\end{proof}
\subsection{A few reductions.}~\label{ss:Reductions}In this section, we make a few reductions and describe the group structure of $\pi_Q(\Gamma)$ using strong approximation~\cite{Nor}.
 
\begin{lem}\label{l:ReductionToConnectedSimplyConnectedCase}
It is enough to prove the sufficiency  part of Theorem~\ref{t:main} under the following additional assumptions on the Zariski-closure $\bbg$ of $\Gamma$ in $(\GL_{n_0})_{\bbq}$:

There are connected, simply-connected, semisimple $\bbq$-group $\bbg_s$, and unipotent $\bbq$-group $\bbu$ such that $\bbg_s$ acts on $\bbu$ and $\bbg\simeq \bbg_s\ltimes \bbu$.  
\end{lem}
\begin{proof}
A similar argument as in the proof of \cite[Lemma 18]{SGsemisimple} works here. For the convenience of the reader, we present an outline of the proof.
 
By the assumption, $\bbg^{\circ}$ is a perfect group. Therefore the radical $\bbu$ of $\bbg^{\circ}$ is unipotent. Let $\bbg_s$ be a Levi subgroup of $\bbg^{\circ}$, $\wt{\bbg}_s$ be the simply-connected cover of $\bbg_s$, and $\iota:\wt{\bbg}_s\ltimes \bbu\rightarrow \bbg_s\ltimes \bbu=\bbg^{\circ}$ be the induced covering map. To lift the problem to $\wt{\bbg}$, we consider
\[
\wt{\Lambda}:=\iota^{-1}(\Gamma\cap \bbg^{\circ}(\bbq)) \cap \wt{\bbg}(\bbq),
\]
and $\Lambda:=\iota(\wt{\Lambda})$.  
Since $\bbg^{\circ}(\bbq)/\iota(\wt{\bbg}(\bbq))$ is a torsion abelian group and $\Gamma$ is a finitely generated group, we have that $\Lambda$ is a normal finite index subgroup of $\Gamma$. Therefore by \cite[Corollary 17]{SGsemisimple} we have that $\Lambda$ has a finite symmetric generating set $\Omega'$ and for any family of positive integers $\ccal$ we have that $\{\Cay(\pi_q(\Gamma),\pi_q(\Omega))\}_{q\in \ccal}$ is a family of expanders if and only if $\{\Cay(\pi_q(\Lambda),\pi_q(\Omega'))\}_{q\in \ccal}$ is a family of expanders. 

Let us fix an embedding $f:\wt{\bbg}\rightarrow (\GL_{n_0'})_{\bbq}$. Since $\Gamma\subseteq \GL_{n_0}(\bbz[1/q_0])$, after passing to a (normal) finite index subgroup of $\Lambda$, if needed, we can and will assume that $f(\wt{\Lambda})\subseteq \GL_{n_0'}(\bbz[1/q_0])$. Since $\iota$ is an isogeny defined over $\bbq$, there is $q_0'\in \bbz$ such that for any positive integer $q$ and $\gamma\in \wt{\Lambda}$ we have
\[
\pi_q(f(\gamma))=1 \h\h\text{implies}\h\h \pi_{q/\gcd(q_0',q)}(\iota(\gamma))=1.
\]
Hence $\pi_q(\wt{\Lambda})$ maps onto $\pi_{q/\gcd(q,q_0')}(\Lambda)$ via the map induced by $\iota$. Therefore if Theorem~\ref{t:main} is proved for the group $\wt{\Lambda}\subseteq \GL_{n_0'}(\bbz[1/q_0])$, we get the desired result for the group $\Gamma$. Hence we can focus on the group $\wt{\Lambda}$ which is Zariski-dense in $\wt{\bbg}$; and the claim follows.  
\end{proof}    

\begin{lem}\label{l:AdelicClosure}
Let $\gcal$ and $\bbg$ be the Zariski-closure of $\Gamma$ in $(\GL_{n_0})_{\bbz[1/q_0]}$ and $(\GL_{n_0})_{\bbq}$, respectively. Let $\bbg_s$ be a semisimple $\bbq$-subgroup of $\bbg$, and $\bbu$ be a unipotent $\bbq$-subgroup of $\bbg$ such that $\bbg=\bbg_s\ltimes \bbu$. 

It is  enough to prove the sufficiency part of Theorem~\ref{t:main} under the following additional assumptions on the closure $\wh{\Gamma}$ of $\Gamma$ in $\prod_{p\nmid q_0}\GL_{n_0}(\bbz_p)$:
\begin{enumerate}
\item $\wh{\Gamma}=\prod_{p\nmid q_0} P_p$ is an open compact subgroup of $\prod_{p\nmid q_0} \gcal(\bbz_p)$.
\item For some non-negative integer $q_0'$ and any prime $p\nmid q_0$, we have $P_p=\bbg(\bbq_p)\cap \GL_{n_0}(\bbz_p)[p^{k_p}]$, where $k_p:=v_p(q_0')$ and $\GL_{n_0}(\bbz_p)[p^{k_p}]:=\{g\in \GL_{n_0}(\bbz_p)|\h \|g-1\|_p\le p^{-k_p}\}$. And $P_p=K_p\ltimes U_p$ where $K_p:=\bbg_s(\bbq_p)\cap \GL_{n_0}(\bbz_p)[p^{k_p}]$, $U_p:=\bbu(\bbq_p)\cap \GL_{n_0}(\bbz_p)[p^{k_p}]$. 
\item There is a prime $p_0$ such that, for $p\ge p_0$, $K_p=\bbg_s(\bbq_p)\cap \GL_{n_0}(\bbz_p)$ is a hyperspecial parahoric subgroup of $\bbg_s(\bbq_p)$, and $U_p=\bbu(\bbq_p)\cap \GL_{n_0}(\bbz_p)$.
\item For $p\ge p_0$, the dimension of any non-trivial irreducible representation of $\pi_p(\Gamma)\simeq \gcal(\f_p)$ is at least $|\pi_p(\Gamma)|^{O_{\dim \bbg}(1)}$.
\end{enumerate}
\end{lem}
\begin{proof}
By \cite[Corollary 17]{SGsemisimple}, we are allowed to pass to a finite-index subgroup of $\Gamma$ if needed. By Lemma~\ref{l:ReductionToConnectedSimplyConnectedCase} we can and will assume that $\bbg$ is connected, simply-connected, and perfect. Hence by Strong approximation we have that $\wh{\Gamma}$ is an open subgroup of $\prod_{p\nmid q_0} \gcal(\bbz_p)$ (see \cite[Theorem 5.4]{Nor}). So passing to a finite-index subgroup we get part 1. 
By \cite[Section 3.8]{TitsSurvey}, for large enough $p$, $K_p$ is a hyperspecial parahoric. (See the same reference for the definition of hyperspecial; here we use the property that $\pi_p(K_p)$ is a product of quasi-simple groups where the number of factors and their ranks are bounded by $\dim \bbg_s$.)  
Since the action of $\bbg_s$ on $\bbu$ is defined over $\bbq$, one gets the other claims of part 2 and part 3. Moreover, for large enough $p$, $\pi_p(\Gamma)=\pi_p(K_p)\ltimes \pi_p(U_p)$ is a perfect group. Hence by \cite[Corollary 14]{SGV} the restriction of any non-trivial representation $\rho$ of $\pi_p(\Gamma)$ to $\pi_p(K_p)$ is non-trivial. Since $K_p$ is hyperspecial, $\pi_p(K_p)$ is a product of quasi-simple groups where the number of factors and their ranks are bounded by $\dim \bbg_s$. Therefore by \cite{LS} we have $\dim \rho\ge |\pi_p(K_p)|^{O_{\dim \bbg_s}(1)}\ge |\pi_p(\Gamma)|^{O_{\dim \bbg}(1)}$.
\end{proof}

We notice that the restriction of $\pi_q$ to $\Gamma$ factors through $\wh{\Gamma}$, and so 
\[
\pi_{\prod_i p_i^{n_i}}(\Gamma)\simeq \prod_i \pi_{p_i^{n_i}}(P_{p_i})\simeq  \prod_i \pi_{p_i^{n_i}}(K_{p_i})\ltimes \pi_{p_i^{n_i}}(U_{p_i}).
\]

\subsection{Algebraic homomorphisms and the congruence maps.}\label{ss:AlgHomoAlmostCommMod} In this section, an auxiliary result on the relation between the congruence maps and $\bbq$-group homomorphisms is proved (see Lemma~\ref{l:AlmostCommute}). This relation is crucial for reducing the proof of the main theorem to the case where the unipotent radical is abelian. 

Here is a general formulation of the issue (without a reference to affine group schemes): suppose $\bbh_1\subseteq (\GL_{n_1})_{\bbq}$ and $\bbh_2\subseteq (\GL_{n_2})_{\bbq}$ are given, and $\rho:\bbh_1\rightarrow \bbh_2$ is a $\bbq$-group homomorphism. For any positive integer $q$ and $i=1,2$, let $\pi_q$ be the induced group homomorphism on $\prod_p\GL_{n_i}(\bbz_p)$. For any prime $p$, let $H_{i,p}:=\bbh_i(\bbq_p)\cap \GL_{n_i}(\bbz_p)$. One would like to have that $\rho$ and $\pi_q$ {\em commute}; that means the following (wrong) isomorphism
\be\label{e:commuting}
\pi_q(\rho(\prod_p H_{1,p}))\simeq \rho(\pi_q(\prod_p H_{1,p})).
\ee

Equation (\ref{e:commuting}) has two issues: (a) $\rho(\prod_p H_{1,p})$ is not necessarily a subgroup of $\prod_p H_{2,p}$, and so it does not make sense to talk about its congruence quotient, (b) $\rho$ is defined by polynomials with {\em rational} coefficients that are not necessarily integer.
So it does not make sense to talk about $\rho$ of the finite group $\pi_q(\prod_p H_{1,p})$.

In Lemma~\ref{l:AlmostCommute} under the assumption that $\rho$ is surjective with a $\bbq$-section, an {\em almost commuting} of $\rho$ and $\pi_q$ is proved; in particular, it is proved that, if $q$ does not have small prime factors, then the isomorphism in Equation (\ref{e:commuting}) holds.

The main ingredient in the proof of Lemma~\ref{l:AlmostCommute} is the fact that a $\bbq$-homomorphism $\rho:\bbh_1\rightarrow \bbh_2$ between two algebraic $\bbq$-groups $\bbh_1$ and $\bbh_2$ induces a continuous homomorphism between the adelic points $\bbh_1(\bba_{\bbq})$ and $\bbh_2(\bba_{\bbq})$. 

After proving Lemma~\ref{l:AlmostCommute}, it will be used for the quotient maps $\rho_1:\bbg\rightarrow\bbg_s$ and  $\rho_2:\bbg\rightarrow\bbg_s\ltimes\bbu/[\bbu,\bbu]$ where $\bbg_s$ is a Levi $\bbq$-subgroup of $\bbg$, and $\bbu$ is the unipotent radical of $\bbg$. 

\begin{lem}\label{l:AlmostCommute}
Let $\bbh\subseteq (\GL_{n_1})_{\bbq}$ be a $\bbq$-subgroup (with a given embedding). Let $\rho:\bbh\rightarrow (\GL_{n_2})_{\bbq}$ be a $\bbq$-homomorphism. Then the following hold.
\begin{enumerate}
	\item There is $g\in\GL_{n_2}(\bbq)$ such that $\rho(\bbh(\bbq)\cap\GL_{n_1}(\bbz[1/q_0]))\subseteq g\GL_{n_2}(\bbz[1/q_0])g^{-1}$. 
	\item There is a positive integer $q'$ such that for any prime $p$ and any $h\in \bbh(\bbq_p)$ we have
	\[
	\|\rho(h)-1\|_p\le |q'|_p^{-1}\|h-1\|_p.
	\]
	\item Let $1\rightarrow \bbl \rightarrow \bbh \xrightarrow{\rho}\bar\bbh \rightarrow 1$ be a short exact sequence of $\bbq$-groups. Suppose there is a $\bbq$-section $s:\bar\bbh\rightarrow \bbh$ such that $s(1)=1$. Let $q'_0$ be a positive integer and $k_p:=v_p(q_0')$. For any prime $p\nmid q_0$, let $Q_p=\bbh(\bbq_p)\cap (1+q'_0\gl_{n_1}(\bbz_p))=\{h_p\in \bbh(\bbq_p)|\h \|h_p-1\|\le p^{-k_p}\}$. Assume that $\rho(Q_p)\subseteq \GL_{n_2}(\bbz_p)$, and $s(\rho(Q_p))\subseteq \GL_{n_1}(\bbz_p)$ for any prime $p\nmid q_0$. Then there is a positive integer $q':=\prod_p p^{k'_p}$ such that for non-negative integers $n_p\ge k'_p+k_p$ we have
	\be\label{e:AlmostCommute}
	\prod_{p\nmid q_0} \pi_{p^{n_p+k'_p}}(Q_p)/\pi_{p^{n_p+k'_p}}(\bbl(\bbq_p)\cap Q_p) \twoheadrightarrow
	\prod_{p\nmid q_0} \pi_{p^{n_p}}(\rho(Q_p)) \twoheadrightarrow
	\prod_{p\nmid q_0} \pi_{p^{n_p-k'_p}}(Q_p)/\pi_{p^{n_p-k'_p}}(\bbl(\bbq_p)\cap Q_p).
	\ee
\end{enumerate}
\end{lem}
\begin{proof}
For any $p$ there is $g_p\in \GL_{n_2}(\bbq_p)$ such that $g_p\rho(\bbh(\bbq_p)\cap \GL_{n_1}(\bbz_p))g_p^{-1}\subseteq \GL_{n_2}(\bbz_p)$. Since $\rho$ is defined over $\bbq$, for large enough $p$, we can assume that $g_p=1$. Hence there is $\wh{g}=(g_p)\in \GL_{n_2}(\bba_{\bbq})$ such that 
\[
\wh{g}\Delta(\rho(\bbh(\bbq)\cap \GL_{n_1}(\bbz[1/q_0]))) \wh{g}^{-1} \subseteq \GL_{n_2}(\bbr)\prod_{p|q_0}\GL_{n_2}(\bbq_p)\prod_{p\nmid q_0}\GL_{n_2}(\bbz_p),
\]
where $\Delta$ is the diagonal embedding.
On the other hand, we have  
\[
\GL_{n_2}(\bba_{\bbq})=\Delta(\GL_{n_2}(\bbq))\cdot \left(\GL_{n_2}(\bbr)\prod_{p|q_0}\GL_{n_2}(\bbq_p)\prod_{p\nmid q_0}\GL_{n_2}(\bbz_p)\right).
\]
 So there is $g\in \GL_{n_2}(\bbq)$ such that 
\begin{align*}
\Delta(g \rho(\bbh(\bbq)\cap \GL_{n_1}(\bbz[1/q_0])) g^{-1}) &\subseteq \left(\GL_{n_2}(\bbr)\prod_{p|q_0}\GL_{n_2}(\bbq_p)\prod_{p\nmid q_0}\GL_{n_2}(\bbz_p)\right)\cap \Delta(\GL_{n_2}(\bbq))\\
&=\Delta(\GL_{n_2}(\bbz[1/q_0])),
\end{align*}
which gives us the first part.

The second part is an easy consequence of the fact that $\rho$ can be represented by a polynomial with rational coefficients. Similarly there is a positive integer $q'$ such that 
\be\label{e:AdelicLipschitz}
|q'|_p\|s(\rho(h_p))-1\|_p\le \|\rho(h_p)-1\|_p \le |q'|_p^{-1} \|h_p-1\|_p,
\ee
for any prime $p$ and $h_p\in \bbh(\bbq_p)$. On the other hand, 
\be\label{e:NormKernel}
h_p\in \ker(\pi_{p^n}\circ \rho) \text{ if and only if } \|\rho(h_p)-1\|_p\le p^{-n}.
\ee
 These imply that for $h_p\in  \ker(\pi_{p^n}\circ \rho)$ we have $\pi_{p^{n-k'_p}}(s(\rho(h_p)))=1$. And so if $h_p\in Q_p$ and $\pi_{p^n}(\rho(h_p))=1$, then we have 
 \be\label{e:FactoringThrough}
  h_p=(h_ps(\rho(h_p))^{-1})\cdot s(\rho(h_p))\in (\bbl(\bbq_p)\cap Q_p) \ker \pi_{p^{n-k'_p}}.
 \ee
 Let $\wt{\phi}_1:Q_p\rightarrow \pi_{p^{n_p-k'_p}}(Q_p)/\pi_{p^{n_p-k'_p}}(\bbl(\bbq_p)\cap Q_p), \wt{\phi}_1(h_p):=\pi_{p^{n_p-k'_p}}(h_p)\pi_{p^{n_p-k'_p}}(\bbl(\bbq_p)\cap Q_p)$. Then clearly $\wt{\phi}_1$ is surjective. By (\ref{e:FactoringThrough}), $\wt{\phi}_1$ factors through $\pi_{p^{n_p}}(\rho(Q_p))$; that means
\[
\pi_{p^{n_p}}(\rho(Q_p)) \twoheadrightarrow
 \pi_{p^{n_p-k'_p}}(Q_p)/\pi_{p^{n_p-k'_p}}(\bbl(\bbq_p)\cap Q_p).
\]
Let $\wt{\phi}_2:Q_p\rightarrow \pi_{p^{n_p}}(\rho(Q_p)), \wt{\phi}_2(h_p):=\pi_{p^{n_p}}(\rho(h_p))$. So clearly $\wt{\phi}_2$ is surjective. By (\ref{e:AdelicLipschitz}) and (\ref{e:NormKernel}), we have, if $\pi_{p^{n_p+k'_p}}(h_p)=1$, then $\|\rho(h_p)-1\|_p\le |q'|_p^{-1}\|h_p-1\|_p\le p^{k'_p}p^{-n_p-k'_p}=p^{-n_p}$; and so $h_p\in \ker \wt{\phi}_2$. Therefore $\wt{\phi}_2$ factors through $\pi_{p^{n_p+k'_p}}(Q_p)$. By the definition of $\wt{\phi}_2$ we have $\bbl(\bbq_p)\cap Q_p\subseteq \ker \wt{\phi}_2$. Hence we get a surjection $ \pi_{p^{n_p+k'_p}}(Q_p)/\pi_{p^{n_p+k'_p}}(\bbl(\bbq_p)\cap Q_p) \twoheadrightarrow
\pi_{p^{n_p}}(\rho(Q_p))$. The claim follows.
\end{proof}

Let $\bbg_s$ be a simply connected semisimple $\bbq$-group which acts on a unipotent $\bbq$-group $\bbu$. Then by the virtue of the proof of~\cite[Corollary 14]{SGV} $\bbg=\bbg_s\ltimes \bbu$ is a perfect group if and only if $\bbg_s\ltimes \bbu/[\bbu,\bbu]$ is perfect. On the other hand, $\bbv:=\bbu/[\bbu,\bbu]$ is a vector $\bbq$-group, and the action of $\bbg_s$ on $\bbu$ induces a representation $\phi_0:\bbg_s\rightarrow \mathbb{GL}(\bbv)$ defined over $\bbq$. And $\bbg_s\ltimes \bbv$ is perfect if and only if the trivial representation is not a subrepresentation of $\phi_0$. Consider the following short exact sequences of $\bbq$-algebraic groups
\be\label{e:Exact}
1\rightarrow \bbu \rightarrow \bbg \xrightarrow{\rho_1} \bbg_s \rightarrow 1
\hspace{1cm}\text{and}\hspace{1cm} 
1\rightarrow [\bbu,\bbu] \rightarrow \bbg \xrightarrow{\rho_2} \bbg_s\ltimes \bbv \rightarrow 1.
\ee
We notice that, since the $\bbu$ is a $\bbq$-unipotent group, there is a $\bbq$-section from $\bbv$ to $\bbu$, which can be extended to a $\bbq$-section for the second exact sequence of (\ref{e:Exact}). The first exact sequence clearly splits. We fix certain $\bbq$-embeddings $\bbg_s\subseteq \bbg\subseteq (\GL_{n_1})_{\bbq}$ and $\bbg_s\ltimes \bbv\subseteq (\GL_{n_2})_{\bbq}$. Let $\Gamma\subseteq \bbg\cap \GL_{n_1}(\bbz[1/q_0])$ be a Zariski-dense subgroup. Then by Lemma~\ref{l:AlmostCommute} after passing to a finite index subgroup, if necessary, we can assume that $\rho_2(\Gamma)\subseteq \GL_{n_2}(\bbz[1/q_0])$. By the above comment, we are also allowed to use (\ref{e:AlmostCommute}) in Lemma~\ref{l:AlmostCommute} for the $p$-adic closure of $\Gamma$ in $\bbg(\bbq_p)$ for any $p\nmid q_0$.

\subsection{Ping-pong players.}  As in~\cite{SGV, SGsemisimple}, we work with a Zariski-dense free subgroup of $\Gamma$.  
Following \cite[Proposition 22]{SGsemisimple}, using \cite[Proposition 7, Proposition 17, Proposition 20]{SGV} we get the following.
\begin{prop}\label{p:EscapeSubgroups}
Let $\Omega$ be a finite symmetric subset of $\GL_n(\bbq)$ which generates a Zariski-dense subgroup $\Gamma$ of a Zariski-connected perfect group $\bbg$. Then there are a finite subset $\overline{\Omega'}$ of $\Gamma$ and $\delta_0>0$ and $l_0$ (which depend on $\Omega$) such that $\overline{\Omega'}$ freely generates a Zariski-dense subgroup of $\bbg(\bbq)$ and
\be\label{e:EscapeSubgroups}
\pcal_{\Omega'}^{(l)}(\bbh(F))\le e^{-\delta_0 l}
\ee
for any field extension $F/\bbq$, and any proper subgroup $\bbh\subsetneq \bbg\times_{\bbq} F$ and $l\ge l_0$, where $\Omega'=\overline{\Omega'}\cup\overline{\Omega'}^{-1}$ and $\pcal_{\Omega'}$ is the probability counting measure on $\Omega'$. Moreover, $\Ad(\rho_1(\overline{\Omega'}))$ freely generates a subgroup of $\Ad(\bbg_s(\bbq))$, where $\rho_1:\bbg\rightarrow \bbg_s$ is the quotient map; in particular, $\ker(\Ad\circ \rho_1)\cap \langle \Omega'\rangle=\{1\}$.
\end{prop}
\begin{proof} This is essentially proved in \cite[Proposition 7, Proposition 17, Proposition 20]{SGV}. Here is an outline of the argument. 

By \cite[Proposition 17]{SGV} there are irreducible representations $\eta_i$ of $\bbg:\rightarrow \mathbb{GL}(\bbw_i)$ defined over local fields $F_i$ and algebraic deformations $\phi_{i,v}:\bbg\rightarrow {\rm Aff}(\bbw_i)$ of $\eta_i$ where the parameter $v$ varies in an affine space $\bbv_i$ with the following properties:
\begin{enumerate}
	\item for any $i$, $\eta_i$ factors through $\bbg_s$, and at least one of them factors through $\Ad(\bbg_s)$;
	\item for any $i$, $\eta_i(\Gamma)\subseteq \GL(\bbw_i(F_i))$ is unbounded;
	\item the linear part of $\phi_{i,v}$ is $\eta_i$ for any $v\in \bbv_i(F_i)$;
	\item for any non-zero vector $v$, the affine action of $\bbg(F_i)$ on $\bbw_i(F_i)$ via $\phi_{i,v}$ has no fixed point;
	\item for any proper algebraic subgroup $\bbh$ of $\bbg$ for some $i$ either 
	\begin{enumerate}
	\item there is a non-zero vector $w\in \bbw_i(F_i)$ such that $\eta_i(\bbh(F_i))[w]=[w]$ where $[w]$ is the line $F_iw$, or
	\item there is a parameter $v\in \bbv_i(F_i)$ of norm 1 and a point $w\in \bbw_i(F_i)$ such that $w$ is fixed by $\bbh$ under the affine action of $\bbg$ on $\bbw_i$ given by $\phi_{i,v}$.
	\end{enumerate}	
\end{enumerate} 
By \cite[Proposition 20 and Proposition 21]{SGV}, there is a finite subset $\overline{\Omega'}\subseteq \Gamma$ such that 
\begin{enumerate}
	\item for any $i$, $\eta_i(\overline{\Omega'})$ freely generates a subgroup of $\eta_i(\Gamma)$; in particular $\Ad(\rho_1(\overline{\Omega'}))$ freely generates a subgroup of $\Ad(\bbg_s(\bbq))$ where $\rho_1:\bbg\rightarrow \bbg_s$ is the projection to the semisimple part.
	\item there is $c_0>0$ such that, for any index $i$, and non-zero vector $w_0\in \bbw_i(F_i)$, we have
	\[
	|\{\gamma\in B_l(\overline{\Omega'})|\h \eta_i(\gamma)[w_0]=[w_0]\}|\le |B_l(\overline{\Omega'})|^{1-c_0},
	\]
	and, for any point $w_1\in \bbw_i(F_i)$, and parameter $v_i\in \bbv_i(F_i)$ with norm one,
	\[
	|\{\gamma\in B_l(\overline{\Omega'})|\h \phi_{i,v}(\gamma)(w_1)=(w_1)\}|\le |B_l(\overline{\Omega'})|^{1-c_0},	
	\]
	where $B_l(\overline{\Omega'})$ is the the set of reduced words over $\overline{\Omega}'$ of length $l$.
	\end{enumerate}
Hence, by the above geometric description of proper algebraic subgroups, for any proper algebraic subgroup $\bbh$ of $\bbg$ we have
\[
|B_l(\overline{\Omega'})\cap \bbh(\bbq)|<|B_l(\overline{\Omega'})|^{1-c_0}.
\]
Let $\Omega':=\overline{\Omega'}\cup\overline{\Omega'}^{-1}$. By a result of Kesten~\cite[Theorem 3]{Kes} we have
$
\pcal_{\Omega'}^{(l)}(I)\le \left(\frac{2|\overline{\Omega'}|-1}{|\overline{\Omega'}|^2}\right)^{l},
$
where $I$ is the identity matrix. As $\Omega'$ is symmetric, we have $\pcal_{\Omega'}^{(2l)}(I)=\sum_{\gamma} \pcal_{\Omega'}^{(l)}(\gamma)^2$. Hence together with the Kesten bound we have  
\be\label{e:KestenBound}
\pcal_{\Omega'}^{(l)}(\gamma)\le \pcal_{\Omega'}^{(2l)}(I)\le \left(\frac{2|\overline{\Omega'}|-1}{|\overline{\Omega'}|^2}\right)^{l}.
\ee
 We also notice that, since the random-walk in the group generated by $\Omega'$ with respect to the probability counting measure $\pcal_{\Omega'}$ on $\Omega'$ can be identified with the random-walk on a regular tree, $\pcal_{\Omega'}^{(l)}$ is constant on the set $B_k(\overline{\Omega'})$ of reduced words in $\overline{\Omega'}$ of length $k$ for any positive integer $k$. Let $p_{l,k}:=\pcal_{\Omega'}^{(l)}(\gamma)$ for some $\gamma\in B_k(\overline{\Omega'})$.

 For a proper subgroup $\bbh$ of $\bbg\times_{\bbq}F$, let $\overline{\bbh}$ be the Zariski-closure of $\Gamma\cap \bbh(F)$ in $\bbg$. Then $\overline{\bbh}$ is a subgroup of $\bbg$ and $\pcal_{{\Omega'}}^{(l)}(\bbh(F))=\pcal_{{\Omega'}}^{(l)}(\overline{\bbh}(\bbq))$. Now we can finish the argument as in the proof of \cite[Proposition 7]{SGV}:
\begin{align*}
\pcal_{{\Omega'}}^{(l)}(\bbh(F))&=\pcal_{{\Omega'}}^{(l)}(\overline{\bbh}(\bbq))= \sum_{k\ge 0} |B_{\overline{\Omega'}}(k)\cap \overline{\bbh}(\bbq)|\h p_{l,k}\\
&\le \left(\sum_{0\le k< l/10} (2|\overline{\Omega'}|)^{k}\right) \left(\frac{2|\overline{\Omega'}|-1}{|\overline{\Omega'}|^2}\right)^{l}+ (2|\overline{\Omega'}|-1)^{-c_0l/10}\sum_{k\ge l/10} |B_k(\overline{\Omega'})| p_{l,k}\\
&\le \frac{(2|\overline{\Omega'}|)^{l/10}}{(|\overline{\Omega'}|/2)^l}+ (2|\overline{\Omega'}|-1)^{-c_0l/10} \le e^{-\delta_0 l},
\end{align*}
for some $\delta_0$ and any $l\ge l_0$ where $l_0$ is a fixed positive integer. 
\end{proof}
\begin{lem}
Let $\Omega$ be a finite symmetric subset of $\GL_{n_0}(\bbz[1/q_0])$ and $\Gamma=\langle \Omega\rangle$. Assume the Zariski-closure $\bbg$ of $\Gamma$ is a connected, simply connected, perfect $\bbq$-group. Suppose $\Omega'\subseteq \Gamma$ is a finite symmetric set which generates a Zariski-dense subgroup $\Lambda$ of $\bbg$. Then for any set of positive integers $\ccal$ consisting of integers coprime to $q_0$
if $\{\Cay(\pi_q(\Lambda),\pi_q(\Omega'))\}_{q\in\ccal}$ is a family of expanders, then $\{\Cay(\pi_q(\Gamma),\pi_q(\Omega))\}_{q\in\ccal}$ is a family of expanders.
\end{lem}
\begin{proof}
By Nori's strong approximation (see \cite[Theorem 5.4]{Nor}), we have that the closure $\wh{\Lambda}$ of $\Lambda$ in $\prod_{p\nmid q_0}\GL_{n_0}(\bbz_p)$ is of finite index in the closure $\wh{\Gamma}$ of $\Gamma$ in $\prod_{p\nmid q_0}\GL_{n_0}(\bbz_p)$. For any $q$, any representation $\rho$ of $\pi_q(\Gamma)$ can be viewed as a representation of $\wh{\Gamma}$. The extension of $\rho$ to $\wh{\Gamma}$ is denoted by $\rho$ too. Only finitely many of them have $\wh{\Lambda}$ in their kernel. Hence the restriction to $\Lambda$ of only finitely many such representation is trivial. Let $\rho$ be an irreducible representation of $\pi_q(\Gamma)$ whose restriction to $\Lambda$ is not trivial. Hence the restriction of $\rho$ to $\Lambda$ is a subrepresentation of $L^2(\pi_q(\Lambda))^{\circ}:=\{f\in L^2(\pi_q(\Lambda))|\h \sum_{g\in \pi_q(\Lambda)}f(g)=0\}$. 

Let $u\in V_{\rho}$ be a unit vector which is an $\vare$-almost invariant vector with respect to $\Omega$; that means 
\[
\|\rho(\gamma)(u)-u\|<\vare, 
\]
for any $\gamma\in \Omega$. Since $\Omega$ generates $\Gamma$ and $\Omega'\subseteq \Gamma$, for some $r_0$ we have $\Omega'\subseteq \prod_{r_0} \Omega$. Thus for any $\lambda\in \Omega'$ we have
\[
\|\rho(\lambda)(u)-(u)\|<r_0\vare.
\]
Assuming $\{\Cay(\pi_q(\Lambda),\pi_q(\Omega'))\}_{q\in\ccal}$ is a family of expanders, there is $\vare_0>0$ such that 
\[
\max_{\lambda\in \Omega'} \|\rho(\lambda)(u)-u\|>\vare_0.
\]
Therefore $\vare>\vare_0/r_0$, which implies that no non-trivial irreducible representation of $\pi_q(\Gamma)$ (for $q\in\ccal$) has an $\vare'$-almost invariant vector for some $\vare'>0$.  
\end{proof}
\begin{cor}\label{c:InitialReductions} 
It is  enough to prove the sufficiency part of Theorem~\ref{t:main} under the additional assumptions:
  \begin{enumerate}
  	\item $\Omega=\overline{\Omega}\cup \overline{\Omega}^{-1}$ where $\Ad(\rho_1(\overline{\Omega}))$ freely generates  a Zariski-dense subgroup of $\Ad(\bbg_s)$. 
  	\item For any field extension $F/\bbq$ and proper subgroup $\bbh$ of $\bbg\times_{\bbq} F$, (\ref{e:EscapeSubgroups}) holds.
  \end{enumerate}
\end{cor} 
\begin{remark}\label{r:ChangingOmega}
Notice that by changing $\Omega$, the closure $\Gamma_p$ of $\Gamma$ in $\GL_{n_0}(\bbz_p)$ is of the form described in Lemma~\ref{l:AdelicClosure} only for large enough $p$. For small primes $\Gamma_p$ is still an open subgroup of $\bbg(\bbq_p)$. 
\end{remark}
\subsection{Summary of the initial reductions.}\label{ss:InitialReductions} In the rest of the article, (unless it is explicitly mentioned otherwise)  we assume that $\Omega\subseteq \GL_{n_0}(\bbz[1/q_0])$ is a finite symmetric set, and set $\Gamma=\langle \Omega\rangle$. And let $\gcal$ and $\bbg$ be the Zariski-closure of $\Gamma$ in $(\GL_{n_0})_{\bbz[1/q_0]}$ and $(\GL_{n_0})_{\bbq}$, respectively. 

\begin{enumerate}
	\item(Zariski closure)  By Lemma~\ref{l:ReductionToConnectedSimplyConnectedCase}, we can and will assume that $\bbg$ is a connected, simply connected, perfect group. We let $\bbg_s$ be a Levi $\bbq$-subgroup of $\bbg$, and $\bbu$ be the unipotent radical of $\bbg$. 
	
	\item(Adelic closure) By Lemma~\ref{l:AdelicClosure} and Remark~\ref{r:ChangingOmega}, we can and will assume that  the closure $\wh{\Gamma}$ of $\Gamma$ in $\prod_{p\nmid q_0}\GL_{n_0}(\bbz_p)$ is an open subgroup of $\prod_{p\nmid q_0} \bbg(\bbq_p)$ which contains a finite index subgroup of the form $\prod_{p\nmid q_0} P_p$ where $P_p=K_p\ltimes U_p$ and $K_p$ and $U_p$ satisfy the properties mentioned in Lemma~\ref{l:AdelicClosure}. In particular, $P_p=\bbg(\bbq_p)\cap\GL_{n_0}(\bbz_p)[p^{k_p}]$, $K_p=\bbg_s(\bbq_p)\cap\GL_{n_0}(\bbz_p)[p^{k_p}]$, and $U_p=\bbu(\bbq_p)\cap\GL_{n_0}(\bbz_p)[p^{k_p}]$ for some non-negative integers $k_p$ that are zero for large enough $p$.  
	
	\item(Generating set) We can and will assume that $\Omega$ is as in Corollary~\ref{c:InitialReductions}. In particular, $\ker(\Ad\circ \rho_1)\cap \Gamma=\{1\}$, and any proper algebraic subgroup of $\bbg$ intersects $\Gamma$ in an {\em exponentially small} subset (see Inequality (\ref{e:EscapeSubgroups})). 
\end{enumerate}

\subsection{Multiplicity bound.}
In order to execute Sarnak-Xue~\cite{SX} trick, we need to have a control on the degree of irreducible representations of $\pi_q(\Gamma)$. Here we prove such a result for an arbitrary modulus $q$ which is of independent interest. 
\begin{prop}\label{p:HighMultiplicity}
Let $\Gamma\subseteq \GL_n(\bbz[1/q_0])$ be a finitely generated group whose Zariski-closure $\bbg$ is a Zariski-connected perfect group. Suppose $\bbg$ is simply-connected; that means $\bbg=\bbg_s\ltimes \bbu$ where $\bbg_s$ is a simply-connected semisimple group and $\bbu$ is a unipotent group. Let $q$ be an integer which is co-prime to $q_0$. Let $\wh{\Gamma}$ be the closure of $\Gamma$ in $\prod_{p\nmid q_0}\GL_n(\bbz_p)$.

Then there are a finite set $\Sigma$ of complex irreducible representations of $\wh{\Gamma}$ and $\delta>0$ which depend on $\Gamma$ such that the following holds.

For any complex irreducible representation $\rho$ of $\pi_q(\Gamma)$ which does not factor through $\pi_{q'}(\Gamma)$ for any proper divisor $q'$ of $q$, we have that either
\begin{enumerate}
	\item The lift of $\rho$ to $\wh{\Gamma}$ is in $\Sigma$, or
	\item $\dim \rho\ge |\pi_q(\Gamma)|^{\delta}$.
\end{enumerate}
\end{prop}
For a group with entries in a number field and a semisimple Zariski-closure, a similar statement is proved in \cite[Proposition 33]{SGsemisimple}. An identical argument works here, but for the sake of the convenience a proof is included. It is worth pointing out that the main (behind the scene) representation theoretic tool is Howe's Kirillov theory for compact $p$-adic analytic groups~\cite{How}. 

For any odd prime $p$, let $a_0(p):=1$, and let $a_0(2):=2$. Then $\exp:p^{a_0(p)}\gl_n(\bbz_p)\rightarrow 1+p^{a_0(p)}\gl_n(\bbz_p)$, and $\log$ going backward, can be defined and satisfy the usual properties. Let $b_0(p):=1$ for any prime $p\ge 5$, and let $b_0(2)=b_0(3)=2$. Then, by Campbell-Hausdorff formula, we have that for any Lie $\bbz_p$-subalgebra $\hfr$ of $p^{a_0(p)}\gl_n(\bbz_p)$ we have that $\exp(\hfr)$ is a subgroup of $1+p^{a_0(p)}\gl_n(\bbz_p)$.

\begin{lem}\label{l:Local}
	Let $\bbh$ be an algebraic subgroup of $(\GL_n)_{\bbq_p}$. Let $H_c:=\bbh(\bbq_p)\cap (1+p^c\gl_n(\bbz_p))$ for any positive integer $c$. Let $\hfr:=\Lie(\bbh)(\bbq_p)\cap \gl_n(\bbz_p)$. Suppose $p^{s_0}\hfr \subseteq [\hfr,\hfr]$. Let $\rho$ be an irreducible complex representation of $H_c$, $c\ge b_0(p)+1$, and 
	\[
	m_{\rho}:=\min\{m\in \bbz^+|\h H_m\subseteq \ker(\rho)\}.
	\]  
	Then $\dim \rho \ge p^{(m_{\rho}-1-c-s_0)/2}$.
\end{lem}
\begin{proof}
	Since the analytic Lie algebra of $\bbh(\bbq_p)$ coincides with $\Lie(\bbh)(\bbq_p)$, we get that 
	\[
	\log:H_m\rightarrow p^m \hfr
	\]
	is a bijection and the exponential map is its inverse (for a detailed proof see~\cite[Proposition 7]{SGZ}). We notice also that $p^{m+s_0} (p^m\hfr) \subseteq [\pfr^m\hfr,\pfr^m\hfr]\subseteq p^m (p^m \hfr)$. Hence by \cite[Lemma 32]{SGsemisimple} we have 
	\[
	\dim \rho\ge p^{(m_{\rho}-1-c-s_0)/2}.
	\]
\end{proof}

\begin{proof}[Proof of Proposition~\ref{p:HighMultiplicity}]
Let $G_{p,m}:=\bbg(\bbq_p)\cap (1+p^m\GL_n(\bbz_p))$ for any prime $p$ and non-negative integer $m$. Let $\wh{\Gamma}$ be the closure of $\Gamma$ in $\prod_{p\nmid q_0} \GL_n(\bbz_p)$, let $\Gamma_p$ be the closure of $\Gamma$ in $\GL_n(\bbz_p)$, and $\Gamma_p[p^m]:=\Gamma_p\cap G_{p,m}$.  As in Lemma~\ref{l:AdelicClosure}, by Nori's strong approximation~\cite[Theorem 5.4]{Nor}, there is a finite-index subgroup $\Lambda$ of $\Gamma$ such that its closure $\wh{\Lambda}$ in $\prod_{p\nmid q_0}\GL_n(\bbz_p)$ is of the form 
 $\prod_{p\nmid q_0} \Lambda_p$, where $\Lambda_p$'s satisfy the following properties. 
\begin{enumerate}
\item For some positive integer $c_0$ we have that, for any prime $p\ge c_0$, $\Lambda_p=G_{p,0}$ and it is a perfect group, and $\pi_p(\Lambda_p)$ is generated by $p$-elements;
\item For any prime $p<c_0$, we have $\Lambda_p=G_{p,c(p)}$ where $c(p)$ is chosen large enough such that $\Lambda_p\subseteq G_{p,8}$.
\end{enumerate}
Since $\bbg$ is perfect, $[\Lie(\bbg)(\bbq),\Lie(\bbg)(\bbq)]=\Lie(\bbg)(\bbq)$. Hence there is a positive integer $s_0$ such that for any prime $p$ we have
\[
p^{s_0} \gfr_p\subseteq [\gfr_p,\gfr_p]
\]
where $\gfr_p:=\Lie(\bbg)(\bbq_p)\cap \gl_n(\bbz_p)$. 
 
 Let $\rho:\pi_q(\Gamma)\rightarrow \GL(V)$ be a complex irreducible representation which does not factor through $\pi_{q'}(\Gamma)$ for any proper divisor $q'$ of $q$. Then $V=\oplus_{i}\rho(g_i)W$ where $W$ is a simple $\pi_q(\Lambda)$-module and $g_i$'s are some elements of $\pi_q(\Gamma)$. Suppose $q=\prod_i p_i^{n_i}$ is the prime factorization of $q$. Since  $\pi_q(\Lambda)\simeq \prod_i \pi_{p_i^{n_i}}(\Lambda_{p_i})$, there are irreducible representations $\eta_i:\Lambda_{p_i}\rightarrow \GL(W_i)$ such that $W$ is isomorphic to $\otimes_i W_i$ as an $\wh{\Lambda}$-module. In particular, $\dim W=\prod_i \dim W_i$. 
 
  Except for finitely many irreducible representations $\rho$ of $\wh{\Gamma}$, we have that $\wh{\Lambda}$ acts non-trivially on $V$, and so $\otimes_i \eta_i$ is a non-trivial irreducible representation of $\wh{\Lambda}$. Suppose that 
  \[
  \Sigma\supseteq \{\rho\in {\rm Irr}(\wh{\Gamma})|\h \wh{\Lambda}\subseteq \ker(\rho)\},
  \]
 where ${\rm Irr}(\wh{\Gamma})$ is the set of all the complex irreducible representations of $\wh{\Gamma}$.
 
 For any $i$, $\eta_i$ is an irreducible representation of $\Lambda_{p_i}=G_{p_i,c(p_i)}$ where $c(p_i)$ is $0$ for $p\ge c_0$ and given as above for $c<p_0$. Since $\rho$ does not factor through $\pi_{q'}(\Gamma)$ for a proper divisor $q'$ of $q$, for any $i$ we have 
 \[
 m_i:=\min\{m\in \bbz^+|\h G_{p_i,m}\subseteq \ker(\eta_i)\}=\max\{n_i,c(p_i)\}.
 \]

  Next we partition the indexes into five types: (1) $n_i-c(p_i)\ge 2s_0+2$, and $c(p_i)\ge 8$; (2) $n_i\ge 2s_0+10$ and $c(p_i)=0$; (3) $1<n_i< 2s_0+10$ and $p_i\ge c_0$; (4) $n_i=1$ and $p_i\ge c_0$; and (5) $p_i<c_0$ and $n_i< 2s_0+10$. Now we find a lower bound for the dimension of $\eta_i$ for the first three types of indexes. 

{\bf Case 1.} $n_i-c(p_i)\ge 2s_0+2$ and $c(p_i)\ge 8$. 

In this case, by Lemma~\ref{l:Local}, we have 
\[
\dim \eta_i\ge p^{(n_i-c(p_i)-1-s_0)/2}\ge p^{(n_i-c(p_i))/4}=|\pi_{p_i^{n_i}}(\Lambda_{p_i})|^{\delta_1}
\]
where $\delta_1$ is a positive number which depends only on $\dim \bbg$.

{\bf Case 2.} $n_i\ge 2s_0+10$ and $c(p_i)=0$.

In this case, the restriction of $\eta_i$ to $G_{p_i,2}$ is non-trivial. Hence by Lemma~\ref{l:Local} we have 
\[
\dim \eta_i\ge p^{(n_i-2-1-s_0)/2}\ge p^{n_i/4} \ge |\pi_{p_i^{n_i}}(\Lambda_{p_i})|^{\delta_2}
\]
 where $\delta_2$ is a positive number which depends only on $\dim \bbg$.

{\bf Case 3.} $1< n_i< 2s_0+10$ and $p\ge c_0$. 

Notice that $\pi_{p_i^{n_i}}(G_{p_i,1})$ is a normal $p_i$-subgroup of $\pi_{p_i^{n_i}}(\Lambda_{p_i})$. Since the dimension of any irreducible representation of a $p_i$-group is a power of $p_i$, we have that either the restriction of $\eta_i$ to $G_{p_i,1}$ factors through $G_{p_i,1}/[G_{p_i,1},G_{p_i,1}]$ or $\dim \eta_i\ge p_i \ge |\pi_{p_i^{n_i}}(\Lambda_{p_i})|^{\delta_3}$, where $\delta_3$ is a positive number which depends on $\bbg$. 

Now suppose the restriction of $\eta_i$ to $G_{p_i,1}$ factors through $G_{p_i,1}/[G_{p_i,1},G_{p_i,1}]$. Then there is a line $l_i$ in $W_i$ which is invariant under $G_{p_i,1}$ and $W_i=\sum_{g \in G_{p_i,0}} \eta_i(g)(l_i)$. Since $G_{p_i,0}$ is perfect, the stabilizer of $l_i$ should be a proper subgroup of $G_{p_i,0}$. Since $\pi_{p_i^{n_i}}(G_{p_i,0})$ is generated by $p_i$-elements, we get that the index of the stabilizer of $l_i$ in $G_{p_i,0}$ is at least $p_i$. And so again $\dim \eta_i\ge p_i \ge |\pi_{p_i^{n_i}}(\Lambda_{p_i})|^{\delta_3}$.

{\bf Case 4.} $n_i=1$ and $p_i\ge c_0$.

By \cite{LS} (see \cite[Section 4.2]{SGV}), we have $\dim \eta_i\ge |\pi_{p_i^{n_i}}(\Lambda_{p_i})|^{\delta_4}$ where $\delta_4$ is a positive number which depends only on $\bbg$.

Since $\prod_{i \text{ in Case 5}} p_i^{n_i}\ll_{\bbg} 1$, by the above cases we get that either $q\ll_{\Gamma}1$ or $\dim \rho\ge |\pi_q(\Gamma)|^{\delta}$ where $\delta$ is a positive number which depends only on $\Gamma$. By enlarging $\Sigma$, we can assume that it contains all the irreducible representations of $\wh{\Gamma}$ which factor through $\pi_q(\Gamma)$ for a positive integer $q\ll_{\Gamma}1$. Such $\Sigma$ and $\delta$ satisfy the claimed properties. 
\end{proof}
\section{Expansion, approximate subgroup, and bounded generation.}
\subsection{Largeness of level-$Q$ approximate subgroups implies super-approximation.}
In this section, the Bourgain-Gamburd machine~\cite{BG1} is used to reduce a proof of spectral gap at level $Q$ to a proof of the {\em largeness} of a {\em generic} level-$Q$ approximate subgroup (see Theorem~\ref{t:ApproximateSubgroup}).

We work in the setting of Section~\ref{ss:InitialReductions}.  As in~\cite{SGsemisimple}, let us introduce the following notation for convenience: 
for a finite symmetric subset $A$ of $\Gamma$, positive integers $l$ and $Q$, and a positive number $\delta$, let $\Pfr_Q(\delta,A,l)$ be the following statement
\be\label{e:P(delta)}
(\pcal_{\Omega}^{(l)}(A)>Q^{-\delta})\h\wedge\h (l>\frac{1}{\delta} \log Q)\h \wedge\h (|\pi_Q(A\cdot A\cdot A)|\le |\pi_Q(A)|^{1+\delta}).
\ee
One can think of $A$ as a {\em generic level-$Q$ approximate subgroup} if $\Pfr_Q(\delta,A,l)$ holds for small enough $\delta$. 

Notice that the random-walk in $\pi_Q(\Gamma)$ gets interesting only at steps $O_{\Omega}(\log |\pi_Q(\Gamma)|)=O_{\Omega}(\log Q)$ since before this time the random-walk with respect to the probability measure $\pcal_{\pi_Q(\Omega)}$ is similar to a random-walk in a regular tree. The first two conditions of $\Pfr_Q(\delta,A,l)$ imply that at $O_{\Omega}(\log Q)$-range of random-walk the probability of hitting $A$ is not {\em exponentially small}. And so $A$ should be fairly generic. For instance Proposition~\ref{p:EscapeSubgroups} implies that $A$ cannot be a subset of a proper algebraic subgroup if $\delta$ is small enough. The genericness of such $A$ was crucial in the proof of \cite[Theorem 35]{SGsemisimple} where the largeness of a generic level-$Q$ approximate subgroup for a semisimple group is proved. 

Let $\ccal$ be either $\{p^n| p\in V_f(\bbq), p\nmid q_0\}$ or $\{q\in \bbz^+|\h \gcd(q,q_0)=1,\h \forall p\in V_f(\bbq), v_p(q_0)\le N\}$, where $v_p(q)$ is the $p$-adic valuation of $q$.

In this section, we deduce Theorem~\ref{t:main}  from the following result.
\begin{thm}[Approximate subgroups]\label{t:ApproximateSubgroup}
In the above setting, for any $\vare>0$ there is $\delta>0$ such that 
\begin{center}
	$\Pfr_Q(\delta,A,l)$ implies that $|\pi_Q(A)| \ge |\pi_Q(\Gamma)|^{1-\vare}$
\end{center}	
if $Q\in \ccal$ and $Q^{\vare^{\Theta_{\bbg}(1)}}\gg_{\Omega} 1$. 
\end{thm}
To show that Theorem~\ref{t:ApproximateSubgroup} implies Theorem~\ref{t:main}, we prove the following proposition based on the same proof as in \cite[Section 3.2]{SGsemisimple}. In particular, we see that the reduction to understanding generic level-$Q$ approximate subgroups can be done in much more generality. 

\begin{prop}\label{p:ReductionToApproximateSubgroup}
	Let $\overline{\Omega}\subseteq\GL_n(\bbz[1/q_0])$ be a finite subset and $\Omega:=\overline{\Omega}\cup\overline{\Omega}^{-1}$. Suppose $\overline{\Omega}$ freely generates a subgroup $\Gamma$. Let $\ccal'$ be a family of positive integers which consists of numbers that are corpime to $q_0$. Suppose $\ccal'$ is closed under divisibility; that means if $q\in \ccal'$ and $q'|q$, then $q'\in \ccal'$. Suppose the assertion of Theorem~\ref{t:ApproximateSubgroup} holds for $\Omega$ and $\ccal'$. Then the family of Cayley graphs $\{\Cay(\pi_q(\Gamma);\pi_q(\Omega))\}_{q\in \ccal'}$ is a family of expanders.   	
\end{prop}
\begin{proof} 	
By Remark~\ref{r:ExpanderSpectralGap}, we have to show $\sup_{q\in \ccal'}\lambda(\pcal_{\pi_q(\Omega)};\pi_q(\Gamma))<1$. For $Q\in \ccal'$, let $T_Q:L^2(\pi_Q(\Gamma))\rightarrow L^2(\pi_Q(\Gamma)), T_Q(f):=\pi_Q[\pcal_{\Omega}]\ast f$ and $\lambda(Q):=\lambda(\pcal_{\pi_Q(\Omega)};\pi_Q(\Gamma))$. It is well-known that $\lambda(Q)$ is the second largest eigenvalue of $T_Q$. Since $L^2(\pi_Q(\Gamma))$ is a completely reducible representation of $\pi_Q(\Gamma)$, there is a unit function $f_0\in L^2(\pi_Q(\Gamma))$ such that $T_Q(f_0)=\lambda(Q) f_0$ and the $\pi_Q(\Gamma)$-module $V_{\rho}$ generated by $f_0$ is a simple $\pi_Q(\Gamma)$-module. Since $\lambda(Q)$ is not one, $\rho$ is not the trivial representation of $\pi_Q(\Gamma)$. Since there are $\dim \rho$-many copies of $\rho$ as subrepresentations of $L^2(\pi_Q(\Gamma))$, we have that the eigenvalue $\lambda(Q)$ of $T_Q$ has multiplicity at least $\dim \rho$. Since $Q\in \ccal'$ and $\ccal'$ is closed under divisibility, we can and will assume that $\rho$ does not factor through $\pi_q(\Gamma)$ for some proper divisor $q$ of $Q$. Hence by Proposition~\ref{p:HighMultiplicity} we have that either the lift of $\rho$ to a representation of $\wh{\Gamma}$, the closure of $\Gamma$ in $\prod_{p\nmid q_0} \GL_n(\bbz_p)$, is in a finite set $\Sigma$ or $\dim \rho\ge |\pi_Q(\Gamma)|^{\vare_0}$ where $\vare_0$ is a positive number depending only on $\Gamma$. As we are seeking a uniform upper bound for $\lambda(Q)$, without loss of generality we can and will assume that the lift of $\rho$ is not in $\Sigma$. And so we have that multiplicity of the eigenvalue $\lambda(Q)$ of $T_Q$ is at least $|\pi_Q(\Gamma)|^{\vare_0}$. Therefore for any positive integer $l$ we have
\[
|\pi_Q(\Gamma)|^{\vare_0}\lambda(Q)^{2l}\le {\rm Tr}(T_Q^{2l})=|\pi_Q(\Gamma)|\|\pcal_{\pi_Q(\Omega)}^{(l)}\|_2^2. 
\]     
So it is enough to show 
\be\label{e:NeededL2UpperBound}
\|\pcal_{\pi_Q(\Omega)}^{(l)}\|_2\le |\pi_Q(\Gamma)|^{-\frac{1}{2}+\frac{\vare_0}{4}},
\ee
for some $l\ll_{\Omega} \log Q$ where the implied constant depends only on $\Omega$; it is worth pointing out that we can and will assume that $Q$ is arbitrarily large depending only on $\Omega$. In particular, we can and will assume {\em the Largeness of level $Q$-Approximate Subgroups} in the following sense:

{\bf The LAS Assumption:} there is $c_2>0$ such that, if $Q>1/c_2$ and $\Pfr_Q(c_2,A,l)$ holds, then $|\pi_Q(A)| \ge |\pi_Q(\Gamma)|^{1-\vare_0/8}$.

On the other hand, the first $O_{\Omega}(\log Q)$ steps of the random-walk with respect to the probability measure $\pcal_{\pi_Q(\Omega)}$ can be identified with the random-walk with respect to the probability measure $\pcal_{\Omega}$. And the latter is the random-walk on a regular infinite tree. Hence by the Kesten bound~\cite[Theorem 3]{Kes} we have
\[
\|\pcal_{\Omega}^{(l)}\|_2^2=\pcal_{\Omega}^{(2l)}(I)\le \left(\frac{2|\overline{\Omega}|-1}{|\overline{\Omega}|^2} \right)^l=e^{-c_1l},
\]
where $c_1$ is a positive number which depends only on $\Omega$. Hence for some positive integer $l_0=\Theta_{\Omega}(\log Q)$ where the implied constants depend only on $\Omega$ and for some positive number $c_3$ depending on $\Omega$ we have
\be\label{e:InitialFlatness}
\|\pcal_{\pi_Q(\Omega)}^{(l_0)}\|_2=\|\pcal_{\Omega}^{(l_0)}\|_2\le |\pi_Q(\Gamma)|^{-c_3}.
\ee
Equation (\ref{e:InitialFlatness}) can be considered as {\em an initial $L^2$-flatness}. 

An important result of Bourgain and Gamburd roughly indicates that the $L^2$-norm of a probability measure should substantially drop after a convolution by itself unless an approximate subgroup has a lot of mass~\cite{BG1} (see also \cite[Lemma 15]{Var}). This result is proved using Tao's non-commutative version of Balog-Szemer\'{e}di-Gowers theorem~\cite{Tao}. Here is a precise formulation of their result:

\begin{prop}[Bourgain-Gamburd's $L^2$-flattening.]\label{BourgainGamurd}
Let $\mu$ be a symmetric probability measure with finite support on a group $G$. Let $K$ be a number more than $2$. If $\|\mu\ast\mu\|_2\ge \frac{1}{K} \|\mu\|_2$, then there is a finite symmetric subset $A$ of $G$ which satisfies the following properties:
\begin{itemize}
	\item (Size of $A$ is comparable with $\|\mu\|_2^{-2}$) 
	$
	K^{-R}\|\mu\|_2^{-2}\le |A| \le K^{R} \|\mu\|_2^{-2};
	$
	\item (Getting an approximate subgroup)   
	$
	|A\cdot A\cdot A|\le K^{R}|A|;
	$	
	\item (Almost equidistribution on $A$) 
	$
	\min_{a\in A}(\mu\ast \mu)(a)\ge K^{-R}|A|^{-1},
	$
\end{itemize}	
	where $R$ is an absolute positive number.
\end{prop} 
For a small (depending only on $\Omega$ and specified later) positive parameter $\delta$, 
consider the following pseudocode: 

\begin{algorithm}
	\caption{$L^2$-flattening process}
	\label{al:flattening}
	\begin{algorithmic}
		\State Give $Q\in \ccal'$, $\delta\in (0,1)$;
		\State $l_0:=\lfloor\log Q/2\log C(\Omega)\rfloor$; (as in Equation~\ref{e:InitialFlatness}) 
		\State $\mu_0:=\pcal_{\pi_Q(\Omega)}^{(l_0)}$;
		\State $i:=0$;
		\While{($\|\mu_i\|_2 >|\pi_Q(\Gamma)|^{-1/2+\vare_0/4}$ and $\|\mu_i\ast \mu_i\|_2\le \|\mu_i\|_2^{1+\delta}$) or $(2^il_0<(1/\delta)\log Q)$}
		\State $\mu_{i+1}:=\mu_i \ast \mu_i$;
		\State $i:=i+1$;

		\EndWhile
	\end{algorithmic}
\end{algorithm}
where  $C(\Omega):=\max_{s\in \Omega}\|s\|$, where $\|s\|$ is the operator norm of $s:\bbr^n\times \prod_{p|q_0} \bbq_p^n \rightarrow \bbr^n\times \prod_{p|q_0} \bbq_p^n$ with $\|(v_{\infty},v_p)_{p|q_0}\|:=\max\{\|v_{\infty}\|,\|v_p\|_p\}_{p|q_0}$ and the sup norm on $\bbr^n$ and $\bbq_p^n$.

{\bf Claim 1.} For any $0<\delta<1$ and $Q\in\ccal'$, the number $i_0(\delta,Q)$ of steps in  Algorithm~\ref{al:flattening} is at most $\frac{\log(2/c_3)}{\log(1+\delta)}+\log(\frac{2\log C(\Omega)}{\delta})$ where $c_3$ is given in Equation (\ref{e:InitialFlatness}).

\begin{proof}[Proof of Claim 1] By the Young inequality we have $\|\mu_{i+1}\|_2\le \|\mu_i\|_2$ for any non-negative integer $i$. Hence by the initial flatness, Equation (\ref{e:InitialFlatness}), we have $\|\mu_i\|_2\le \|\mu_0\|_2\le |\pi_Q(\Gamma)|^{-c_3}$ for any non-negative integer $i$. Let $j_0$ be the smallest integer such that $2^{j_0}l_0\ge (1/\delta)\log Q$. Then for any $j_0\le i<i_0(\delta,Q)$ we have 
	$
	\|\mu_{i+1}\|_2\le \|\mu_i\|_2^{(1+\delta)}.
	$
And so 
\[
\|\mu_{i_0(\delta,Q)}\|_2\le \|\mu_{j_0}\|_2^{(1+\delta)^{i_0-j_0}}.
\]
If $i_0(\delta,Q)>\frac{\log(2/c_3)}{\log(1+\delta)}+\log(\frac{2\log C(\Omega)}{\delta})$, then $i_0-j_0>\frac{\log(2/c_3)}{\log(1+\delta)}$ which implies $\|\mu_{i_0-1}\|_2 \le|\pi_Q(\Gamma)|^{-1}$. And this contradicts $\|\mu_{i_0-1}\|_2>|\pi_Q(\Gamma)|^{-1/2+\vare_0/4}$.
\end{proof}

{\bf Claim 2.} There is a positive number $\delta:=\delta(\Omega)$ depending only on $\Omega$, such that, for any $Q\in \ccal'$ which is at least $1/\delta$, $\|\mu_{i_0(\delta,Q)}\|_2\le |\pi_Q(\Gamma)|^{-1/2+\vare_0/4}$, where $i_0(\delta,Q)$ is given in Claim 1.

\begin{proof}[Proof of Claim 2] We proceed by contradiction. So to the contrary assumption for any positive number $\delta$ there is $Q(\delta)\in \ccal'$ such that 
\begin{align}
\label{e:QGoesToInfinity} Q(\delta)>1/\delta; & \text{ in particular } \lim_{\delta\rightarrow 0}Q(\delta)=\infty, \\
\label{e:LargeL2Norm} \|\mu_{i_0(\delta,Q(\delta))}\|_2&>|\pi_{Q(\delta)}(\Gamma)|^{-1/2+\vare_0/4}, \text{ and}\\
\label{e:L2NonFlatness} \|\mu_{i_0(\delta,Q(\delta))}\ast \mu_{i_0(\delta,Q(\delta))}\|_2&\ge \|\mu_{i_0(\delta,Q(\delta))}\|_2^{1+\delta}. 
\end{align}	

 Hence by (\ref{e:L2NonFlatness}) and  Proposition~\ref{BourgainGamurd} (for $K=\|\mu_{i_0(\delta,Q(\delta))}\|_2^{-\delta}$) there is a symmetric subset $\overline{A}(\delta)$ of $\pi_{Q(\delta)}(\Gamma)$ such that
 \begin{align}
 \label{e:SizeContorl}  \|\mu_{i_0(\delta,Q(\delta))}\|_2^{-2+R\delta} &\le
 |\overline{A}(\delta)|\le\|\mu_{i_0(\delta,Q(\delta))}\|_2^{-2-R\delta},\\
 \label{e:ApproximateSubgroup}  |\overline{A}(\delta)\cdot \overline{A}(\delta) \cdot \overline{A}(\delta)|&\le \|\mu_{i_0(\delta)}\|_2^{-R\delta}|\overline{A}(\delta)|, \text{ and}\\
 \label{e:LargeMass}  \mu_{i_0(\delta,Q(\delta))+1}(\overline{A}(\delta))&\ge \|\mu_{i_0(\delta,Q(\delta))}\|_2^{R\delta}.
 \end{align}
Hence there is a function $\delta'(\delta)$ of $\delta$ such that 
\begin{align}
\label{e:ParameterGoesToZero}  &\lim_{\delta\rightarrow 0}\delta'(\delta)=0 \text{ and } \delta'(\delta)\ge \delta,&\\
\label{e:TripleProduct} \text{(because of Inequalities (\ref{e:SizeContorl}) and (\ref{e:ApproximateSubgroup}),) }\hspace{1cm}& |\overline{A}(\delta)\cdot \overline{A}(\delta) \cdot \overline{A}(\delta)|\le |\overline{A}(\delta)|^{1+\delta'(\delta)}&\\
\label{e:LargeMassReformulate} \text{(because of Inequalities (\ref{e:LargeL2Norm}) and (\ref{e:LargeMass}),) }\hspace{1cm}&  \mu_{i_0(\delta,Q(\delta))+1}(\overline{A}(\delta))\ge Q(\delta)^{-\delta'(\delta)}.&
\end{align}

We notice that $\mu_j=\pcal_{\pi_{Q(\delta)}(\Omega)}^{(2^jl_0(\delta))}$, and let $A(\delta):=\pi_{Q(\delta)}^{-1}(\overline{A}(\delta))\cap {\rm supp}(\pcal_{\Omega}^{(l(\delta))})$, where $l(\delta):=2^{i_0(\delta,Q(\delta))+1}l_0(\delta)$; then we have
\be
\label{e:Symmetry} 
A(\delta) \text{ is a finite symmetric subset of } \Gamma
\text{ because $\overline{A}(\delta)$ and $\pcal_{\Omega}$ are symmetric,} 
\ee
\be
\label{e:LowerBoundForTheNumberOfSteps} 
l(\delta)\ge (1/\delta)\log Q(\delta)\ge (1/\delta'(\delta))\log Q(\delta)
\ee
because of the second condition of the loop in Algorithm~\ref{al:flattening} and the inequality in (\ref{e:ParameterGoesToZero}),
\begin{align}
\label{e:LargeMassOfTheLift}
\pcal_{\Omega(\delta)}^{(l(\delta))}(A(\delta))=\mu_{i_0(\delta)+1}(\overline{A}(\delta))\ge Q(\delta)^{-\delta'(\delta)}&
\text{ by Inequality (\ref{e:LargeMassReformulate}),} 
\\
\label{e:ApproximateStructure}
|\pi_{Q(\delta)}(A(\delta)\cdot A(\delta)\cdot A(\delta))|\le  |\pi_{Q(\delta)}(A(\delta))|^{1+\delta'(\delta)}&
\text{ by Inequality (\ref{e:TripleProduct}).}
\end{align}
Statements (\ref{e:Symmetry}), (\ref{e:LowerBoundForTheNumberOfSteps}), (\ref{e:LargeMassOfTheLift}), and (\ref{e:ApproximateStructure})
imply that $\Pfr_{Q(\delta)}(\delta'(\delta),A(\delta),l(\delta))$ holds for any $\delta$. Now suppose $\delta$ is small enough so that $\delta'(\delta)$ is less than $c_2$ where $c_2$ is given by the LAS Assumption (this can be done thanks to (\ref{e:ParameterGoesToZero})). Then $Q(\delta)>1/\delta\ge 1/\delta'(\delta)\ge 1/c_2$, and so by the LAS Assumption we have
\be
\label{e:SizeBound}
|\pi_{Q(\delta)}(A(\delta))|=|\overline{A}(\delta)|\ge |\pi_{Q(\delta)}(\Gamma)|^{1-\vare_0/8}.
\ee 
On the other hand, by Inequalities (\ref{e:SizeContorl}) and (\ref{e:LargeL2Norm}), we have
\[
|\overline{A}(\delta)|\le \|\mu_{i_0(\delta,Q(\delta))}\|_2^{-2-R\delta}\le |\pi_{Q(\delta)}(\Gamma)|^{1+R\delta/2-\vare_0/2-R\delta \vare_0/4}.
\]
Hence, if $\delta$ is small enough, we get $|\overline{A}(\delta)|\le |\pi_{Q(\delta)}(\Gamma)|^{1+R\delta/2-\vare_0/2-R\delta \vare_0/4}\le |\pi_{Q(\delta)}(\Gamma)|^{1-\vare_0/4}$, which contradicts (\ref{e:SizeBound}).
\end{proof}
{\em (Going back to the proof of Proposition~\ref{p:ReductionToApproximateSubgroup})} By Claim 1 and Claim 2, for $Q\gg_{\Omega} 1$ in $\ccal'$ we have $\|\pcal_{\pi_Q(\Omega)}^{(l)}\|_2\le |\pi_Q(\Gamma)|^{-1/2+\vare_0/4}$ where $l:=2^{i_0(\delta(\Omega),Q)}l_0\ll_{\Omega} \log Q$. And this gives us the needed $L^2$-norm upper bound (see Inequality (\ref{e:NeededL2UpperBound})) for $Q\gg_{\Omega} 1$. The rest are only finitely many $Q$'s for which we do have spectral gap.
\end{proof}
\subsection{Reduction to bounded generation and abelian unipotent radical.}\label{ss:ReductionToBG}
The main goal of this section is to deduce Theorem~\ref{t:ApproximateSubgroup} from the following.
\begin{thm}[Bounded generation]\label{t:BoundedGeneration}
In the setting of Section~\ref{ss:InitialReductions}, suppose $\bbu$ is abelian. For any positive integer $q$, let $\Gamma[q]:=\Gamma\cap \GL_{n_0}(\bbz[1/q_0])[q]$, where 
\[
\GL_{n_0}(\bbz[1/q_0])[q]:=\ker(\GL_{n_0}(\bbz[1/q_0])\xrightarrow{\pi_q}\GL_{n_0}(\bbz[1/q_0]/q\bbz[1/q_0])).
\]
 Then for any $0<\vare\ll_{\Omega} 1$ there are $0<\delta\ll_{\Omega,\vare} 1$ and positive integer $C\gg_{\Omega,\vare} 1$ such that for a finite symmetric subset $A$ of $\Gamma$
\begin{center}
$\pcal_{\Omega}^{(l)}(A)>Q^{-\delta}$ for some $l>(1/\delta)\log Q$ implies $\pi_Q(\Gamma[q])\subseteq \prod_C\pi_Q(A)$ for some $q|Q$ and $q\le Q^{\vare}$
\end{center}
if $Q\in\ccal$ and $Q^{\vare^{O_{\bbg}(1)}}\gg_{\Omega} 1$. 
\end{thm}
{\bf The idea of the proof of Theorem~\ref{t:ApproximateSubgroup} modulo Theorem~\ref{t:BoundedGeneration}.} It is not clear to the author if a bounded generation result is true when the unipotent radical $\bbu$ is not abelian (specially since a perfect group might have a non-trivial center.). So to prove the {\em largeness of a generic level-$Q$ approximate subgroup $A$ of $\Gamma$} with Zariki-closure $\bbg:=\bbg_s\ltimes \bbu$, first we go to the quotient $\bbg_s\ltimes \bbu/[\bbu,\bbu]$. Let $\rho_2:\bbg\rightarrow \bbg_s\ltimes \bbu/[\bbu,\bbu]$ be the quotient map. For simplicity let us assume $\rho_2$ commutes with the congruence maps $\pi_Q$. This assumption together with the way the generating set $\Omega$ is chosen (recall that we have $\Omega=\overline{\Omega}\cup\overline{\Omega}^{-1}$ and $\rho_2(\overline{\Omega})$ freely generates a free group) implies that $\rho_2(A)$ is {\em a generic level-$Q$ approximate subgroup of $\rho_2(\Gamma)$}. Now by the bounded generation result for $\rho_2(\Gamma)$, we get a {\em large} congruence subgroup of $\rho_2(\Gamma)$ modulo $Q$ in a {\em controlled} number of steps. Next we use the techniques of studying a nilpotent group in order to translate our information on $\bbu/[\bbu,\bbu]$ to some data on $\bbu$. For a unipotent group $\bbu$ (and in general any nilpotent group), it is a useful method to consider  
\[
L_{\bbu}:=\gamma_1(\bbu)/\gamma_2(\bbu)\oplus \gamma_2(\bbu)/\gamma_3(\bbu) \oplus \cdots \oplus \gamma_c(\bbu)/\gamma_{c+1}(\bbu),
\]  
where $\gamma_i(\bbu)$ is the $i$-th lower central series; that means $\gamma_1(\bbu):=\bbu$ and $\gamma_{i+1}:=[\gamma_i(\bbu),\bbu]$, and $\gamma_{c+1}(\bbu)=\{1\}$. It is well-known that $[g_i\gamma_{i+1}(\bbu),g_j\gamma_j(\bbu)]:=g_i^{-1}g_j^{-1}g_ig_j \gamma_{i+j+1}(\bbu)$ for $g_i\in \gamma_i(\bbu)$ and $g_j\in \gamma_j(\bbu)$ defines a graded Lie ring structure on $L_{\bbu}$. And $L_{\bbu}$ is generated by its degree 1 elements, i.e. $\gamma_1(\bbu)/\gamma_2(\bbu)$, as a Lie ring. So we get an onto multi-linear map
\[
\phi_k:\underbrace{\gamma_1(\bbu)/\gamma_2(\bbu)\times \cdots \times \gamma_1(\bbu)/\gamma_2(\bbu)}_\text{$k$-times}\rightarrow \gamma_k(\bbu)/\gamma_{k+1}(\bbu), \hspace{1cm}
\phi_k(x_1,\ldots,x_k):=[x_1,\ldots,x_k],
\] 
where $[x_1,\ldots,x_k]:=[[x_1,\ldots,x_{k-1}],x_k]$. Having a set $X$ which maps onto a set of representatives for a {\em large} congruence subgroup modulo $Q$ in $\gamma_1(\bbu)/\gamma_2(\bbu)$, using the above multi-linear maps $\phi_k$, we can get a {\em subset} of $\bbu$ which is {\em large} modulo $Q$. {\em Notice that at this step we can get a control only on the {\em size} of the product set $\pi_Q(\prod_{O_{\bbu}(1)} X)$. Though we do get a set of representatives of a {\em large} congruence subset of the associated graded Lie ring, we do not get such a structural understanding of the product set itself.} Altogether we get that $|\prod_C \pi_Q(A)|$ is {\em large}. On the other hand, since $|\pi_Q(A)\cdot \pi_Q(A)\cdot \pi_Q(A)|\le |\pi_Q(A)|^{1+\delta}$, by the Ruzsa inequality (see~\cite[Lemma 2.2]{Hel1}), we get 
\[
\textstyle |\prod_C\pi_Q(A)|\le |\pi_Q(A)| \left(\frac{|\prod_3\pi_Q(A)|}{|\pi_Q(A)|}\right)^{C-2} \le |\pi_Q(A)|^{1+(C-2)\delta},
\]        
which helps us to finish the proof.

\begin{proof}[Theorem~\ref{t:BoundedGeneration} implies Theorem~\ref{t:ApproximateSubgroup}.] 
By the contrary assumption, there is $\varepsilon_0>0$ such that for any $\delta>0$ there are a finite symmetric subset $A_{\delta}$, a positive integer $l_{\delta}$ and ${Q_{\delta}}\in\ccal$ such that $\Pfr_{Q_{\delta}}(\delta,A_{\delta},l_{\delta})$ holds and at the same time $|\pi_{Q_{\delta}}(A_{\delta})|< |\pi_{Q_{\delta}}(\Gamma)|^{1-\varepsilon_0}$. Since $l_{\delta}>\log Q_{\delta}/\delta$, we have $\lim_{\delta\to 0} l_{\delta}=\infty$. Now we claim that $Q_{\delta}\rightarrow \infty$ as $\delta\rightarrow 0$. If not, for infinitely many $\delta$ we have $Q_{\delta}=Q$, and $\pi_Q(A_{\delta})=\overline{A}$ is independent of $\delta$. And so
\[
\lim_{\delta\rightarrow 0} \pcal_{\pi_Q(\Omega)}^{(l_{\delta})}(\pi_Q(A_{\delta}))=\frac{|\overline{A}|}{|\pi_Q(\Gamma)|}<\frac{1}{|\pi_Q(\Gamma)|^{\vare_0}}.
\]
On the other hand, we have
\[
\pcal_{\pi_Q(\Omega)}^{(l_{\delta})}(\pi_Q(A_{\delta}))\ge \pcal_{\Omega}^{(l_{\delta})}(A_{\delta})> Q_{\delta}^{-\delta}=Q^{-\delta}.
\]
Hence as $\delta\rightarrow 0$, we get $|\pi_Q(\Gamma)|^{-\vare_0}>1$, which is a contradiction. 

Let $\{k_p\}_{p\nmid q_0}$ be such that $P_p=\bbg(\bbq_p)\cap\GL_{n_0}(\bbz_p)[p^{k_p}]$ where $\{P_p\}_p$ is as in Section~\ref{ss:InitialReductions}. Let $\rho_2:\bbg\rightarrow \bbg_s\ltimes \bbu/[\bbu,\bbu]$ be the quotient map. Let us recall that, by the discussion at the end of Section~\ref{ss:AlgHomoAlmostCommMod}, we can and will assume $\rho_2(P_p)\subseteq \GL_{n_2}(\bbz_p)$.

Let $q_0'=\prod_{p\nmid q_0} p^{k'_p}$ be a positive integer given by Lemma~\ref{l:AlmostCommute}; that means for any sequence $\{m_p\}_{p\nmid q_0}\in \bigoplus_{p\nmid q_0} \bbz$ such that $m_p\ge k_p+k_p'$ we have
\be\label{e:CommutingDiscrepency}
\prod_{p\nmid q_0} \pi_{p^{m_p+k_p'}}(P_p)/\pi_{p^{m_p+k_p'}}
([\bbu,\bbu](\bbq_p)\cap P_p) \twoheadrightarrow
	\prod_{p\nmid q_0} \pi_{p^{m_p}}(\rho_2(P_p)) \twoheadrightarrow
	\prod_{p\nmid q_0} \pi_{p^{m_p-k'_p}}(P_p)/\pi_{p^{m_p-k'_p}}([\bbu,\bbu](\bbq_p)\cap P_p).
\ee
In order to use (\ref{e:CommutingDiscrepency}), we modify $Q_{\delta}$ so that for any prime $p$ either $v_p(Q_{\delta})\ge k_p'$ or $v_p(Q_{\delta})=0$. To this end, it is enough to show that we are allowed to change $Q_{\delta}$ to $Q_{\delta}/q_{\delta}'$ where $q_{\delta}':=\prod_{p\nmid q_0, 0<v_p(Q_{\delta})<2k_p'+k_p} p^{v_p(Q_{\delta})}$.

\begin{lem}~\label{l:UptoConstant}
Let $A_{\delta}$, $Q_{\delta}$, and $l_{\delta}$ be as above. Suppose  $q_{\delta}'|Q_{\delta}$ and $q_{\delta}'<C'$ (for some constant $C'$). Then $\Pfr_{Q_{\delta}/q_{\delta}'}(\Theta_{C',\Omega}(\delta),A_{\delta},l_{\delta})$ holds  and $|\pi_{Q_{\delta}/q_{\delta}'}(A_{\delta})|\le |\pi_{Q_{\delta}/q_{\delta}'}(\Gamma)|^{1-\vare_0/2}$ for $0<\delta\ll_{C'} 1$.
\end{lem}
\begin{proof}
Since $\lim_{\delta\to 0}Q_{\delta}=\infty$, for $C''>1$ and small enough $\delta$ we have $Q_{\delta}/q_{\delta}'>Q_{\delta}^{1/C''}$. And so
\[ \pcal^{(l_{\delta})}(A_{\delta})>Q_{\delta}^{-\delta}>(Q_{\delta}/q_{\delta}')^{-C''\delta},\h\text{and}\hspace{1cm} 
 l_{\delta}>\frac{\log Q_{\delta}}{\delta}>\frac{\log (Q_{\delta}/q_{\delta}')}{C''\delta}.
\]
To show that $\Pfr_{Q_{\delta}/q_{\delta}'}(C''\delta,A_{\delta},l_{\delta})$ holds, it is enough to prove that we have $|\pi_{Q_{\delta}/q_{\delta}'}(\prod_3 A_{\delta})|\le |\pi_{Q_{\delta}/q_{\delta}'}(A_{\delta})|^{1+C''\delta}$
for large enough $C''$.

We notice that, if $c$ is small enough (depending on $\Omega$), $\pi_{Q_{\delta}/q_{\delta}'}$ induces an injection on the ball of radius $c\log(Q_{\delta}/q_{\delta}')$ (with respect to the word metric). By \cite[Remark, page 1060]{BG3} (see \cite[Lemma 8]{SGsemisimple}), we have $\pcal^{(\lfloor 2c\log(Q_{\delta}/q_{\delta}')\rfloor)}(A_{\delta}\cdot A_{\delta})> Q_{\delta}^{-2\delta}$. Hence by the Kesten bound (see Inequality (\ref{e:KestenBound})) for small enough $\delta$ we have
\begin{align}
\label{e:Size} |\pi_{Q_{\delta}/q_{\delta}'}(A_{\delta}\cdot A_{\delta})|&=
|A_{\delta}\cdot A_{\delta}\cap B_{\lfloor 2c\log(Q_{\delta}/q_{\delta}')\rfloor}|\\
\notag &\ge 
\pcal^{(\lfloor 2c\log(Q_{\delta}/q_{\delta}')\rfloor)}(A_{\delta}\cdot A_{\delta})/\|\pcal^{(\lfloor 2c\log(Q_{\delta}/q_{\delta}')\rfloor)}\|_{\infty}\\
\notag &\ge
Q_{\delta}^{-2\delta}\cdot (2c\log(Q_{\delta}/q_{\delta}'))^{\Theta_{\Omega}(1)}\ge
Q_{\delta}^{\Theta_{C',\Omega}(1)},
\end{align}
which implies that $|\pi_{Q_{\delta}/q_{\delta}'}(A_{\delta})|\ge Q_{\delta}^{\Theta_{C',\Omega}(1)}$.

Since $|\pi_{Q_{\delta}}(\prod_3 A_{\delta})|\le |\prod_{Q_{\delta}}(A_{\delta})|^{1+\delta}$, by the Ruzsa inequality (see \cite[Lemma 2.2]{Hel1}) and \cite[Lemma 7.4]{Hel2} we have
\[
|\pi_{Q_{\delta}}(A_{\delta})|^{3\delta}\ge \frac{|\pi_{Q_{\delta}}(\prod_5 A_{\delta})|}{|\pi_{Q_{\delta}}(A_{\delta})|}\ge \frac{|\pi_{Q_{\delta}/q_{\delta}'}(\prod_3 A_{\delta})|}{|\pi_{Q_{\delta}/q_{\delta}'}(A_{\delta})|},
\]
which implies that $|\pi_{Q_{\delta}/q_{\delta}'}(\prod_3 A_{\delta})|/|\pi_{Q_{\delta}/q_{\delta}'}(A_{\delta})|\le |\pi_{Q_{\delta}/q_{\delta}'}(A_{\delta})|^{C''\delta}$ for large enough $C''$. Since $Q_{\delta}$ goes to infinity, it is clear that for small enough $\delta$ we have $|\pi_{Q_{\delta}/q_{\delta}'}(A_{\delta})|\le |\pi_{Q_{\delta}/q_{\delta}'}(\Gamma)|^{1-\vare_0/2}$. 
\end{proof}

By Lemma~\ref{l:UptoConstant} we can and will modify $\{Q_{\delta}\}_{\delta}$ and assume that, if $k_p+k_p'>0$, then either $v_p(Q_{\delta})=0$ for any $\delta$ or $k_p+2k_p'\le \min_{\delta}v_p(Q_{\delta})<\sup_{\delta} v_p(Q_{\delta})=\infty$. In particular, without loss of generality we can and will assume that $q_0'^2|Q_{\delta}$ for any $\delta$.

On the other hand, by Corollary~\ref{c:InitialReductions}, we have that $\ker(\rho_2)\cap \Gamma=\{1\}$. So $\pcal_{\rho_2(\Omega)}^{(l_{\delta})}(\rho_2(A_{\delta}))=\pcal_{\Omega}^{(l_{\delta})}(A_{\delta})>Q_{\delta}^{-\delta}>(q_0'Q_{\delta})^{-\Theta_{\Omega}(\delta)}$ for suitable constant and small enough $\delta$. Hence by Theorem~\ref{t:BoundedGeneration} for any $0<\vare'\ll_{\Omega} 1$ there is a positive integer $C=O_{\Omega,\vare'}(1)$ such that for any $0<\delta\ll_{\vare',\Omega} 1$ we have
\be\label{e:BGAbelian}
\textstyle
\pi_{q_0'Q_{\delta}}(\rho_2(\Gamma)[q_{\delta}])\subseteq \prod_C \pi_{q_0'Q_{\delta}}(\rho_2(A_{\delta}))
\ee
for some $q_{\delta}| q_0'Q_{\delta}$ and $q_{\delta}\le (q_0'Q_{\delta})^{\vare'}\le Q_{\delta}^{\Theta(\vare')}$.

On the other hand we have $\pi_{q_0'Q_{\delta}}(\prod_{p\nmid q_0}\rho_2(U_p))\subseteq \pi_{q_0'Q_{\delta}}(\rho_2(\Gamma))$ (recall that $U_p:=\bbu(\bbq_p)\cap\GL_{n_0}(\bbz_p)[p^{k_p}]$. Let $\overline{U}_p:=\GL_{n_2}(\bbz_p)\cap \bbu(\bbq_p)$); and so we have that $\overline{U}_p=U_p$ for large enough $p$, and $U_p$ is an open subgroup of $\overline{U}_p$.  By Lemma~\ref{l:UptoConstant}, we can and will assume that $v_p(Q_{\delta})$ are either zero or large enough depending on the index of $U_p$ in $\overline{U}_p$ for all the primes $p$. And so by enlarging $q_{\delta}$ by a multiplicative constant we can and will assume that 
\[
\textstyle
\pi_{q_0'Q_{\delta}}(\rho_2(\prod_{p\nmid q_0} \overline{U}_p)^{q_{\delta}})\subseteq \pi_{q_0'Q_{\delta}}(\rho_2(\Gamma)[q_{\delta}])\subseteq \prod_C \pi_{q_0'Q_{\delta}}(\rho_2(A_{\delta})),
\]
where $\rho_2(\prod_{p\nmid q_0} \overline{U}_p)^{q_{\delta}}:=\{x^{q_{\delta}}|\h x\in \rho_2(\prod_{p\nmid q_0} \overline{U}_p)\}$.
So there is $B_{\delta}\subseteq \prod_C A_{\delta}$ such that 
\be\label{e:LargeAbelianUnipotent}
\pi_{q_0'Q_{\delta}}(\rho_2(B_{\delta}))=\pi_{q_0'Q_{\delta}}(\rho_2(\prod_{p\nmid q_0} \overline{U}_p)^{q_{\delta}}).
\ee
\begin{lem}\label{l:UnipotentPropagation}
Let $\bbu$ be a unipotent $\bbq$-group with a given $\bbq$-embedding in $\mathbb{GL}_n$. For any prime $p$, let $\overline{U}_p:=\bbu(\bbq_p)\cap \GL_n(\bbz_p)$. Let $\phi:\bbu\rightarrow \bbu/[\bbu,\bbu]$ be the quotient map. Fix a $\bbq$-basis $\{\ebf_i\}$ of $(\bbu/[\bbu,\bbu])(\bbq)$ such that $\phi(\overline{U}_p)\subseteq \sum_i \bbz_p \ebf_i$ for any prime $p$. Let $q|Q$ be two positive integers. Suppose $X$ is a finite subset of $\prod_p \overline{U}_p$ such $\pi_{Q}(\phi(X))=\pi_Q(q\sum_i\bbz \ebf_i)$. Then 
\[
|\textstyle \prod_C \pi_Q(X)|\ge \frac{|\pi_Q(\prod_p \overline{U}_p))|}{Cq^{C}},
\] 
where $C$ depends on the dimension of $\bbu$ and the choice of embeddings.  
\end{lem}
\begin{proof}
Let $\gamma_i(\bbu)$ be the $i$-th lower central series; that means $\gamma_1(\bbu):=\bbu$ and $\gamma_{i+1}(\bbu):=[\gamma_i(\bbu),\bbu]$. It is well-known that the long commutators $(g_1,\ldots,g_k):=((g_1,\ldots,g_{k-1}),g_k)$, where $(g_1,g_2):=g_1^{-1}g_2^{-1}g_1g_2$, induce multilinear maps $f_k$ from $\bbu/[\bbu,\bbu]\times \cdots \times \bbu/[\bbu,\bbu]$ ($k$ times) to $\gamma_k(\bbu)/\gamma_{k+1}(\bbu)$, and the $\bbq$-span of $\{f_k(\ebf_{i_1},\ldots,\ebf_{i_k})\}$ is $(\gamma_k(\bbu)/\gamma_{k+1}(\bbu))(\bbq)$. For any $k$, let $\phi_k:\gamma_k(\bbu)\rightarrow \gamma_k(\bbu)/\gamma_{k+1}(\bbu)$. Then for any $k$ we have
\[
\pi_Q(\phi_k(\textstyle\prod_{4k}X\cap \gamma_k(\bbu)(\prod_p \bbq_p))\supseteq \cup_{i_1,\ldots,i_k} \pi_Q(q^k \bbz f_k(\ebf_{i_1},\ldots,\ebf_{i_k})),
\] 
and so 
\[
\pi_Q(\phi_k(\textstyle\prod_{4k\dim \bbu}X\cap \gamma_k(\bbu)(\prod_p \bbq_p)))\supseteq
\pi_Q(q^k \sum_{i_1,\ldots,i_k}\bbz f_k(\ebf_{i_1},\ldots,\ebf_{i_k}))
\]
In particular, we have
\be~\label{e:UnipotentLevelk}
|\pi_Q(\phi_k(\textstyle\prod_{4k\dim \bbu}X\cap \gamma_k(\bbu)(\prod_p \bbq_p)))|\ge (Q/q^k)^{\dim \gamma_k(\bbu)-\dim \gamma_{k+1}(\bbu)}.
\ee
Since $\phi_k$'s are regular morphisms defined over $\bbq$, there is a positive integer $q'_1$ depending on the embedding of $\bbu$ and the choice of $\{\ebf_i\}$) such that for any prime power $p^m$ we have
\[
\phi_k\left(\gamma_k(\bbu)(\prod_p \bbq_p)\cap \prod_p \overline{U}_p[q_1'Q]\right)\subseteq \prod_p (Q\sum_{i_1,\ldots,i_k} \bbz_p f_k(\ebf_{i_1},\ldots,\ebf_{i_k})).
\] 
And so $\phi_k$ induces homomorphism $\overline{\phi}_k:\pi_{q'_1Q}(\gamma_k(\bbu)(\prod_p \bbq_p)\cap \prod_p \overline{U}_p)\rightarrow \pi_Q(\sum_{i_1,\ldots,i_k}\bbz f_k(\ebf_{i_1},\ldots,\ebf_{i_k}))$, and we have $\overline{\phi}_k\circ \pi_{q_1'Q}=\pi_Q\circ \phi_k$.
Thus by (\ref{e:UnipotentLevelk}) we have 
\be~\label{e:UnipLowerBound}
 |\pi_{q'_1Q}({\textstyle\prod_{4\dim \bbu^3} X})|\ge \prod_k |\pi_Q(\phi_k({\textstyle\prod_{4k\dim \bbu}X}\cap \gamma_k(\bbu)(\prod_p \bbq_p)))|\ge (Q/q^{\dim \bbu})^{\dim \bbu}.
\ee
Since $\bbu$ is a $\bbq$-unipotent, there is a positive integer $q'_2$ (depending on the embedding and $\bbu$) such that ${q'_2}^{-1}Q^{\dim \bbu}\le |\pi_Q(\prod_p \overline{U}_p)|\le q'_2Q^{\dim\bbu}$. And so
\[
|\pi_{Q}(\textstyle\prod_{4\dim \bbu^3} X)|\ge \frac{|\pi_Q(\prod_p \overline{U}_p)|}{{q'_2}^2 {q'_1}^{\dim\bbu}q^{\dim\bbu^2}};
\] 
and the claim follows.
\end{proof}
By~(\ref{e:LargeAbelianUnipotent}) there is $X_{\delta}\subseteq \prod_p U_p$ such that $\pi_{q_0'Q_{\delta}}(X_{\delta})=\pi_{q_0'Q_{\delta}}(B_{\delta})$, which implies that $\pi_{Q_{\delta}}(\rho_2(X_{\delta}))=\pi_{Q_{\delta}}(\rho_2(B_{\delta}))$. Hence by Lemma~\ref{l:UnipotentPropagation} we have 
\be\label{e:LargeUnipotentPart}
\textstyle
|\pi_{Q_\delta}(\prod_{O_{\Omega,\vare'}(1)} A_{\delta})\cap \pi_{Q_{\delta}}(\prod_{p\nmid q_0} \overline{U}_p)|\ge \frac{|\pi_{Q_{\delta}}(\prod_{p\nmid q_0} \overline{U}_p)|}{\Theta_{\dim \bbu}(1) q_{\delta}^{\Theta_{\bbu}(1)}}\ge \frac{|\pi_{Q_{\delta}}(\prod_{p\nmid q_0} \overline{U}_p)|}{ Q_{\delta}^{\Theta_{\dim\bbu}(\vare')}}\ge |\pi_{Q_{\delta}}(\prod_{p| Q_{\delta}} \overline{U}_p)|^{1-\Theta_{\bbu}(\vare')},
\ee
for small enough $\delta$ (depending on $\vare'$) as $\lim_{\delta\to 0} Q_{\delta}=\infty$. 

On the other hand, by~(\ref{e:BGAbelian}), we have 
\[
\textstyle
\pi_{Q_{\delta}}(\prod_{p\nmid q_0} K_p[q_{\delta}])\subseteq \prod_{C} \pi_{Q_{\delta}}(A_{\delta}).
\]
Therefore we have 
\be\label{e:LargeSemisimplePart}
|\pi_{Q_{\delta}}(\textstyle \prod_C A_{\delta})\cap \pi_{Q_{\delta}}(\prod_{p|Q_{\delta}}K_p)|\ge |\pi_{Q_{\delta}}(\prod_{p|Q_{\delta}} K_p)|^{1-\Theta_{\bbg_s}(\vare')}, 
\ee
for small enough $\delta$ (depending on $\vare'$) as $\lim_{\delta\to 0}Q_{\delta}=\infty$. 

By (\ref{e:LargeUnipotentPart}), (\ref{e:LargeSemisimplePart}), $\lim_{\delta\to 0}Q_{\delta}=\infty$, and the fact that $\prod_{p\nmid q_0} K_p\ltimes U_p$ is of finite index in $\wh{\Gamma}$, we have
\be\label{e:LowerBound}
|\pi_{Q_{\delta}}(\textstyle\prod_C A_{\delta})|\ge |\pi_{Q_{\delta}}(\Gamma)|^{1-\Theta_{\Omega}(\vare')},
\ee
for small enough $\delta$ (depending on $\vare'$). 

On the other hand, since $|\pi_{Q_{\delta}}(\prod_3A_{\delta})|\le |\pi_{Q_{\delta}}(A_{\delta})|^{1+\delta}$ by the Ruzsa inequality (see~\cite{Hel1}), we have
\be\label{e:UpperBound}
|\pi_{Q_{\delta}}(\textstyle\prod_C A_{\delta})|\le |\pi_{Q_{\delta}}(A_{\delta})|^{1+(C-2)\delta}\le |\pi_{Q_{\delta}}(\Gamma)|^{(1-\vare_0)(1+(C-2)\delta)}.
\ee
By (\ref{e:LowerBound}) and (\ref{e:UpperBound}), for any $\vare'$ and small enough $\delta$ (in particular it can approach zero), we have
\[
1-\Theta_{\Omega}(\vare')\le (1-\vare_0)(1+(C(\vare')-2)\delta),
\] 
which is a contradiction.
\end{proof}

\section{Super-approximation: bounded power of square-free integers case.}  
By Section~\ref{ss:ReductionToBG}, to prove Theorem~\ref{t:main} for $\ccal_{N}:=\{q\in \bbz^+|\h \forall p\nmid q_0, v_p(q)\le N\}$ it is enough to prove Theorem~\ref{t:BoundedGeneration} (Bounded generation) for $\ccal_{N}$. In particular, we can assume that $\bbu$ is abelian. Since I believe it is interesting to know if some of the auxiliary results are true for more general perfect groups, I do not assume $\bbu$ is abelian till towards the end of this section. As a result some extra lemmas are proved which are not necessary for the proof of Theorem~\ref{t:BoundedGeneration} for $\ccal_{N}$.\footnote{If the reader is only interested in the proof of Theorem~\ref{t:BoundedGeneration} for $\ccal_{N}$, s/he can skip Lemma \ref{l:ProperAdInvariant}, Lemma \ref{l:UnipotentModPSquare}, and Lemma \ref{l:CommutatorPGroup}.}
\subsection{The residue maps do not split.}
The main goal of this section is to show Lemma~\ref{l:NotSplittingModPSquare} which roughly asserts that $\pi_p$ does split under mild conditions.
\begin{lem}~\label{l:ProperAdInvariant}
Let $\bbg=\bbh\ltimes \bbu$ be a $\bbq$-group with a given $\bbq$-embedding into $(\GL_n)_{\bbq}$, where $\bbh$ is a connected semisimple $\bbq$-group and $\bbu$ is a unipotent $\bbq$-group. 
Let $\gcal$ (resp. $\cal$, $\ucal$) be the Zariski-closure of $\bbg$ (resp. $\bbh$, $\bbu$) in the $\bbz$-group scheme $(\GL_n)_{\bbz}$, and $\gfr=\Lie(\gcal)$, $\hfr=\Lie(\cal)$, and $\ufr=\Lie(\ucal)$. For large enough $p$, if $V\subsetneq \gfr(\f_p)$ is a $\gcal(\f_p)$-invariant proper subspace, then $V+[\ufr(\f_p),\ufr(\f_p)]$ is a proper subspace. Moreover if $\bbg$ is perfect, then $V+\zfr(\f_p)$ is also a proper subspace, where $\zfr=\Lie(Z(\gcal))$ is the Lie ring of the schematic center of $\gcal$.  
\end{lem}
\begin{proof}
Let $A_G:=\f_p[\Ad(\gcal(\f_p))]\subseteq {\rm End}_{\f_p}(\gfr(\f_p))$ be the $\f_p$-span of $\Ad(\gcal(\f_p))$, and $A_H:=\f_p[\Ad(\cal(\f_p))]$ be its subalgebra. Since $\bbh$ is connected semisimple, for large enough $p$, $\gfr(\f_p)$ is a faithful completely reducible $A_H$-module. Thus $A_H$ is a semisimple algebra, and so its Jacobson radical $J(A_H)$ is zero. Let $\afr$ be the ideal generated by $\{u-1|\h u\in \ucal(\f_p)\}$. Then we have the following short exact sequence of $A_G$-modules:
\[
0\rightarrow \afr \rightarrow A_G \rightarrow A_H\rightarrow 0.
\] 
Since $A_H$ is semisimple, $J(A_G)\subseteq \afr$. On the other hand, since $\bbu$ is a unipotent normal subgroup of $\bbg$, $\afr$ is a nilpotent ideal. So overall we have $\afr= J(A_G)$. 

Suppose to the contrary that there is a proper $A_G$-submodule $V$ of $\gfr(\f_p)$ such that $V+[\ufr(\f_p),\ufr(\f_p)]=\gfr(\f_p)$. If we show that $J(A_G)\gfr(\f_p)\supseteq [\ufr(\f_p),\ufr(\f_p)]$, then we get that $V+J(A_G)\gfr(\f_p)=\gfr(\f_p)$. And so by Nakayama's lemma we have $V=\gfr(\f_p)$, which is a contradiction.  

Since $\bbu$ is unipotent, $\log$ and $\exp$ define $\bbq$-morphisms between $\bbu$ and its Lie algebra. And for large enough $p$ they induce bijections between $\ucal(\f_p)$ and $\ufr(\f_p)$. Furthermore for any $t\in\f_p$, $x,y\in \ufr(\f_p)$ we have
\be\label{e:ExpAndAd}
\Ad(\exp(tx))(y)=\exp(t\ad(x))(y).
\ee
Thus for $t\in\f_p$, $x,y\in \ufr(\f_p)$ we have
\be\label{e:AugmentationIdeal}
t^{-1}(\Ad(\exp(tx))(y)-y)=[x,y]+\sum_{i=1}^{\dim\bbu} \frac{t^i}{i!} \ad(x)^i(y)\in J(A_G)\gfr(\f_p).
\ee
Therefore for large enough $p$ by the Vandermonde determinant we have that $[x,y]\in J(A_G)\gfr(\f_p)$ which shows our claim.

To finish the proof, it is enough to notice that $\bbg$ is perfect if and only if $\bbh$ acts without a non-zero fixed vector on $\bbu/[\bbu,\bbu]$. And so when $\bbg$ is perfect, for large enough $p$ we have $\zfr(\f_p)\subseteq [\ufr,\ufr](\f_p)=[\ufr(\f_p),\ufr(\f_p)]$. 
\end{proof}
\begin{lem}\label{l:UnipotentModPSquare}
Let $\bbu\subseteq (\GL_n)_{\bbq_p}$ be a $\bbq_p$-subgroup, and $U:=\bbu(\bbq_p)\cap \GL_n(\bbz_p)$. Then if $p$ is large enough depending only on $n$, then we have
\[
\pi_{p^2}(U)=\langle g\in\pi_{p^2}(U)|\h g^p\neq 1\rangle,\h\text{and}\h\h\h U[p^m]=U^{p^m}:=\{u^{p^m}|\h u\in U\},
\]
for any positive integer $m$.
\end{lem}
\begin{proof}
Let $\exp$ and $\log$ define $\bbq_p$-morphisms between $\Lie(\bbu)$ and $\bbu$. And so for large enough prime $p$ we have 
\be\label{e:ExpLog}
\|\exp x-I\|_p=\|x\|_p,\h \|\log u\|_p=\|u-I\|_p\hspace{1cm} \text{if}\h \|x\|_p,\|u\|_p\le 1 \text{ and } x^n=0, (u-I)^n=0.
\ee
Hence $\exp$ and $\log$ induce homeomorphisms between $\ufr:=\Lie(\bbu)(\bbq_p)\cap \gl_n(\bbz_p)$ and $U$, and moreover for any positive integer $m$ 
\be\label{e:Uniform}
U[p^m]=U^{p^m}:=\{u^{p^m}|\h u\in U\}.
\ee
We have (\ref{e:Uniform}) because of the following observation:
\begin{align*}
u\in U[p^m] &\Leftrightarrow \|u-I\|_p\le p^{-m} \Leftrightarrow u=\exp x,\h x\in \Lie(\bbu)(\bbq_p)\text{ and}\h \|x\|_p\le p^{-m}, \\
& \Leftrightarrow u=\exp(p^mx')=(\exp x')^{p^m} \h\text{and}\h x'\in \ufr,\\ 
&\Leftrightarrow u\in U^{p^m}.
\end{align*}
If $g^p=1$ for some $g\in \pi_{p^2}(U)$, then for some $u\in U$ we have 
$g=\pi_{p^2}(u)$ and $u^p\in \ker(\pi_{p^2})$. So there is a (unique) $u'\in U\setminus U^p$ such that $u'^{p^m}=u^p$ for some $m\in \bbz^{\ge 2}$. Another use of the logarithm map, we have $u=u'^{p^{m-1}}$. Hence 
\[
\pi_{p^2}(U)=\{g^{p^i}|\h g\in \pi_{p^2}(U),\h g^p\neq 1,\h i=0,1,2\}.
\] 
\end{proof}
\begin{lem}\label{l:HyperspecialParahoric}
Let $\bbh$ be a simply-connected semisimple $\bbq$-group with a given embedding into $(\GL_n)_{\bbq}$. Let $P_0=\bbh(\bbq_p)\cap\GL_n(\bbz_p)$ and $P_1:=\ker(P_0\xrightarrow{\pi_p} \GL_n(\bbz/p\bbz))$. Then for large enough $p$ we have
\[
P_i=\overline{\langle \bbu(\bbq_p)\cap P_i|\h \bbu\subseteq \bbh\h\text{unipotent $\bbq_p$-subgroup}\rangle},
\] 
for $i=0,1$.
\end{lem}
\begin{proof}
For large enough $p$, we have that $\bbh$ is quasi-split over $\bbq_p$, and splits over an unramified extension of $\bbq_p$, and $P_0$ is a hyper-special parahoric subgroup. 
Let $Q_i=\langle \bbu(\bbq_p)\cap P_i|\h \bbu\subseteq \bbh\h\text{unipotent $\bbq_p$-subgroup}\rangle.$ Then clearly $Q_i$ is a normal subgroup of $P_0$. By the Bruhat-Tits theory, there is a smooth $\bbz_p$-group scheme $\cal$ such that
\begin{enumerate}
	\item $P_0=\cal(\bbz_p)$,
	\item The generic fiber of $\cal$ is isomorphic to $\bbh$,
	\item The special fiber $\cal_p:=\cal\otimes_{\bbz_p}\f_p$ is a simply-connected semisimple $\f_p$-group since $\bbh$ is simply connected quasi-split over $\bbq_p$ and splits over an unramified extension of $\bbq_p$,  
	\item $\pi_p(P_0)=\cal_p(\f_p)$ is a product of quasi-simple groups. 
\end{enumerate} 
Since $\bbh$ is quasi-split over $\bbq_p$, $\pi_p(Q_0)$ intersects any quasi-simple factor subgroup of $\cal_p(\f_p)$ in a non-central subgroup. As $\pi_p(Q_0)$ is a normal subgroup of $\cal_p(\f_p)$, $\cal_p(\f_p)^+$ is contained in $\pi_p(Q_0)$. As $\cal_p$ is simply-connected, we get that $\pi_p(Q_0)=\pi_p(P_0)$. 
For any positive integer $i$, let $P_i:=\ker(P\xrightarrow{\pi_{p^i}} \cal_p(\bbz_p/p^i\bbz_p))$. It is well-known that $P_i/P_{i+1}$ as an $\cal_p(\f_p)$-module is isomorphic to $\Lie(\cal_p)(\f_p)$ via the adjoint action. Let's identify $\pi_{p^i}(Q_1\cap P_{i+1})$ with $\hfr_i\subseteq \Lie(\cal_p)(\f_p)$ . Since $Q_1\lhd P$, $\hfr_i$ is $\cal_p(\f_p)$-invariant. On the other hand, since $\bbh$ is quasi-split, $\hfr_i$ intersects each one of the Lie algebras of the simple factors of $\cal_p$. Since for large enough $p$ the $\f_p$-points of the Lie algebra of the simple factors of $\cal_p$ are simple $\cal_p(\f_p)$-module, $\hfr_i=\Lie(\cal_p)(\f_p)$. Hence by induction we have $\pi_{p^i}(Q_j)=\pi_{p^i}(P_j)$ for $j=0,1$ and any positive integer $i$, which finishes the proof.
\end{proof}

\begin{lem}\label{l:NotSplittingModPSquare}
Let $\bbg\subseteq (\GL_n)_{\bbq}$ be a perfect $\bbq$-group. Assume that the semisimple part of $\bbg$ is simply-connected; that means $\bbg\simeq \bbh\ltimes \bbu$ where $\bbh$ is a simply-connected semisimple $\bbq$-group and $\bbu$ is a unipotent $\bbq$-group. Let $P_p=\bbg(\bbq_p)\cap \GL_n(\bbz_p)$, $\gcal$ be the closure of $\bbg$ in $(\GL_n)_{\bbz}$, and $\gfr=\Lie(\gcal)$. Then 
\begin{enumerate}
\item For large enough $p$, the following is a short exact sequence:
\[
1\rightarrow \gfr(\f_p)\rightarrow \pi_{p^2}(P_p) \rightarrow \pi_p(P_p) \rightarrow 1.
\]
\item For large enough $p$ and any proper subgroup $V\subsetneq \gfr(\f_p)$ which is a normal subgroup of $\pi_{p^2}(P_p)$, the following short exact sequence does not split
\[
1\rightarrow \gfr(\f_p)/V\rightarrow \pi_{p^2}(P_p) /V \rightarrow \pi_p(P_p) \rightarrow 1
\]
\end{enumerate}
\end{lem}
\begin{proof}
For large enough $p$, $Q_p:=\bbh(\bbq_p)\cap \GL_n(\bbz_p)$ is a hyperspecial parahoric subgroup which acts on $U_p:=\bbu(\bbq_p)\cap \GL_n(\bbz_p)$, and $P_p=Q_p\ltimes U_p$. Furthermore for large enough $p$, the assumptions of Lemma~\ref{l:HyperspecialParahoric} hold and so $\pi_{p^i}(P_p)=\pi_{p^i}(\langle P_p\cap \bbu'(\bbq_p)|\h \bbu'\h\text{unipotent $\bbq_p$-subgroup}\rangle)$. Hence by Lemma~\ref{l:UnipotentModPSquare} we have 
\[
\pi_{p^2}(P_p)=\langle g\in\pi_{p^2}(P_p\cap\bbu'(\bbq_p))|\h \bbu'\h\text{unipotent $\bbq_p$-subgroup},\h g^p\neq 1\rangle.
\]
On the other hand, if for any $\bbq_p$-unipotent subgroup $\bbu'$ of $\bbg$ we have $\pi_{p^2}( P_p[p]\cap \bbu'(\bbq_p)) \subseteq V$, then by Lemma~\ref{l:HyperspecialParahoric} we get that $\pi_{p^2}(P_p[p])\subseteq V$ which is a contradiction. Hence for some unipotent $\bbq_p$-subgroup $\bbu'$ there is $x\in \pi_{p^2}(P_p[p]\cap \bbu'(\bbq_p)) \setminus V$. By Lemma~\ref{l:UnipotentModPSquare} there is $g\in \GL_n(\bbz_p)\cap \bbu'(\bbq_p)$ such that $\pi_{p^2}(g^p)=x$. Since $\bbu'$ is a $\bbq_p$-subgroup of $\bbg$ and $P_p=\GL_n(\bbz_p)\cap \bbg(\bbq_p)$, we have that $g\in P_p$, and the order of $\pi_{p^2}(g)$ in $\pi_{p^2}(P_p)/V$ is $p^2$. Now suppose to the contrary that the given exact sequence splits; then $\pi_{p^2}(P_p)/V\simeq \pi_p(P_p)\ltimes (\gfr(\f_p)/V)$, and the latter is a linear group over $\f_p$. However any $p$-element of a linear group over $\f_p$ is unipotent, and so it is of order $p$ (for large enough $p$), which is a contradiction.  
\end{proof}
\subsection{Statistical non-splitting of residue maps.} 
In this section, following~\cite{BV}, we prove that any section of $\pi_p$ is statistically far from being a group homomorphism. 
\begin{lem}\label{l:SectionFarFromHomo}
Let $\bbg$, $\gfr$, and $P_p$ be as in Lemma~\ref{l:NotSplittingModPSquare}. Then there is a positive number $\delta$ (depending on $\bbg$) such that for large enough $p$ (depending on $\bbg$ and its embedding in $(\GL_n)_{\bbq}$) the following holds: Let
$V\subsetneq \gfr(\f_p)$ be a proper subgroup which is normal in $\pi_{p^2}(P_p)$, and $\psi:\pi_p(P_p)\rightarrow \pi_{p^2}(P_p)/V$ be a section of the quotient map $\pi:\pi_{p^2}(P_p) /V\rightarrow \pi_p(P_p)$; that means $\pi\circ\psi={\rm id}_{\pi_p(P_p)}$. Let $\mu=\pcal_{\pi_p(P_p)}$ be the probability Haar measure (i.e. the normalized counting measure) on $\pi_p(P_p)$. Then
\[
(\mu\times\mu)(\{(x,y)\in \pi_p(P_p)\times \pi_p(P_p)|\h \psi(x)\psi(y)=\psi(xy)\})<p^{-\delta}.
\]
\end{lem}
\begin{proof}
We notice that by Lemma~\ref{l:ProperAdInvariant}, for large enough $p$, $V+\zfr(\f_p)$ is a proper subspace of $\gfr(\f_p)$ (which is normal in $\pi_{p^2}(P_p)$). And so without loss of generality we can assume that $V$ contains $\zfr(\f_p)$.

Let $\acal:=\{(x,y)\in \pi_p(P_p)\times \pi_p(P_p)|\h \psi(x)\psi(y)=\psi(xy)\}$ and suppose to the contrary that $(\mu\times\mu)(\acal)\ge p^{-\delta}.$
We will get a contradiction if $\delta$ is small enough. Let $\nu$ be the push-forward of $\mu$ via $\psi$. For any $g\in \pi_p(P_p)$ we have
\begin{align*}
(\nu\ast\nu)(\psi(g))&=\sum_{h_1h_2=\psi(g)}\nu(h_1)\nu(h_2)=\sum_{{\scriptsize \begin{array}{c}x,y\in \pi_p(P_p)\\ \psi(x)\psi(y)=\psi(g)\end{array}}} \mu(x)\mu(y)\hspace{2.25cm}\text{(since $\supp \nu={\rm Im} \psi$)}\\
&=\sum_{{\scriptsize\begin{array}{c}(x,y)\in \acal\\ xy=g\end{array}}} \mu(x)\mu(y). 
\hspace{4.5cm}
\begin{array}{c}
(\text{since } g=\pi(\psi(g))=\pi(\psi(x)\psi(y))
	\\
=\pi(\psi(x))\pi(\psi(y))=xy.)	
\end{array}
\end{align*}

In particular, we have $(\nu\ast\nu)(\supp \nu)=(\mu\times \mu)(\acal)$. And so by the Cauchy-Schwarz inequality, we get 
\[
\sqrt{|\supp \nu|}\|\nu\ast\nu\|_2\ge (\nu\ast\nu)(\supp \nu)=(\mu\times\mu)(\acal)\ge p^{-\delta}.
\]
Hence $\|\nu\ast\nu\|_2\ge p^{-\delta} \|\nu\|_2$. Thus by \cite[Lemma 15]{Var} (which is based on \cite{BG1} (see Proposition~\ref{BourgainGamurd})) there is a symmetric subset $S\subseteq \pi_{p^2}(P_p)/V$ with the following properties:
\begin{enumerate}
	\item[(P1)] $|\pi_p(P_p)|^{1-\Theta_{\bbg}(\delta)}\le |S|\le |\pi_p(P_p)|^{1+\Theta_{\bbg}(\delta)}$,
	\item[(P2)] $|\prod_3 S|\le |\pi_p(P_p)|^{\Theta_{\bbg}(\delta)} |S|$,
	\item[(P3)] $\min_{s\in S}(\widetilde{\nu}\ast\nu)(s)\ge (|\pi_p(P_p)|^{\Theta_{\bbg}(\delta)}|S|)^{-1}$.
\end{enumerate}
Since the push-forward $\pi[\nu]$ of $\nu$ via $\pi$ is the probability Haar measure of $\pi_p(P_p)$, we have $\mu=\pi[\widetilde{\nu}\ast\nu]$. And so  
the third part of properties of $S$ implies that $\mu(\pi(S))\ge |\pi_p(P_p)|^{-\Theta_{\bbg}(\delta)}$. Therefore $|\pi(S)|\ge |\pi_p(P_p)|^{1-\Theta_{\bbg}(\delta)}$. On the other hand, by \cite[Corollary 14]{SGV} and the main theorem of \cite{LS}, we have that $\pi_p(P_p)$ is a quasi-random group; that means the minimal degree of its non-trivial irreducible representations is at least $|\pi_p(P_p)|^{\Theta_{\bbg}(1)}$ for large enough $p$. Thus by Gowers's result~\cite{Gow} (we use the formulation in~\cite[Corollary 1]{NP}) we have that $\pi(\prod_3S)=\pi_p(P_p)$ if $\delta$ is small enough. So $S'=\prod_3 S$ has the following properties:
\begin{enumerate}
	\item $\pi(S')=\pi_p(P_p)$,
	\item $S'$ is a $|\pi_p(P_p)|^{\Theta_{\bbg}(\delta)}$-approximate subgroup by \cite[Corollary 3.11]{Tao},
	\item $|S'|\le |\pi_p(P_p)|^{1+\Theta_{\bbg}(\delta)}$ by properties (P1) and (P2) of $S$.
\end{enumerate}

Next we prove that there is $0\neq x\in \prod_3 S'\cap (\gfr(\f_p)/V)$. We proceed by contradiction. For any $s_1,s_2\in S'$, there is $s_3\in S'$ such that $\pi(s_1s_2)=\pi(s_3)$. Hence $s_1s_2s_3^{-1}\in \prod_3 S'\cap (\gfr(\f_p)/V)$, and so by the contrary assumption we have $s_1s_2=s_3$. Thus $S'$ is a subgroup which contradicts Lemma~\ref{l:NotSplittingModPSquare}.   

Since $V$ contains $\zfr(\f_p)$, $\pi(S)=\pi_p(P_p)$ acts on $(\gfr(\f_p)/V)$ via the adjoint action without non-zero fixed point for large enough $p$. Hence by \cite[Lemma 30]{SGV}, $\sum_{O_{\bbg}(1)}\pi(S')\cdot x$ contains a non-zero subspace of $\gfr(\f_p)/V$. Therefore $|\prod_{O_{\bbg}(1)} S' \cap (\gfr(\f_p)/V)|\geq p$, and we get $
|\prod_{O_{\bbg}(1)} S'|\ge p|\pi_p(P_p)|.$
Since $S'$ is an approximate subgroup, we get
\[
p|\pi_p(P_p)|\leq |\textstyle\prod_{O_{\bbg}(1)} S'|\leq |\pi_p(P_p)|^{1+\Theta_{\bbg}(\delta)}.
\]
So $p\leq |\pi_p(P_p)|^{\Theta_{\bbg}(\delta)}$ which is a contradiction for small enough $\delta$ (independent of $p$).    
\end{proof}
\subsection{Bounded generation of perfect groups by commutators.}
\begin{lem}\label{l:CommutatorPGroup}
Let $P$ be a finite $p$-group. Suppose $|P|=p^n$ and $|P/[P,P]|=p^m$. Then 
\[
\textstyle \prod_{n-m} \{[g_1,g_2]|\h g_1,g_2\in P\}=[P,P],
\]
where $[g_1,g_2]:=g_1g_2g_1^{-1}g_2^{-1}$.
\end{lem}
\begin{proof}
We proceed by induction on the nilpotency class of $P$. If $P$ is abelian, there is nothing to prove. Now suppose the nilpotency class of $P$ is $c$; that means its $c+1$-th lower central series $\gamma_{c+1}(P)$ is trivial and $\gamma_c(P)$ is not trivial. It is well-known that $\gamma_c(P)$ is an $\f_p$-vector space which is spanned by the long commutators. So $\gamma_c(P)$ is contained in $\prod_{\dim_{\f_p} \gamma_c(P)} \{[g_1,g_2]|\h g_1,g_2\in P\}$. By the induction hypothesis, we have 
\[
\textstyle\prod_{n-m-\dim_{\f_p} \gamma_c(P)} \{[\bar{g}_1,\bar g_2]|\h \bar g_1,\bar g_2\in P/\gamma_c(P)\}=P/\gamma_c(P).
\]
And so $\prod_{n-m} \{[g_1,g_2]|\h g_1,g_2\in P\}=P$.
\end{proof}
\begin{lem}\label{l:CommutatorPerfect}
Let $\bbg$ be a Zariski-connected perfect $\bbq$-subgroup of $(\GL_n)_{\bbq}$, and let $\gcal$ be its Zariski-closure in $(\GL_n)_{\bbz}$. Suppose that $\bbg$ is simply-connected; that means its semisimple part is simply-connected. Then there is a positive integer $C$ depending on $\bbg$ such that for large enough prime $p$ we have 
\[
\textstyle\prod_{C}\{[g_1,g_2]|\h g_1,g_2\in \gcal_p(\f_p)\}=\gcal_p(\f_p),
\]   
where $\gcal_p=\gcal\times_{\Spec \bbz} \Spec \f_p$.
\end{lem}
\begin{proof}
As it was mentioned earlier, for large enough $p$, $\gcal_p\simeq \cal_p\ltimes \ucal_p$ is a perfect $\f_p$-group, where $\cal_p$ is a simply connected semisimple $\f_p$-group and $\ucal_p$ is a unipotent $\f_p$-group. Moreover $\gcal_p(\f_p)=\cal_p(\f_p)\ltimes \ucal_p(\f_p)$, and $V_p:=\ucal_p(\f_p)/[\ucal_p(\f_p),\ucal_p(\f_p)]$ is a completely reducible $\cal_p(\f_p)$-module with no non-zero fixed vector. Let $w(x_1,x_2)=[x_1,x_2]$ and, for any group $K$, let $w(K):=\{w(k_1,k_2)|\h k_1,k_2\in K\}$. Then by \cite[Theorem 1.1]{Sha} (see an alternative approach in \cite[Theorem 3]{NP}) we have that $\prod_3 w(\cal_p(\f_p))=\cal_p(\f_p)$ for large enough $p$. 

Let $H_p=\cal_p(\f_p)$ and $G_p=H_p\ltimes V_p$. Let $\rho:H_p\rightarrow \GL(V_p)$ be the homomorphism induced from the action of $H_p$ on $V_p$. Then $\prod_{7} w(G_p)$ contains 
\[
\{\rho(h)((\rho(h')-I)(v))|\h h,h'\in H_p,\h v\in V_p\}.
\] 
And so $\prod_{7\dim_{\f_p} V_p} w(G_p)$ contains the $\f_p[H_p]$-submodule $M_p$ of $V_p$ generated by $\{(\rho(h)-I)(v)|\h h\in H_p,\h v\in V_p\}$. Therefore $H_p$ acts trivially on $V_p/M_p$. Since $V_p$ is a completely reducible $H_p$-module with no non-zero fixed vectors, $M_p=V_p$. Thus 
\[
\textstyle\prod_{O(\dim \bbg)} w(G_p)=G_p.
\]          
And so by Lemma~\ref{l:CommutatorPGroup} applied for $P:=\ucal_p(\f_p)$ the claim follows.
\end{proof}
\begin{lem}\label{l:CommutatorCongruencePerfect}
Let $\bbg$ and $\gcal$ be as in Lemma~\ref{l:CommutatorPerfect}, and $N$ be a positive integer. Suppose the unipotent radical of $\bbg$ is abelian. Then there is a positive integer $C$ depending on $\gcal$ and $N$ such that for large enough $p$ we have
\[
\textstyle \prod_C \{[g_1,g_2]|\h g_i\in \gcal(\bbz/p^N\bbz)\}=\gcal(\bbz/p^N \bbz).
\]
\end{lem}
\begin{proof}
By Lemma~\ref{l:CommutatorPerfect} we get the case of $N=1$. It is well-known that $\pi_p$ induces the following short exact sequence
\[
1\rightarrow \gfr(\f_p) \rightarrow \gcal(\bbz/p^2\bbz) \xrightarrow{\pi_p} \gcal(\f_p) \rightarrow 1,
\]
where $\gfr=\Lie(\gcal)$, and the conjugation action of $\gcal(\bbz/p^2\bbz)$ on $\gfr(\f_p)$ factors through the adjoint action of $\gcal(\f_p)$. Let $\ufr$ be the Lie algebra of the Zariski-closure of the unipotent radical of $\bbg$ in $\gcal$. For large enough $p$, $\gcal(\f_p)$ and $\gfr(\f_p)$ are perfect, and $\gfr(\f_p)$ is a completely reducible $\gcal(f_p)$-module without any non-zero invariant vectors (as $\bbu$ is abelian). So by a similar argument as in the proof of Lemma~\ref{l:CommutatorPerfect}, we have
\[
\textstyle\prod_{O_{\bbg}(1)}w(\gcal(\bbz/p^2\bbz))\supseteq \gfr(\f_p),
\]
and so by Lemma~\ref{l:CommutatorPerfect}, we have that
\[
\textstyle\prod_{O_{\bbg}(1)}w(\gcal(\bbz/p^2\bbz))=\gcal(\bbz/p^2\bbz).
\]   
Now using the facts that $\gcal(\bbz_p)[p^i]/\gcal(\bbz_p)[p^{i+1}]\simeq \gfr(\f_p)$, the group commutators are mapped to the Lie algebra commutators, and $\gfr(f_p)$ is perfect, we get the desired result. (See \cite[Lemma 34]{SGsemisimple} for the relations between the group commutators and the Lie algebra commutators.)
\end{proof}
\subsection{Proof of Theorem~\ref{t:main} for bounded powers of square-free integers.}\label{ss:ProofOfBoundedPowerOfSquareFrees}
Now we go to the proof of Theorem~\ref{t:main} for $\ccal_{N}:=\{q\in \bbz^+|\h \gcd(q,q_0)=1,\h \forall p\in V_f(\bbq), v_p(q_0)\le N\}$.  Having the above lemmas the rest of the proof is a modification of the argument in~\cite[Section 3]{BV}.

By Section~\ref{ss:ReductionToBG}, it is enough to prove Theorem~\ref{t:BoundedGeneration} for $\ccal_{N}$. First we point out that it is enough to prove Theorem~\ref{t:BoundedGeneration} for $\ccal_{N}^{\ge p_0}:=\{q\in \bbz^+|\h \gcd(q,q_0)=1,\h \forall p\in V_f(\bbq), v_p(q_0)\le N, \forall p\le p_0, v_p(q_0)=0\}$. This Lemma helps us to avoid all the small primes $p$ where $\pi_p(\Gamma)$ does not behave nicely; in particular we will be able to assume Lemmas \ref{l:ProperAdInvariant}-\ref{l:CommutatorCongruencePerfect} to hold for all the prime divisors of $Q$. 

\begin{lem}\label{l:PrimeFactorsAreArbitrarilyLarge}
 	Let $p_0$ be a prime number. Then a bounded generation in the sense of Theorem~\ref{t:BoundedGeneration} for $\ccal_{N}^{\ge p_0}$ implies Theorem~\ref{t:BoundedGeneration} for $\ccal_{N}$ where the implied constants depend also on $p_0$.
\end{lem}
\begin{proof}
	Fix a positive number $\delta_0:=\delta_0(\vare,\Omega,p_0)$ and a positive integer $C_0:=C_0(\vare,\Omega,p_0)$ for which Theorem~\ref{t:BoundedGeneration} holds for $\ccal_N^{\ge p_0}$. 
		
	For $Q\in \ccal_N$, let $q_0:=\prod_{p|Q, p\le p_0} p^{v_p(Q)}$. Since $q_0< p_0^{Np_0}$, we have $[\Gamma:\Gamma[q_0]]\ll_{\bbg,p_0,N} 1$. Suppose $Q$ is large enough (depending only on $\vare,\Omega,p_0,N$) such that the bounded generation for $\ccal_N^{\ge p_0}$ can be applied for $Q/q_0\in\ccal_N^{\ge p_0}$. Moreover we assume $Q$ is large enough (again depending only on $\vare,\Omega,p_0,N$) such that $Q^{-\delta_0/2}/[\Gamma,\Gamma[q_0]]>(Q/q_0)^{-\delta_0}$. Now suppose $\delta$ is a positive number which is at most $\delta_0/4$, $A$ is a symmetric set, and $l$ is a positive integer which is more than $(1/\delta)\log Q$ such that $\pcal_{\Omega}^{(l)}(A)>Q^{-\delta}$. Then 
	\[
	\pcal_{\Omega}^{(2l)}(A\cdot A\cap \Gamma[q_0])\ge Q^{-2\delta}/[\Gamma:\Gamma[q_0]]>Q^{-\delta_0/2}/[\Gamma:\Gamma[q_0]]>(Q/q_0)^{-\delta_0}.
	\]
	Hence $\Pfr_{Q/q_0}(\delta_0,A\cdot A\cap \Gamma[q_0],l)$ holds. Therefore by the bounded generation for $\ccal_N^{\ge p_0}$ we have that 
	\be\label{e:LargeSubgroup}
	\textstyle	\pi_{Q/q_0}(\Gamma[q])\subseteq \pi_{Q/q_0}(\prod_{C_0} (A\cdot A\cap \Gamma[q_0])),
	\ee
	for some $q|Q/q_0$ and $q\le (Q/q_0)^{\vare}$. Since $Q$ and $Q/q_0$ are coprime, $\pi_Q(\Gamma)$ can be diagonally embedded into $\pi_{q_0}(\Gamma)\oplus \pi_{Q/q_0}(\Gamma)$, and by enlarging $p_0$, if needed, we can and will assume that $\pi_{Q}(\Gamma[qq_0])$ gets identified with $\{1\}\oplus \pi_{Q/q_0}(\Gamma[q])$ under the diagonal embedding of $\pi_Q(\Gamma)$ into $\pi_{q_0}(\Gamma)\oplus \pi_{Q/q_0}(\Gamma)$. Notice that $\pi_Q(A\cdot A \cap \Gamma[q_0])$ gets mapped to a subset of $\{1\}\oplus \pi_{Q/q_0}(\Gamma[q])$ and so by (\ref{e:LargeSubgroup}) we get that 
	\[
	\textstyle\pi_Q(\Gamma[qq_0])\subseteq \prod_{2C_0} \pi_Q(A).
	\] 
	And we have $qq_0|Q$ and $qq_0\le Q^{\vare}q_0^{1-\vare}\le Q^{2\vare}$ for large enough $Q$ (again depending on $\vare,\Omega,p_0,N$). 
\end{proof}

Next we show a reduction which works for any family of positive integers $\ccal'$: 

\begin{lem}
	\label{l:ToProveBGWeCanAssumeAS}
	To prove Theorem~\ref{t:BoundedGeneration} for $Q$ in a family of positive integers $\ccal'$, it is enough to prove the bounded generation claim for a subset $A$ for which $\Pfr_Q(\delta, A, l)$ and $\Pfr_Q(\delta,\rho_1(A),l)$ hold (where $\rho_1:\bbg\rightarrow \bbg_s$ is the quotient map). That means for any $0<\vare\ll_{\Omega} 1$ there are $0<\delta\ll_{\Omega,\vare} 1$ and positive integer $C\gg_{\Omega,\vare} 1$ for which the following holds:
	\begin{center}
	 {\rm {\bf (Restricted BG)}:}	$\Pfr_Q(\delta, A, l)\wedge\Pfr_Q(\delta,\rho_1(A),l)$ implies $\pi_Q(\Gamma[q])\subseteq \prod_C \pi_Q(A)$ for some $q|Q$ such that $q\le Q^{\vare}$ 	if $Q\in\ccal'$ and $Q^{\vare^{O_{\bbg}(1)}}\gg_{\Omega}1$.
	\end{center} 
\end{lem} 
\begin{proof}
	Suppose $\delta$ is a positive number and $\pcal_{\Omega}^{(l)}(A)>Q^{-\delta}$ for a symmetric set $A$ and a positive integer $l>(1/\delta)\log Q$. Since $\ker(\rho_1)\cap \Gamma=\{1\}$, we have $\pcal_{\rho_1(\Omega)}^{(l)}(\rho_1(A))>Q^{-\delta}$.	Then by a similar argument as in (\ref{e:Size}) (based on the Kesten bound (see Inequality~(\ref{e:KestenBound}))) we have 
	\be\label{e:LowerBoundOnSize}
	|\pi_Q(A)|>Q^{\eta_0}\h\h {\rm and }\h\h |\pi_Q(\rho_1(A))|>Q^{\eta_0},
	\ee
	for small enough $\delta$ (depending only on $\Omega$) and some positive number $\eta_0:=\eta_0(\Omega)$ which depends only on $\Omega$. 

	Fix a positive number $\delta_0:=\delta_0(\vare,\Omega)$ and a positive integer $C_0:=C_0(\vare,\Omega)$ for which the {\bf Restricted BG} holds. And let 
	\begin{align}
	\label{e:SetOfGoodIndexes}
	 I_{A}& \textstyle :=\{i\in \bbz^+|\h |\prod_{3^{i+1}}\pi_Q(A)|\le |\prod_{3^i}\pi_Q(A)|^{1+\delta_0}\} \text{ and }\\
	\label{e:SetOfGoodIndexesSS}
	I_{\rho_1(A)}&\textstyle :=\{i\in \bbz^+|\h |\prod_{3^{i+1}}\pi_Q(\rho_1(A))|\le |\prod_{3^i}\pi_Q(\rho_1(A))|^{1+\delta_0}\}.
	\end{align}
	Notice that $|\pi_Q(\Gamma)|\le Q^{N_0(\bbg)}$ and $|\pi_Q(\rho_1(\Gamma))| \le Q^{N_0(\bbg)}$  for some positive integer $N_0$ which depends only on the dimension of $\bbg$.  Now let $N_1:=N_1(\delta_0)$ be the smallest positive integer such that
	\be\label{e:NumberOfTripling}
	(1+\delta_0)^{N_1}\eta_0>N_0.
	\ee
	
	So by Inequalities (\ref{e:LowerBoundOnSize}) and (\ref{e:NumberOfTripling}), the complements of $I_{A}$ and $I_{\rho_1(A)}$ in $\bbz^+$ have at most $N_1$ elements. Hence there is a positive integer $C_1\le 3^{2N_1+1}\ll_{\Omega,\vare} 1$ such that 
	\begin{align*}
	\textstyle
	|\pi_Q(\prod_{C_1}\rho_1(A))\cdot \pi_Q(\prod_{C_1} \rho_1(A))\cdot \pi_Q(\prod_{C_1} \rho_1(A))|&	\textstyle\le |\pi_Q(\prod_{C_1} \rho_1(A))|^{1+\delta_0}\text{ and } \\
	\textstyle	
	|\pi_Q(\prod_{C_1}A)\cdot \pi_Q(\prod_{C_1} A)\cdot \pi_Q(\prod_{C_1} A)|&
	\textstyle\le |\pi_Q(\prod_{C_1} A)|^{1+\delta_0}.
	\end{align*}
	Therefore, if $\delta<\delta_0$, then $\Pfr_Q(\delta_0,\prod_{C_1}A,l)$ and $\Pfr_Q(\delta_0,\prod_{C_1}(\rho_1(A)),l)$ hold. Thus, by the {\bf Restricted BG} and the fact that $C_1C_0\ll_{\Omega,\vare} 1$, we are done. 
\end{proof}
{\em (Going back to the proof of Theorem~\ref{t:BoundedGeneration} for $\ccal_{N}$:)}
Suppose $p_0$ is a prime number such that $p_0\ll_{\Omega} 1$ and all the claims of Lemmas \ref{l:ProperAdInvariant}-\ref{l:CommutatorCongruencePerfect} hold for all the primes $p\ge p_0$ (later we will assume $p_0$ is large enough depending on $\Omega$ to gain additional properties for $\pi_p(\Gamma)$ where $p$ is a prime divisor of $Q$.). So by Lemma~\ref{l:PrimeFactorsAreArbitrarilyLarge} we can assume that all the prime factors of $Q$ are at least $p_0$. By Lemma~\ref{l:ToProveBGWeCanAssumeAS} (for $\ccal':=\ccal_N^{\ge p_0}$), without loss of generality we can and will assume that $\Pfr_Q(\delta, A, l)$ and $\Pfr_Q(\delta,\rho_1(A),l)$ hold. 

 Let $Q_s:=\prod_{p|Q} p$ (, and so $Q|Q_s^N$). By \cite[Theorem 1]{SGV}, \cite[Lemma 31]{SGsemisimple}, and $l>\frac{N}{\delta}\log Q_s$ we have that
\[
\left|\pcal_{\pi_{Q_s}(\Omega)}^{(l)}(\pi_{Q_s}(A))-\frac{|\pi_{Q_s}(A)|}{|\pi_{Q_s}(\Gamma)|}\right|\le \frac{1}{|\pi_{Q_s}(\Gamma)|}
\]
for small enough $\delta$. And so we have
\[
{Q_s}^{-N\delta}\leq Q^{-\delta}\leq \pcal_{\pi_{Q_s}(\Omega)}^{(l)}(\pi_{Q_s}(A)) \leq \frac{|\pi_{Q_s}(A)|+1}{|\pi_{Q_s}(\Gamma)|}, 
\]
which implies that $|\pi_{Q_s}(A)|\geq |\pi_{Q_s}(\Gamma)|Q_s^{-\Theta_{\Omega,N}(\delta)}$.
Since we assumed that the prime factors of $Q$ are large enough, we have that $\pi_{Q_s}(\Gamma)\simeq \prod_{p|Q} \pi_p(\Gamma)$. Writing $Q_s=\prod_{j=1}^l p_j$, to the product set $\prod_{i=1}^l \pi_{p_i}(\Gamma)$ one can associate a rooted tree with $l$ levels. The root is considered to be the zero level. The vertices of the $j$-th level are elements of $\prod_{i=1}^j \pi_{p_i}(\Gamma)$, and each vertex is connected to its projection in the previous level. By a regularity argument similar to the proof of \cite[Lemma 12]{SGsemisimple} (see \cite[Lemma 5.2]{BGS} or \cite[Page 26]{Var}), one can show that there are a set $A'\subseteq A$ and a sequence of integers $\{k_i\}_{i=1}^l$ such that for any $g\in A'$ and for any $1\le i\le l$ we have
\[
|\{x\in \pi_{p_i}(\Gamma)|\h \exists h\in A':\h \pi_{p_1\cdots p_{i-1}}(h)=\pi_{p_1\cdots p_{i-1}}(g)\h\text{and}\h\pi_{p_i}(h)=x\}|=k_i,
\]
and 
\[
|A'|=\prod_{i=1}^l k_i\ge (\prod_{i=1}^l 2\log |\pi_{p_i}(\Gamma)|)^{-1} |A|.
\]

By \cite[Corollary 14]{SGV} and the main theorem of \cite{LS}, $\pi_p(\Gamma)$ is a quasi-random group for large enough $p$; that means the dimension of any non-trivial irreducible representation of $\pi_p(\Gamma)$ is at least $|\pi_p(\Gamma)|^{\Theta_{\bbg}(1)}$ (for $p\gg_{\Omega} 1$) (see also Proposition~\ref{p:HighMultiplicity}). So by a result of Gowers~\cite{Gow} (see~\cite[Corollary 1]{NP}) there is a positive number $c_0$ depending on $\bbg$ such that for large enough $p$ the following holds:
if $A,B,C\subseteq \pi_p(\Gamma)$ and $|A|,|B|,|C|\ge |\pi_p(\Gamma)|^{1-c_0}$, then $A\cdot B\cdot C=G$. We assume that the prime factors of $Q$ are large enough so that $2\log |\pi_{p_i}(\Gamma)|<|\pi_{p_i}(\Gamma)|^{c_0\vare}$. Hence
\[
\prod_{i=1}^l k_i\ge |\pi_{Q_s}(\Gamma)|^{1-\Theta_{\bbg}(\delta)-c_0\vare}.
\]
Let $I:=\{i|\h k_i<|\pi_{p_i}(\Gamma)|^{1-c_0}\}$; then we have
\[
|\pi_{Q_s}(\Gamma)|/\prod_{i\in I}|\pi_{p_i}(\Gamma)|^{c_0}\ge \prod_{i=1}^l k_i\ge |\pi_{Q_s}(\Gamma)|^{1-\Theta_{\bbg}(\delta)-c_0\vare},
\]
and so 
\[
\prod_{i\in I} p_i\le Q_s^{\Theta_{\bbg}(\delta)+2\vare}\le Q_s^{4\vare},
\]
for small enough $\delta$ (depending on $\Omega$, $N$ and $\vare$). Now by a similar argument as in~\cite[Page 26]{Var} (using the mentioned result of Gowers) one can show that
\be\label{e:FirstLevel}
\pi_{\prod_{i\not\in I}p_i}(A\cdot A\cdot A)=\pi_{\prod_{i\not\in I}p_i}(\Gamma).
\ee
Let $Q'_s=\prod_{i\not\in I}p_i$ and $q'=Q_s/Q_s'$. By (\ref{e:FirstLevel}) there is a function 
$f_1:\pi_{Q_s'}(\Gamma)\rightarrow \pi_{q'^N}(\Gamma)$ such that $\{(g,f_1(g))|\h g\in \pi_{Q_s'}(\Gamma)\}$ is contained in $\pi_{Q_s'q'^N}(\prod_3 A)$. Thus there is a fiber of $f_1$ with at least $Q_s^{1-\Theta_{\bbg,N}(\vare)}$-many elements. Hence, as before by a regularization argument and another application of the mentioned Gowers's result (for $\vare\ll_{\bbg,N}1$), there is $q''|Q_s'$ and a function $f_2:\pi_{Q_s'/q''}(\Gamma)\rightarrow \pi_{q''}(\Gamma)$ such that $q''\le Q^{\Theta_{\bbg,N}(\vare)}$ and 
\[
\{(g,f_2(g),1)\in \pi_{Q_s'/q''}(\Gamma)\times \pi_{q''}(\Gamma)\times \pi_{q'^N}(\Gamma)| g\in \pi_{Q_s'/q''}(\Gamma)\}
\]
is contained in $\pi_{Q_s'q'^N}(\prod_{O_{\bbg}(1)}A)$. Hence, for any $g_1,g_2\in \pi_{Q_s'/q''}(\Gamma)$, we have 
\[
\textstyle ([g_1,g_2],1,1)=[(g_1,f_2(g_1),1),(g_2,1,f_1(g_2,1))]\in \pi_{Q_s'q'^N}(\prod_{O_{\bbg}(1)}A).
\]
Therefore by Lemma~\ref{l:CommutatorPerfect} 
\[
\textstyle \pi_{Q_s'q'^N}(\Gamma[q'^N])\subseteq \pi_{Q_s'q'^N}(\prod_{Q_{\bbg}(1)}A).
\]
Let $A'$ be a subset of $\prod_{Q_{\bbg}(1)}A$ such that $\pi_{Q_s'q'^N}(\Gamma[q'^N])=\pi_{Q_s'q'^N}(A')$. Now we will go to the second level; that means we will show that there are a positive integer $C$ and  $\bar q|Q'_s$ such that $\bar q\le Q^{\Theta_{\bbg}(\vare)}$ and $\pi_{{Q'_s}^2}(\prod_C A')$ contains $\pi_{{Q'_s}^2}(\Gamma[{\bar q}^2])$. 

\begin{lem}\label{l:Enlarging}
In the above setting, assume $\bbu$ is abelian. Then for any $\vare>0$ there are $\delta>0$ and a positive integer $C$ such that the following holds.

Assume $A\subseteq \Gamma$ and a square-free integer $Q_s$ have the following properties:
\begin{enumerate}
\item $Q_s^{\vare}\gg_{\Omega} 1$.
\item The prime factors of $Q_s$ are sufficiently large (depending on $\Omega$); in particular the assertions of Lemma~\ref{l:NotSplittingModPSquare} hold, and $\pi_{Q_s^2}(\Gamma)\simeq \prod_{p|Q_s} \pi_{p^2}(\Gamma)$.
\item $\pi_{Q_s}(A)=\pi_{Q_s}(\Gamma)$.
\item For any prime factor $p$ of $Q_s$, let $V_p\subsetneq \gfr(\f_p)$ be a $\pi_p(\Gamma)$-invariant under the adjoint action. And $\prod_{p|Q_s} V_p\subseteq \pi_{Q_s^2}(A)$. 
\end{enumerate}  
Then there are $Q_s'$ and $\{M_p\}_{p|Q_s'}$ such that
\begin{enumerate}
\item $Q_s'|Q_s$ and $Q_s^{1-\vare}\le Q_s'$.
\item For any $p$, we have $V_p\subsetneq M_p\subseteq \gfr(\f_p)$, and $M_p$ is $\pi_p(\Gamma)$-invariant under the adjoint action. 
\item $\prod_{p|Q_s'} M_p\oplus \prod_{p|(Q_s/Q_s')} \{0\} \subseteq \pi_{{Q_s}^2}(\prod_C A)$.
\end{enumerate}
\end{lem}
\begin{proof}
Since $V_p\subseteq \gfr(\f_p)$, the map $\pi_p:\pi_{p^2}(\Gamma)\rightarrow \pi_p(\Gamma)$ factors through $\pi_{p^2}(\Gamma)/V_p$. Let us abuse the notation and still denote the induce homomorphism by $\pi_p:\pi_{p^2}(\Gamma)/V_p\rightarrow \pi_p(\Gamma)$. Let $\psi:\pi_{Q_s}(\Gamma)\rightarrow A\subseteq \Gamma$ be a section of $\pi_{Q_s}$, and $\psi_p:\pi_{Q_s}(\Gamma)\rightarrow \pi_{p^2}(\Gamma)/V_p$ be $\psi_p(g)=\pi_{p^2}(\psi(g))V_p$. As in the proof of~\cite[Proposition 3]{BV}, let us consider the following expectation with respect to the probability Haar measure:
\begin{align}
\label{e:AverageWeightedBadPrimes}\bbe\left(\sum_{p|Q_s,\h \psi_p(xy)=\psi_p(x)\psi_p(y)} \log p \right)=&
\sum_{p|Q_s} \log p \cdot (\pcal_{\pi_p(\Gamma)}\times \pcal_{\pi_p(\Gamma)})(\{(x,y)|\h \psi_p(xy)=\psi_p(x)\psi_p(y)\})\\ \notag
\le & \sum_{p|Q_s} p^{-\Theta_{\bbg}(1)} \log p\le \log (Q_s^{\vare}),
\end{align}
where the first inequality is given by Lemma~\ref{l:SectionFarFromHomo} and the second inequality holds for $Q_s^{\vare}\gg_{\bbg} 1$. Inequality~(\ref{e:AverageWeightedBadPrimes}) implies that 
\[
\bbe\left(\log\left(\prod_{p|Q_s \psi_p(x)\psi_p(y)\psi_p(xy)^{-1}\neq 1} p\right)\right)\ge \log(Q_s^{1-\vare}).
\]

So there are $x,y\in \pi_{Q_s}(\Gamma)$ and $Q_s'|Q_s$ such that $Q_s'\ge Q_s^{1-\vare}$ and for any prime $p|Q_s'$ we have $z_p:=\psi_p(x)\psi_p(y)\psi_p(xy)^{-1}\neq 1$. Hence there is $a\in \prod_3 A$ such that for any $p|Q_s'$ we have that $\pi_{p^2}(a)$ belongs to $\gfr(\f_p)\setminus V_p$ where $\gfr(\f_p)$ is identified with $\pi_{p^2}(\Gamma[p])$. By our assumptions on $\bbg$, $\pi_p(\Gamma)$ is generated by its $p$-elements and $\gfr(\f_p)$ is a completely reducible $\pi_p(\Gamma)$-module with no non-zero invariant vector (the latter holds as $\bbg$ is perfect and its unipotent radical is abelian). Therefore by \cite[Corollary 31]{SGV} $\sum_{O_{\bbg}(1)}\pi_p(\Gamma)\cdot z_p$ is the $\pi_p(\Gamma)$-subspace $M_p/V_p$ of $\gfr(\f_p)/V_p$ that is generated by $z_p$.

In the group $\pi_{Q_s^2}(\Gamma)/\prod_{p|Q_s} V_p$, for any $\gamma\in \Gamma$, we have
\be\label{e:ConjugationAdjoint}
 \pi_{Q_s^2}(\gamma)\pi_{Q_s^2}(a)\pi_{Q_s^2}(\gamma)^{-1}\left(\prod_{p|Q_s}V_p\right)=(\pi_p(\gamma)\cdot z_p)_{p|Q_s}\in \prod_{p|Q_s}\gfr(\f_p)/V_p.
\ee
Since $\pi_{Q_s}(A)=\pi_{Q_s}(\Gamma)$ and $a\in \prod_3 A$, by (\ref{e:ConjugationAdjoint}) we have that  $\pi_{Q_s'^2}(\prod_{O_{\bbg}(1)} A)\left(\prod_{p|Q_s}V_p\right)$ (as a subset of $\pi_{Q_s^2}(\Gamma)/\prod_{p|Q_s} V_p$) contains 
\[
{\textstyle\sum_{O_{\bbg}(1)}}\pi_{Q_s'}(A)\cdot (z_p)_{p|Q_s'}=\prod_{p|Q_s'} {\textstyle\sum_{O_{\bbg}(1)}} \pi_{p}(\Gamma)\cdot z_p=\prod_{p|Q_s'} M_p/V_p.
\] 
(Here $\pi_{{Q_s}^2}(\Gamma)/\prod_{p|Q_s} V_p$ is viewed as $(\pi_{{Q_s'}^2}(\Gamma)/\prod_{p|Q_s'} V_p)\oplus (\pi_{{(Q_s/Q_s')}^2}(\Gamma)/\prod_{p|Q_s/Q_s'} V_p)$.) And so 
\[
\prod_{p|Q_s'} M_p\subseteq\pi_{Q_s^2}\textstyle(\prod_{O_{\bbg}(1)} A);
\]
and the claim follows. 
\end{proof}
{\em (Going back to the proof of Theorem~\ref{t:BoundedGeneration} for $\ccal_N$:)} By Equation~(\ref{e:FirstLevel}) and a repeated use of Lemma~\ref{l:Enlarging} (for $O_{\bbg}(1)$-many times), we get that there is $q|Q_s$ such that $q\le Q_s^{\Theta_{\bbg}(\vare)}$ (if $Q_s^{\vare}\gg_{\bbg} 1$) and 
\be\label{e:LargeLevel2CongruenceSubgroup}
\pi_{{Q_s}^2}({\textstyle\prod_{O_{\bbg}(1)}} A)\supseteq \pi_{{Q_s}^2}(\Gamma[q^2])=\prod_{p|(Q_s/q)} \pi_{p^2}(\Gamma)\times \prod_{p|q} \{1\}.
\ee

Let us recall that for large enough $p$ (depending only on $\Omega$) we have that, for any positive integer $i$, there is a $\pi_p(\Gamma)$-module isomorphism $\lin{p^i}{p^{i+1}}:\Gamma[p^i]/\Gamma[p^{i+1}]\rightarrow \gfr(\f_p)$; and for any $g_i\in \Gamma[p^i]$ and $g_j\in \Gamma[p^j]$ we have $[g_i,g_j]:=g_ig_jg_i^{-1}g_j^{-1}\in \Gamma[p^{i+j}]$ and 
\be\label{e:CommutatorsGraded}
\lin{p^{i+j}}{p^{i+j+1}}([g_i,g_j]\Gamma[p^{i+j+1}])
=[\lin{p^i}{p^{i+1}}(g_i\Gamma[p^{i+1}]),\lin{p^j}{p^{j+1}}(g_j\Gamma[p^{j+1}])],
\ee
where the right hand side is the Lie bracket in $\gfr(\f_p)$. (These maps are called the {\em truncated or finite logarithmic maps}. We refer the interested reader to \cite[Section 2.9]{SGsemisimple} for a general and quick treatment of these maps and their basic properties.) Equation (\ref{e:CommutatorsGraded}) defines a graded Lie algebra structure on 
\[
\Gamma[p]/\Gamma[p^2]\oplus \Gamma[p^2]/\Gamma[p^3] \oplus \cdots \oplus \Gamma[p^{N-1}]/\Gamma[p^{N}]
\]
and shows that it is isomorphic to the graded Lie algebra $\gfr(\f_p)\otimes (t\f_p[t]/t^N\f_p[t])$ where $\f_p[t]$ is the ring of polynomials in a single variable $t$, $[x_i\otimes \bar t^i,x_j\otimes \bar t^j]:=[x_i,x_j]\otimes \bar t^{i+j}$ for any $x_i,x_j\in \gfr(\f_p)$, and $\bar t:=t+t^N\f_p[t]$. Since, for large enough $p$ (depending only on $\Omega$), $\gfr(\f_p)$ is a perfect Lie algebra, $\gfr(\f_p)\otimes (t\f_p[t]/t^N\f_p[t])$ is generated by the degree one elements, that means $\gfr(\f_p)\otimes \bar t$. So by (\ref{e:LargeLevel2CongruenceSubgroup}) and (\ref{e:CommutatorsGraded}) we get that 
\be\label{e:NthLevel}
 \pi_{{(Q_s/q)}^N}({\textstyle\prod_{O_{\bbg,N}(1)}} A)=\pi_{(Q_s/q)^N}(\Gamma)=\prod_{p|(Q_s/q)} \pi_{p^N}(\Gamma).
\ee 
Hence there is a function $f_1:\pi_{(Q_s/q)^N}(\Gamma)\rightarrow \pi_{q^N}(\Gamma[q])$ such that $G(f_1):=\{(g,f_1(g))|g\in \pi_{(Q_s/q)^N}(\Gamma)\}$ is contained in $\pi_{Q_s^N}(\prod_{O_{\bbg,N}(1)} A)$. Let $w(x_1,x_2):=x_1x_2x_1^{-1}x_2^{-1}$ and for any subset $X$ of a group $w(X):=\{w(x_1,x_2)|\h x_1,x_2\in X\}$. Then $w(G(f_1))\subseteq \pi_{(Q_s/q)^N}(\Gamma)\times \pi_{q^N}(\Gamma[q^2])$. And so by Lemma~\ref{l:CommutatorCongruencePerfect} we have that there is a function $f_2:\pi_{(Q_s/q)^N}(\Gamma)\rightarrow \pi_{q^N}(\Gamma[q^2])$ such that its graph $G(f_2)$ is contained in $\pi_{Q_s^N}(\prod_{O_{\bbg,N}(1)} A)$. Repeating this argument $\log N$ times we get that $\pi_{Q_s^N}(\Gamma[q^N])$ is contained in $\pi_{Q_s^N}(\prod_{O_{\bbg,N}(1)} A)$, which finishes the proof of Theorem~\ref{t:BoundedGeneration} for $\ccal_{N}$.

\section{Super-approximation: the $p$-adic case.}
To prove Theorem~\ref{t:main} for $\ccal:=\{p^m|\h p\in V_f(\bbq), p\nmid q_0, m\in \bbz^+\}$, by Section~\ref{ss:ReductionToBG}, it is enough to Prove Theorem \ref{t:BoundedGeneration} for $\ccal$. 

In this section, we work in the setting of Section~\ref{ss:InitialReductions}. In addition, we assume that $\bbu$ is abelian. So we assume that $\bbv$ is a vector $\bbq$-group, $\bbg_s$ is a connected, simply-connected, semisimple $\bbq$-group, there is a $\bbq$-homomorphism $\bbg\rightarrow\mathbb{GL}(\bbv)$ with no non-zero fixed vector, and $\bbg=\bbg_s\ltimes \bbv$ is the Zariski-closure of $\Gamma=\langle \Omega\rangle$. We also fix a $\bbq$-embedding $\bbg\subseteq (\GL_{N_0})_{\bbq}$.

By Lemma~\ref{l:ToProveBGWeCanAssumeAS} (for $\ccal':=\ccal$), we can and will assume that $\Pfr_Q(\delta,A,l)$ and $\Pfr_Q(\delta,\rho_1(A),l)$ hold where $\rho_1:\bbg\rightarrow \bbg_s$ is the quotient map.

 It is worth pointing out that Theorem~\ref{t:BoundedGeneration} for semisimple groups and powers of primes is proved in~\cite[Theorem 36]{SGsemisimple}.      
\subsection{Escape from proper subgroups.}\label{ss:EscapeModQ}
In this section, we explore the pro-$p$ structure of an open subgroup  of the $p$-adic closure $\Gamma_p$ of $\Gamma$. The main goal is to escape proper subgroups of $\pi_Q(\Gamma)$ where $Q=p^n$ is a power of a prime $p$.
\begin{prop}\label{p:EscapeModQ}
In the above setting, there is a positive number $\delta$ (depending on $\Omega$) such that for any $n\gg_{\Omega} 1$ and any proper subgroup $H$ of $\pi_{p^n}(\Gamma)$ we have
\[ 
\pcal_{\pi_{p^n}(\Omega)}^{(l)}(H)\le [\pi_{p^n}(\Gamma):H]^{-\delta},
\]
for $l\ge \frac{n}{\delta}\log p$.
\end{prop} 
We start with a (well-known) lemma which gives us the Frattini subgroup of the congruence subgroups of $\gcal(\bbz_p)$, where (as before) $\gcal$ is the closure of $\bbg$ in $(\GL_{N_0})_{\bbz_S}$. This kind of result for semisimple groups and large $p$ goes back to Weisfeiler~\cite{Wei}.
\begin{lem}\label{l:Frattini}
Let $\gcal$ be as above, and $G:=\gcal(\bbz_p)$. Suppose either $p\gg_{\gcal} 1$ or $m_0\gg_{\gcal} 1$. If $\wt{H}$ is a closed subgroup of $G[p^{m_0}]$ and $\wt{H}G[p^{m_0+1}]=G[p^{m_0}]$, then $\wt{H}=G[p^{m_0}]$. And so the Frattini subgroup $\Phi(G[p^{m_0}])=G[p^{m_0+1}]$.
\end{lem} 
\begin{proof}
By induction on $k$ we prove  that $\wt{H}G[p^{m_0+k}]=G[p^{m_0}]$. The base of the induction is given by the assumption. To get the induction step, it is enough to prove that $(\wt{H}\cap G[p^{m_0+k}])G[p^{m_0+k+1}]=G[p^{m_0+k}]$. Let $\gfr:=\Lie(\gcal)(\bbz_p)$. By \cite[Lemma 34]{SGsemisimple}, we have $\lin{p^{m_0}}{p^{m_0+1}}:G[p^{m_0}]/G[p^{m_0+1}]\rightarrow \gfr/p\gfr$ is an isomorphism if either $p\gg_{\gcal} 1$ or $m_0\gg_{\gcal}1$ (see \cite[Section 2.9]{SGsemisimple} for the definition and properties of {\em finite logarithmic map} $\lin{p^{m_0}}{p^{m_0+1}}$). And so by the assumption, for any $x\in \gfr$ there is $h\in \wt{H}$ such that $\pi_{p^{m_0+1}}(h)=\pi_{p^{m_0+1}}(I+p^{m_0}x)$. On the other hand, if $m_0\ge 2$, then $\pi_p(((I+p^{m_0}x')^p-I)/p^{m_0+1})=\pi_p(x')$. And so $\linValue{p^{m_0+k}}{p^{m_0+k+1}}{h^{p^k}}=\pi_p(x)$. Since, by \cite[Lemma 34]{SGsemisimple}, $\lin{p^{m_0+k}}{p^{m_0+k+1}}:G[p^{m_0+k}]/G[p^{m_0+k+1}]\rightarrow \gfr/p\gfr$ is a bijection, we get the induction step. Since $\wt{H}$ is a closed subgroup, we have that $\wt{H}=G[p^{m_0}]$.
\end{proof}  

The following lemma which is a module theoretic version of ~\cite[Lemma 3.5]{SGChar} is proved next. It will be needed in the following sections, too.  

\begin{lem}\label{l:SimpleModulesOrder}
Let $R\subseteq M_{n_0}(\bbz_p)$ be a $\bbz_p$-subalgebra (not necessarily with the identity element). Suppose $\bbq_p^{n_0}$ is a simple $\bbq_p[R]$-module, where $\bbq_p[R]$ is the $\bbq_p$-span of $R$. Then 
\[
\sup_{v\in \bbz_p^{n_0}\setminus p\bbz_p^{n_0}} [\bbz_p^{n_0}:Rv]<\infty.
\]  
\end{lem}
\begin{proof}
Suppose to the contrary that there is a sequence of unit vectors $v_i$ such that $[\bbz_p^{n_0}:Rv_i]$ goes to infinity as $i$ goes to infinity. Since the set of unit vectors is a compact set, by passing to a subsequence we can assume that $\lim_{i\rightarrow \infty} v_i=v$ where $v$ is a unit vector. Since $\bbq_p^{n_0}$ is a simple $\bbq_p[R]$-module, the $\bbq_p$-span of $Rv$ is $\bbq_p^{n_0}$. And so $Rv$ is of finite-index in $\bbz_p^{n_0}$. Therefore $Rv$ contains $p^{k_0}\bbz_p^{n_0}$ for some positive integer $k_0$. If $i$ is large enough, $v-v_i\in p^{k_0+1}\bbz_p$. So $Rv_i+p^{k_0+1}\bbz_p^{n_0}=Rv+p^{k_0+1}\bbz_p^{n_0}\supseteq p^{k_0}\bbz_p^{n_0}$. Since $Rv_i$ is complete and the Frattini subgroup of $p^{k_0}\bbz_p^{n_0}$ is $p^{k_0+1}\bbz_p^{n_0}$, $Rv_i\supseteq p^{k_0}\bbz_p^{n_0}$ which is a contradiction. 
\end{proof}

The following lemma roughly shows that an almost invariant subspace of $\Lie(\bbg_s)(\bbq_p)$ is close to an invariant subspace.

\begin{lem}\label{l:AlmostInvariantSubspaceSemisimple}
Let $\bbg_s\subseteq (\GL_{N_0})_{\bbq}$ be (as before) a semisimple group. Let $G:=\bbg_s(\bbq_p)\cap \GL_{N_0}(\bbz_p)[p^{k_0}]$ for a fixed positive integer $k_0$. Suppose $m$ is large enough depending on the embedding of $\bbg_s$ in $(\GL_{N_0})_{\bbq}$   and $k_0$. Then the following holds: 

Suppose $W$ is a subspace of $\Lie(\bbg_s)(\bbq_p)$ such that $\pi_{p^{m}}(W\cap \gl_{N_0}(\bbz_p))$ is invariant under $G$. Then there is a  normal closed subgroup $\bbg_W$ of $\bbg_s$ and a positive integer $C$ which depends only on the embedding of $\bbg_s$ in $(\GL_{N_0})_{\bbq}$ and $k_0$ such that 
\be\label{e:CloseToInvariant}
p^{2C} \pi_{p^{m-C}}(\gfr_W)\subseteq 
p^C	   \pi_{p^{m-C}}(W\cap \gl_{N_0}(\bbz_p))\subseteq
\pi_{p^{m-C}}(\gfr_W ),
\ee	
where $\gfr_W:=\Lie(\bbg_W)(\bbq_p)\cap \gl_{N_0}(\bbz_p)$.
\end{lem}
\begin{proof}
Since $\bbg_s$ is semisimple, it is the almost product of $\bbq_p$-simple factors $\bbg_i$. And so $\Lie(\bbg)=\oplus_i \Lie(\bbg_i)(\bbq_p)$. Let $\gfr:=\Lie(\bbg)(\bbq_p)\cap \gl_{N_0}(\bbz_p)$ and $\gfr_i:=\Lie(\bbg_i)(\bbq_p)\cap \gl_{N_0}(\bbz_p)$. Then there is a positive number $c_1$ which depends only on the embedding of $\bbg_s$ in $(\GL_{N_0})_{\bbq}$ such that 
\be\label{e:ProjectionToSimpleFactors}
\bigoplus_i \gfr_i\subseteq \gfr \subseteq \bigoplus_i p^{-c_1}\gfr_i;
\ee
in particular, for any $x\in \Lie(\bbg)(\bbq_p)$, we have 
\be\label{e:ComparingNorms}
\|x\|_p\le \max_i \|x_i\|_p \le p^{c_1} \|x\|_p,
\ee
where $x=\sum_i x_i$ and $x_i\in\Lie(\bbg_i)(\bbq_p)$. 

Let $\Wfr:=W\cap \gl_{N_0}(\bbz_p)$. For $x\in \Wfr$, let $x_i\in \Lie(\bbg_i)(\bbq_p)$ be such that $x=\sum_i x_i$. For any $g_j\in \bbg_j(\bbq_p)\cap\GL_{N_0}(\bbz_p)[p^{k_0}]$, $\Ad(g_j)$ changes only the $j$-th component. So, using the assumption that $\pi_{p^m}(\Wfr)$ is invariant under $\Ad(g_j)$, we get that 
\be\label{e:OneComponent}
\Ad(g_j)x_j-x_j \in \Wfr+p^m\gl_{N_0}(\bbz_p).
\ee
Hence $\afr_j x_j\subseteq \Wfr+p^m\gl_{N_0}(\bbz_p)$ where $\afr_j$ is the $\bbz_p$-span of $\{\Ad(g_j)-I|\h g_j\in \bbg_j(\bbq_p)\cap\GL_{N_0}(\bbz_p)[p^{k_0}]\}$ (here we restrict to the action on $\Lie(\bbg_j)$, and consider $\afr_j$ as a subset of ${\rm End}_{\bbq_p}(\Lie\bbg_j(\bbq_p))$). Notice that $\afr_j$ is an ideal of the $\bbz_p$-span $R_j$ of $\Ad(\bbg_j(\bbq_p)\cap\GL_{N_0}(\bbz_p)[p^{k_0}])$. And $\Lie(\bbg_j)(\bbq_p)$ is a simple $\bbq_p[R_j]$-module, where $\bbq_p[R_j]$ is the $\bbq_p$-span of $R_j$. Hence $\bbq_p[R_j]$ is a simple $\bbq_p$-algebra as it has a simple faithful module. Therefore the $\bbq_p$-span $\bbq_p[\afr_j]$ of $\afr_j$ is $R_j$, and $\Lie(\bbg_j)(\bbq_p)$ is a simple $\bbq_p[\afr_j]$-module.  Hence, by Lemma 38, there is a positive integer $c_2$ which depends only on $k_0$ and the embedding of $\bbg_s$ in $(\GL_{N_0})_{\bbq}$ such that
\be\label{e:SingleFactor}
\bigoplus_j p^{c_2} \|x_j\|^{-1} \gfr_j\subseteq \bigoplus_j \afr_j x_j \subseteq \Wfr+p^m\gl_{N_0}(\bbz_p). 
\ee

Let $I_W:=\{j|\h \exists x\in W, \|x\|_p=1 \text{ and } \|x_j\|_p\ge 1, \text{ where } x=\sum_i x_i, \text{ and } x_i\in \Lie(\bbg_i)(\bbq_p)\},$ $\bbg_W$ be the product of the simple factors $\bbg_i$ for $i\in I_W$; and so $\gfr_W:=\left(\oplus_{i\in I_W}\Lie(\bbg_i)(\bbq_p)\right)\cap \gl_{N_0}(\bbz_p)$.

We can and will assume that $c_2\ge c_1$, and we claim that (\ref{e:CloseToInvariant}) holds for $C:=c_1+c_2$ and $\gfr_W$. We start with the {\em first inclusion}  in (\ref{e:CloseToInvariant}). 

For any $j\in I_W$, there is $x\in \Wfr$ such that its $j$-th simple component $x_j$ has norm at least 1. Therefore by (\ref{e:SingleFactor}) we have $p^{c_2}\gfr_j\subseteq \Wfr+p^m\gl_{N_0}(\bbz_p)$. And so by (\ref{e:ProjectionToSimpleFactors}) we have
\be\label{e:FirstInclusion}
p^{c_1+c_2} \gfr_W\subseteq \bigoplus_{j\in I_W} p^{c_2}\gfr_j \subseteq \Wfr+p^m\gl_{N_0}(\bbz_p).
\ee
To show the {\em second inclusion} in (\ref{e:CloseToInvariant}), we proceed by contradiction. So assume to the contrary that there is $x\in \Wfr$ such that $p^{C}\pi_{p^{m-C}}(x)\not\in \pi_{p^{m-C}}(\gfr_W)$. By (\ref{e:ProjectionToSimpleFactors}) there are $x_j\in \gfr_j$ such that $p^{c_1}x=\sum_j x_j$. By (\ref{e:FirstInclusion}), there are $x'\in \Wfr$ and $e\in \gfr$ such that 
\be\label{e:IWComponents}
p^{c_2}\sum_{i\in I_W}x_i=x'+p^me.
\ee
Hence 
\be\label{e:NewVector}
x'':=p^{c_1+c_2}x-x'=\sum_{j\not\in I_W}p^{c_2}x_j+p^m e=
\sum_{j\not\in I_W}(p^{c_2}x_j+p^m e_j)+\sum_{j\in I_W} p^m e_j,
\ee
where $e=\sum_i e_i$ and $e_i\in \Lie(\bbg_i)(\bbq_p)$, is contained in $\Wfr$.
Since $e\in \gfr$, we have $\|e\|_p\le 1$. Therefore, by (\ref{e:ComparingNorms}), we have 
\be\label{e:ErrorTerm}
\|p^me_i\|_p\le p^{-m+c_1}.
\ee
On the other hand, since $p^{C}\pi_{p^{m-C}}(x)\not\in \pi_{p^{m-C}}(\gfr_W)$ and $p^C x=p^{c_2}\sum_{i\in I_W} x_i+p^{c_2}\sum_{i\not\in I_W} x_i\in \gfr_W+p^{c_2}\sum_{i\not\in I_W} x_i$, we have  $\|p^{c_2}\sum_{j\not\in I_W} x_j\|_p>p^{-m+C}$, which implies that there is $j_0\not\in I_W$ such that
\be\label{e:OneLargeTerm}
\|p^{c_2}x_{j_0}\|_p> p^{-m+C}.
\ee
So by the inequalities given in (\ref{e:ErrorTerm}) and (\ref{e:OneLargeTerm}) we have
\be\label{e:NormLowerBound}
\|p^{c_2}x_{j_0}+p^me_{j_0}\|=\max\{\|p^{c_2}x_{j_0}\|_p,\|p^me_{j_0}\|_p\}>p^{-m+C}.
\ee
Therefore 
\begin{align*}
\|x''\|_p&\ge p^{-c_1} \|p^{c_2}x_{j_0}+p^me_{j_0}\|&\text{ (by (\ref{e:NewVector}) and (\ref{e:ComparingNorms})) }\\
&>p^{-c_1}p^{-m+C}=p^{-m+c_2}\ge p^{-m+c_1}& \text{ (by (\ref{e:NormLowerBound})) }\\
&\ge \|p^me_i\|_p.& \text{ (by (\ref{e:ErrorTerm})) }
\end{align*}
So after normalizing $x''$, we get a vector $\overline{x}\in W$ such that $\|\overline{x}\|_p=1$ and, for any $i\in I_W$, $\|\overline{x}_i\|_p<1$, where $\overline{x}=\sum_i\overline{x}_i$ and $\overline{x}_i\in \Lie(\bbg_i)(\bbq_p)$. This is contrary to the definition of $I_W$.

\end{proof}

\begin{lem}\label{l:AlmostInvariantSubspacePerfect}
As before let $\bbg=\bbg_s\ltimes \bbv\subseteq (\GL_{N_0})_{\bbq}$, where $\bbg_s$ is a semisimple connected simply-connected $\bbq$-group and $\bbv$ is a $\bbq$-vector group. Let $G:=\bbg(\bbq_p)\cap \GL_{N_0}(\bbz_p)[p^{k_0}]$ for a fixed positive integer $k_0$. Suppose $m$ is large enough depending on the embedding of $\bbg$ in $(\GL_{N_0})_{\bbq}$ and $k_0$. Then the following holds:

Suppose $W$ is a subspace of $\Lie(\bbg)(\bbq_p)$ such that the projection of $W$ onto $\Lie(\bbg_s)(\bbq_p)$ is onto and $\pi_{p^m}(W\cap \gl_{N_0}(\bbz_p))$ is invariant under $G$. Then $W=\Lie(\bbg)(\bbq_p)$.    
\end{lem}
\begin{proof}
Without loss of generality we can and will assume that $\bbv$ is non-zero. Suppose $\bbg_s$ acts on $\bbv$ via $\rho: \bbg_s\rightarrow \GL(\bbv)$; that means for any $g\in\bbg_s(R)$ and $v\in \bbv(R)$ we have $(g,0)(1,v)(g,0)^{-1}:=(1,\rho(g)(v))$ in $\bbg(R)$ for any $\bbq$-algebra $R$. To get a clear computation of the adjoint action, we use the {\em dual numbers} $\bbq_p[\vare]:=\bbq_p\oplus \bbq_p\vare$ where $\vare^2=0$ (this lemma is the only place where $\vare$ is {\em not} a real number, and it is an element of the dual numbers.). Recall that we can
 identify $\Lie(\bbg)(\bbq_p)$ with 
\[
\ker\left(\bbg_s(\bbq_p[\vare])\ltimes \bbv(\bbq_p[\vare]) \xrightarrow{\pi_{\vare}} \bbg_s(\bbq_p)\ltimes \bbv_p(\bbq_p)\right)
\] 
where $\pi_{\vare}$ is the group homomorphism induced by the ring homomorphism $\pi_{\vare}:\bbq_p[\vare]\rightarrow \bbq_p$. Under this identification the adjoint action is given by conjugation; that means, for $x\in \Lie(\bbg)(\bbq_p)$, we have $1+\vare x\in \bbg(\bbq_p[\vare])$; and for $g\in \bbg(\bbq_p)$, we have $1+\vare \Ad(g)(x)=g(1+\vare x)g^{-1}$. 

The identity element of $\bbg=\bbg_s\ltimes \bbv$ is $(1,0)$. And so $(x,w)\in \Lie(\bbg)(\bbq_p)$ if and only if $(1+\vare x,\vare w)\in \bbg(\bbq_p[\vare])$; and to compute $\Ad(v)(x,w)$ we have to compute
$(1,v)(1+\vare x,\vare w)(1,v)^{-1}$ in $\bbg(\bbq_p[\vare])$. We have
\begin{align*}
(1,v) (1+\vare x, \vare w) (1,-v)&= (1,v)(1+\vare x,0) (1,\vare w-v)\\ 
&=(1+\vare x, \rho(1-\vare x)(v)-v+\vare w)= (1+\vare x, \vare(d\rho(x)(v)+w)),
\end{align*}
which means $\Ad(1,v)(x,w)=(x,d\rho(x)(v)+w)$.
Notice that we are slightly abusing the notation and use the addition for the group operation of $\bbv$ though it is realized as a subgroup of $(\GL_{N_0})_{\bbq}$. By this computation, for any $(x,w)\in W$ and $v\in \bbv(\bbq_p)\cap \GL_{N_0}(\bbz_p)$, we have
\[
p^{k_0}d\rho(x)(v):=(0,p^{k_0}d\rho(x)(v))=\Ad(p^{k_0} v)(x,w)-(x,w)\in (W\cap \gl_{N_0}(\bbz_p))+p^m \gl_{N_0}(\bbz_p).
\]
Thus we have 
\[
p^{k_0}d\rho\left(\Lie(\bbg_s)(\bbq_p)\cap \gl_{N_0}(\bbz_p)\right)\left(\bbv(\bbq_p)\cap \gl_{N_0}(\bbz_p)\right) \subseteq (W\cap \gl_{N_0}(\bbz_p))+p^m \gl_{N_0}(\bbz_p).
\]
 Since $\bbv$ is a completely reducible $\bbg_s$ module with no trivial factors, we have 
 \[
 d\rho\left(\Lie(\bbg_s)(\bbq_p)\cap \gl_{N_0}(\bbz_p)\right)\left(\bbv(\bbq_p)\cap \gl_{N_0}(\bbz_p)\right) \supseteq p^{O(1)} (\bbv(\bbq_p)\cap \gl_{N_0}(\bbz_p)),
 \]
 where the implied constant depends on $\bbg$ and its embedding into $(\GL_{N_0})_{\bbq}$. Therefore $\bbv(\bbq_p)\subseteq W$; and the claim follows.
\end{proof}
The next lemma is crucial in proving the bounded generation claim and its proof is fairly involved. By the strong-approximation, we know that the $S$-arithmetic group $\Gamma_0:=\bbg(\bbq)\cap\GL_{N_0}(\bbz_S)$, when infinite, is dense in $G_p:=\bbg(\bbq_p)\cap\GL_{N_0}(\bbz_p)$ for any prime $p$ which is not in $S$ (here $\bbg$ is as before a simply-connected, connected, perfect $\bbq$-group.). And so for any open subgroup $\wt{H}$ of $G_p$, $\Gamma_0\cap \wt{H}$ is dense in $\wt{H}$. The next lemma shows that elements of $\Gamma_0\cap \wt{H}$ with the $S$-norm at most $[G_p:\wt{H}]^{O_{\Gamma_0}(1)}$ are, however, trapped within a proper algebraic subgroup of $\bbg$. 

Let $\nu_{\infty}$ be the Archimedean place of $\bbq$, and $\wt{S}:=\{\nu_{\infty}\}\cup S$. Notice that $\bbz_S$ can be diagonally embedded in $\prod_{\nu\in \wt{S}}\bbq_{\nu}$ as a discrete subgroup. Let $\|\xbf\|_{\nu_{\infty}}:=\sqrt{\sum_i x_i^2}$ for $\xbf=(x_1,\ldots,x_{N_0})\in\bbr^{N_0}$, and $\|\xbf\|_p:=\max\{|x_i|_p\}_i$ for $\xbf=(x_1,\ldots,x_{N_0})\in \bbq_p^{N_0}$. For any $(\xbf_{\nu})\in \prod_{\nu\in\wt{S}}\bbq_{\nu}^{N_0}$, let $\|(\xbf_{\nu})\|_{\wt{S}}:=\max_{\nu\in\wt{S}} \|\xbf_{\nu}\|_{\nu}$. For any $\lambda\in \GL_{N_0}(\bbz_S)$, we let $\|\lambda\|_S$ be the operator norm of $\lambda:\prod_{\wt{S}}\bbq_{\nu}^{N_0}\rightarrow \prod_{\wt{S}}\bbq_{\nu}^{N_0}$.
\begin{lem}\label{l:SmallLifts}
As before, let $\bbg=\bbg_s\ltimes \bbv\subseteq (\GL_{N_0})_{\bbq}$, where $\bbg_s$ is a semisimple, connected, simply-connected $\bbq$-group and $\bbv$ is a $\bbq$-vector group. Let $\gcal$ be the Zariski-closure of $\bbg$ in $(\GL_{N_0})_{\bbz}$; in particular $\gcal(\bbz_S)=\bbg(\bbq)\cap\GL_{N_0}(\bbz_S)$ for any finite set of primes $S$. Let $S$ be a finite set of primes such that $\gcal(\bbz_S)$ is infinite. 
Suppose either $p\gg_{\gcal} 1$ or $m_0\gg_{\gcal} 1$. Let $G:=\gcal(\bbz_p)$. Then there is a positive number $\delta$ such that for any prime $p\not\in S$ and an open subgroup $\wt{H}$ of $G[p^{m_0}]$ the set
\[
\lcal_{\delta}(\wt{H}):=\{\lambda\in \gcal(\bbz_S)\cap \wt{H}|\h \|\lambda\|_S\le [\gcal(\bbz_p):\wt{H}]^{\delta}\}
\]
is in a proper Zariski-closed subgroup of $\bbg$.
\end{lem}
\begin{proof}
Since $\GL_{N_0}(\bbz_S)$ is a discrete set with respect to the $S$-norm, we can and will assume that $[\gcal(\bbz_p):\wt{H}]$ is large enough. Otherwise choosing $\delta$ small enough, we have that $\lcal_{\delta}(\wt{H})$ is a subset of a finite subgroup and we are done.

Suppose $p$ or $m_0$ are large enough so that Lemma~\ref{l:Frattini} holds. Let $G:=\gcal(\bbz_p)$, $\gfr:=\Lie(\gcal)(\bbz_p)$, and 
\[
l(\wt{H}):=\min\{k\in \bbz^+|\h \wt{H}\supseteq G[p^k]\}.
\]
 Notice that, since $\wt{H}$ is an open subgroup, $l(\wt{H})$ is finite. By Lemma~\ref{l:Frattini}, for any $m_0\le k\le l(\wt{H})-1$, we have that $(\wt{H}\cap G[p^k]) G[p^{k+1}]/G[p^{k+1}]$ is a proper $\f_p$-subspace of $G[p^k]/G[p^{k+1}]\simeq \gfr/p\gfr$. Thus we have 
\be\label{e:LevelIndex}
p^{l(\wt{H})-m_0}\le [G[p^{m_0}]:\wt{H}] \le p^{d (l(\wt{H})-m_0)},
\ee
where $d=\dim_{\f_p}\gfr/p\gfr$. 

Our first goal is to find a {\em linear separation} of an $\Ad(\wt{H})$-orbit with a {\em good margin} (see Claim 1) when $l(\wt{H})$ is large enough depending only on $\gcal$. To achieve this goal, we start with finding  {\em primitive} $\bbz_p$-submodules of $\gfr$ and $\gfr_s:=\Lie(\bbg_s)(\bbq_p)\cap \gfr$ that are {\em almost} $\Ad(\wt{H})$-invariant (see (\ref{e:H-Invariance}) and (\ref{e:PrimitiveSemisimpleInvariant})).

Notice that, since $\bbv$ is a direct sum of non-trivial simple $\bbg_s$-modules, for any prime $p$, $\bbg(\bbq_p)=\bbg_s(\bbq_p)\ltimes \bbv(\bbq_p)$ is compactly generated. Hence, by a result of Kneser~\cite{Kne},
$\gcal(\bbz_S)$ is a finitely generated group. We fix a finite generating set $\Omega'$ of $\gcal(\bbz_S)$.

{\bf Claim 1.} For any constant $C_0>1$ (later it will depend on $\dim \bbg$), there is $c(\gcal)$ such that, if $l(\wt{H})\gg_{\gcal,C_0} 1$, then  
for some $w\in \gfr$ and $L\in \gfr^*:=\Hom_{\bbz_p}(\gfr,\bbz_p)$ the following holds
\begin{enumerate}
	\item for any $h\in \wt{H}$,  $|L(\Ad(h)(w))|_p\le [\gcal(\bbz_p):\wt{H}]^{-c}$;
	\item for some $\gamma\in \Omega'$, $|L(\Ad(\gamma)(w))|_p\ge [\gcal(\bbz_p):\wt{H}]^{-c/C_0}$.
\end{enumerate}
\begin{proof}[Proof of Claim 1]
Let 
\[
\overline{V}:=\lin{p^{l'}}{p^{l}}((\wt{H}\cap G[p^{l'}])G[p^l]/G[p^l])\subseteq \gfr/p^{l-l'}\gfr,
\]
 where $l:=l(\wt{H})$ and $l':=\lceil l/2\rceil+m_0$ (Recall that $\lin{p^{l'}}{p^{l}}(1+p^{l'}x):=\pi_{p^{l-l'}}(x)$ is a {\em finite logarithmic map}.). So there are a $\bbz_p$-basis $\{\ebf_1,\ldots,\ebf_d\}$ of $\gfr$ and positive integers $n_1\le \cdots 
 \le n_{d'}$ such that $\overline{V}=\pi_{p^{l-l'}}(\sum_{i=1}^{d'} p^{n_i} \bbz_p \ebf_i)$ (for some $d'\le d$). Since
 \[
 0=n_0\le n_1\le \cdots \le n_{d'} \le \lfloor l/2 \rfloor-m_0=: n_{d'+1},
 \]
 there is $i_0$ such that 
 \be\label{e:LargeJump}
 (1/2d) l-m_0-1< n_{i_0}-n_{i_0-1}\le (1/2) l-m_0.
 \ee
  Let $\Wfr_1:=\sum_{i=1}^{i_0-1} p^{n_i} \bbz_p \ebf_i\subset \gfr$ (if $i_0=1$, then $\Wfr_1=0$). So
 \be\label{e:H-Invariant-Initial}
 \Ad(\wt{H})(\Wfr_1)\subseteq \Wfr_1+p^{n_{i_0}}\gfr.
 \ee
 Let $\Wfr_2:=\sum_{i=1}^{i_0-1}\bbz_p \ebf_i$. Notice that $\Wfr_2=W_1\cap \gfr$ where $W_1=\sum_{i=1}^{i_0-1}\bbq_p \ebf_i$ is the $\bbq_p$-span of $\Wfr_1$. Since $\pi_{p^{n_{i_0}}}(\Wfr_1)$ is $\wt{H}$-invariant, for any $1\le i\le i_0-1$ and $h\in\wt{H}$ we have
 \[
 \Ad(h)(p^{n_i}\ebf_i)=\sum_{j=1}^{i_0-1}c_jp^{n_j}\ebf_j+\sum_{j=1}^{d}p^{n_{i_0}}c_{j}'\ebf_j,
 \]
 where $c_j,c_j'\in \bbz_p$. Therefore
 \[
 \Ad(h)(\ebf_i)=\sum_{j=1}^{i_0-1}(c_j p^{n_j-n_i}+p^{n_{i_0}-n_i}c_j')\ebf_j+\sum_{j=i_0}^{d}p^{n_{i_0}-n_i}c_j'\ebf_j.
 \]
 
 Since $\gfr$ is $\wt{H}$-invariant, we have $c_j p^{n_j-n_i}+p^{n_{i_0}-n_i}c_j'\in \bbz_p$. Since $n_j$ is increasing, we have $p^{n_{i_0}-n_i}\ge p^{n_{i_0}-n_{i_0-1}}$ for any $i<i_0$. Hence we have
 \be\label{e:H-Invariance}
 \Ad(\wt{H})(\Wfr_2)\subseteq \Wfr_2+p^{n_{i_0}-n_{i_0-1}}\gfr.
 \ee
 Next we would like to get some invariance in the semisimple part of $\gfr$. We notice that $\Lie(\bbg)(\bbq_p)=\Lie(\bbg_s)(\bbq_p)\oplus \bbv(\bbq_p)$ (we are identifying $\bbv$ with its Lie algebra.). For any $x\in \Lie(\bbg)(\bbq_p)$, we write its components in the above decomposition by $x_s$ and $x_v$; that means $x=x_s+x_v$ where $x_s\in \Lie(\bbg_s)(\bbq_p)$ and $\Lie(\bbv)(\bbq_p)$. There is a positive integer $d_1$ depending only on $\gcal$ such that 
 \be\label{e:sv-component}
 \|x\|_p \le \max\{\|x_s\|_p,\|x_v\|_p\}\le p^{d_1} \|x\|_p\hspace{1cm}\text{ and }\hspace{1cm} \gfr_s\oplus \Vfr \subseteq \gfr\subseteq p^{-d_1}(\gfr_s\oplus \Vfr),
 \ee
 where $\gfr_s:=\Lie(\bbg_s)(\bbq_p)\cap \gfr$ and $\Vfr:=\bbv(\bbq_p)\cap \gfr$.  Let $\Wfr_{2,s}:=\{x_s|\h x\in \Wfr_2\}$ be the projection to $\Lie(\bbg_s)(\bbq_p)$ of $\Wfr_2$. And so, by (\ref{e:H-Invariance}) and the fact that $\bbv$ is a normal subgroup of $\bbg$, we have
 \be\label{e:SemisimpleInvariance}
 \Ad(\wt{H})(\Wfr_{2,s})\subseteq \Wfr_{2,s}+\bbv(\bbq_p)+p^{n_{i_0}-n_{i_0-1}}\gfr.
 \ee
 By (\ref{e:sv-component}), $p^{d_1}\Wfr_{2,s}$ is a $\bbz_p$-submodule of $\gfr_s$. Hence $\gfr_s$ has a $\bbz_p$-basis $\fbf_1,\ldots, \fbf_{d_s}$ such that 
 \be\label{e:BasisSemisimpleProjection}
 \Wfr_{2,s}=\sum_{i=1}^{d_s'}p^{n_i'}\bbz_p\fbf_i
 \ee
 for some integers $-d_1=:n_0'\le n_1'\le \cdots\le n'_{d_s'}< n'_{d_s'+1}:=\infty$. Let $j_0$ be the smallest index such that 
 \be\label{e:Jump}
 n_{j_0+1}'-n_{j_0}'\ge \frac{n_{i_0}-n_{i_0-1}-d_1}{2d_s}. 
 \ee
 (Notice that, since $n_0':=0$ and $n_{d_s'+1}':=\infty$, there is such $j_0$.) In particular, we have
 \be\label{e:UpperBoundwhereJumpHappens}
 n_{j_0}'< \frac{n_{i_0}-n_{i_0-1}-d_1}{2}.  
 \ee
 
 Let $\Wfr_{3,s}:=\sum_{i=1}^{j_0}\bbz_p \fbf_i$. Notice that $\Wfr_{3,s}=W_{3,s}\cap \gfr_s$ where $W_{3,s}=\sum_{i=1}^{j_0}\bbq_p \fbf_i$. So by (\ref{e:SemisimpleInvariance}) and (\ref{e:sv-component}) for any $1\le i\le j_0$ and $h\in\wt{H}$, we have
 \[
 \Ad(h)(p^{n_i'}\fbf_i)=\sum_{j=1}^{d_s'} p^{n_j'}c_{j,s}\fbf_j+p^{n_{i_0}-n_{i_0-1}}\sum_{j=1}^{d_s}p^{-d_1}c_{j,s}'\fbf_j+v_i,
 \] 
 where $c_{j,s}, c_{j,s}'\in \bbz_p$ and $v_i\in \bbv(\bbq_p)$. Therefore
 \be\label{e:RestrictionActionComputation}
 (\Ad(h)(\fbf_i))_s=\sum_{j=1}^{d_s'}(p^{n_j'-n_i'}c_{j,s}+p^{n_{i_0}-n_{i_0-1}-d_1-n_i'}c_{j,s}')\fbf_j+\sum_{j=d_s'}^{d_s}p^{n_{i_0}-n_{i_0-1}-d_1-n_i'}c_{j,s}'\fbf_j.
 \ee
  Since, for any $x_s\in \gfr_s$ and $h\in \wt{H}$, $(\Ad(h)(x_s))_s\in \gfr_s$, we have $p^{n_j'-n_i'}c_{j,s}+p^{n_{i_0}-n_{i_0-1}-d_1-n_i'}c_{j,s}' \in \bbz_p$ and $p^{n_{i_0}-n_{i_0-1}-d_1-n_i'}c_{j,s}'\in \bbz_p$. By (\ref{e:UpperBoundwhereJumpHappens}) we have 
  \be\label{e:PartOfError}
  n_{i_0}-n_{i_0-1}-d_1-n_i'> \frac{n_{i_0}-n_{i_0-1}-d_1}{2}.
  \ee
  By (\ref{e:Jump}), for $j\ge j_0+1$ and $i\le j_0$, we have
  \be\label{e:AnotherPartOfError}
  n_j'-n_i'\ge \frac{n_{i_0}-n_{i_0-1}-d_1}{2d_s}.
  \ee
 By (\ref{e:RestrictionActionComputation}), (\ref{e:PartOfError}), and (\ref{e:AnotherPartOfError}), we get
 \be\label{e:PrimitiveSemisimpleInvariant}
 \Ad(\wt{H})(\Wfr_{3,s})\subseteq \Wfr_{3,s}+\bbv(\bbq_p)+p^{\lfloor(n_{i_0}-n_{i_0-1}-d_1)/(2d_s)\rfloor}\gfr_s.
 \ee
 Having (\ref{e:H-Invariance}) and (\ref{e:PrimitiveSemisimpleInvariant}), we are ready to find the desired {\em linear separation} of an $\Ad(\wt{H})$-orbit. We do this considering various cases.
 
 {\bf Case 1.} {\em Suppose
 \be\label{e:Notinvariant}
 \Ad(\gcal(\bbz_S))(\Wfr_{3,s})\not\subseteq \Wfr_{3,s}+\bbv(\bbq_p)+p^{\lfloor(n_{i_0}-n_{i_0-1}-d_1)/(4d_sC_0)\rfloor}\gfr_s.
 \ee
}

By (\ref{e:Notinvariant}), we have that there is $\lambda\in \Omega'$ and $1\le i\le j_0$ such that
\be\label{e:LargeTransverse}
(\Ad(\lambda)(\fbf_i))_s\not\in \Wfr_{3,s}+p^{\lfloor(n_{i_0}-n_{i_0-1}-d_1)/(4d_sC_0)\rfloor}\gfr_s.
\ee
For $1\le i\le j_0$, let $\fbf_i^\ast:\gfr \rightarrow \bbz_p$ be a $\bbz_p$-linear map such that 
\[
\fbf_i^{\ast}(\fbf_i)=p^{d_1},\h
\fbf_i^{\ast}(\fbf_j)=0 \text{ for $j\neq i$, and }\h  
\fbf_i^{\ast}(x)=\fbf_i^{\ast}(x_s).
\]
(Here $x_s$ is the $\Lie(\bbg_s)(\bbq_p)$-component of $x$; and notice that because of (\ref{e:sv-component}) the image of $\fbf_i^{\ast}$ is a subset of $\bbz_p$.) By (\ref{e:LargeTransverse}), we have that there is $j>j_0$ such that 
\be\label{e:LargeMarginCase1}
|\fbf_j^{\ast}(\Ad(\lambda)(\fbf_i))|_p> p^{-\lfloor(n_{i_0}-n_{i_0-1}-d_1)/(4d_sC_0)\rfloor}.
\ee
And, by (\ref{e:PrimitiveSemisimpleInvariant}) and $j>j_0\ge i$, for any $h\in \wt{H}$, we have 
\be\label{e:SeparationCase1}
|\fbf_j^{\ast}(\Ad(h)(\fbf_i)|_p\le p^{-\lfloor(n_{i_0}-n_{i_0-1}-d_1)/(2d_s)\rfloor}.
\ee
So by (\ref{e:LargeJump}), for $l(\wt{H})\gg_{\gcal,C_0} 1$, we get that (\ref{e:LargeMarginCase1}) and (\ref{e:SeparationCase1}) imply $w:=\fbf_i$ and $L:=\fbf_j^{\ast}$ satisfy the conditions of Claim 1.

{\bf Case 2.} {\em  Suppose $\Lie(\bbg_s)(\bbq_p)$ is the $\bbq_p$-span $W_{3,s}$ of $\Wfr_{3,s}$.}

Let us recall that $W_1$ is the $\bbq_p$-span of $\Wfr_1$. And $\Wfr_{3,s}$ is a subset of the projection of $W_1$ to $\Lie(\bbg_s)(\bbq_p)$. So by the assumption of Case 2, we have that the projection of $W_1$ onto $\Lie(\bbg_s)(\bbq_p)$ is onto. Since $W_1$ is a proper subspace of $\Lie(\bbg)(\bbq_p)$, by Lemma~\ref{l:AlmostInvariantSubspacePerfect} and the strong approximation there is $c_3\ll_{\gcal} 1$ such that $\pi_{p^{c_3}}(W_1\cap \gfr)$ is not $\gcal(\bbz_S)$-invariant. Recall that $\Wfr_2=W_1\cap \gfr=\sum_{i=1}^{i_0}\bbz_p\ebf_i$. So there is $\lambda\in \Omega'$ and $1\le i\le i_0$ such that
\be\label{e:LargeTransverseCase2}
\Ad(\lambda)(\ebf_i)\not\in \Wfr_2+p^{c_3}\gfr.
\ee
Let $\{\ebf_1^{\ast},\ldots,\ebf_{d}^{\ast}\}$ be the dual of $\{\ebf_1,\ldots,\ebf_d\}$. Then, by (\ref{e:LargeTransverseCase2}) we have that there is $j\ge i_0$ such that
\be\label{e:LargeMarginCase2}
|\ebf_j^{\ast}(\Ad(\lambda)(\ebf_i))|_p>p^{-c_3}.
\ee 
And by (\ref{e:H-Invariance}) and $j\ge i_0>i$, for any $h\in \wt{H}$, we have
\be\label{e:SeparationCase2}
|\ebf_j^{\ast}(\Ad(h)(\ebf_i))|_p\le p^{-(n_{i_0}-n_{i_0-1})}.
\ee
So by (\ref{e:LargeJump}), for $l(\wt{H})\gg_{\gcal,C_0} 1$, we get that (\ref{e:LargeMarginCase2}) and (\ref{e:SeparationCase2}) imply $w:=\ebf_i$ and $L:=\ebf_j^{\ast}$ satisfy the conditions of Claim 1.

{\bf Case 3.} {\em Suppose  
\be\label{e:invariant}
\Ad(\gcal(\bbz_S))(\Wfr_{3,s})\subseteq \Wfr_{3,s}+\bbv(\bbq_p)+p^{\lfloor(n_{i_0}-n_{i_0-1}-d_1)/(4d_sC_0)\rfloor}\gfr_s,
\ee
and $\Lie(\bbg_s)(\bbq_p)$ is not the $\bbq_p$-span $W_{3,s}$ of $\Wfr_{3,s}$.}

By the strong approximation, the projection $K_p$ of the closure of $\gcal(\bbz_S)$ to $\bbg_s(\bbq_p)$ contains $\bbg_s(\bbq_p)\cap \GL_{N_0}(\bbz_p)[p^{k_0}]$ where $k_0$ depends only on $\gcal$. So by Lemma~\ref{l:AlmostInvariantSubspaceSemisimple} and (\ref{e:invariant}), there are a positive integer $c_4$, which depends only on $\gcal$, and a normal subgroup $\bbg_{W_{3,s}}$ of $\bbg_s$ such that
\be\label{e:AlmostToExact}
p^{2c_4} \pi_{p^{m-c_4}}(\gfr_{W_{3,s}})\subseteq 
p^{c_4}	   \pi_{p^{m-c_4}}(\Wfr_{3,s})\subseteq
\pi_{p^{m-c_4}}(\gfr_{W_{3,s}}),
\ee	
where $\gfr_{W_{3,s}}:=\Lie(\bbg_{W_{3,s}})(\bbq_p)\cap \gfr$ and $m:=\lfloor(n_{i_0}-n_{i_0-1}-d_1)/(4d_sC_0)\rfloor$.
  
Since $W_{3,s}$ is a proper subspace of $\Lie(\bbg_s)(\bbq_p)$, by the first inclusion in (\ref{e:AlmostToExact}) we have that $\bbg_{W_{3,s}}$ is a proper normal subgroup of $\bbg_s$. So at least one simple factor $\bbg_{i_0'}$ is missing in $\Lie(\bbg_{W_{3,s}})(\bbq_p)$. Let $\gfr_{i_0'}:=\Lie(\bbg_{i_0'})(\bbq_p)\cap \gfr$. Then we have that 
$\gfr_{i_0'}\cap \gfr_{W_{3,s}}=\{0\}$ and $[\gfr_{i_0'},\gfr_{W_{3,s}}]=\{0\}$. Let  $x_0\in \gfr_{i_0'}$ be a unit vector. Next we show that $\Wfr_{3,s}+\bbz_p x_0+\bbv(\bbq_p)$ is {\em almost} $\Ad(\wt{H})$-invariant. 

Let $h\in \wt{H}$ and suppose $h\in G[p^k]\setminus G[p^{k+1}]$ for some $k< l':=\lceil l/2\rceil+m_0$. Then 
\[
\log h \in (p^k \gl_{N_0}(\bbz_p)\cap \gfr)\setminus p^{k+1} \gl_{N_0}(\bbz_p).
\]
So 
\[
\log h^{p^{l'-k}}=p^{l'-k} \log h \in p^{l'} \gl_{N_0}(\bbz_p) \cap \log \wt{H}.  
\]
Hence $\pi_{p^{n_{i_0}}}(p^{-k} \log h)\in \pi_{p^{n_{i_0}}}(\Wfr_1)$. This implies that there is $y_h\in \Wfr_{2,s}$ such that 
\[
p^{-k}\log h\in y_h+p^{n_{i_0}}\gfr+\bbv(\bbq_p).
\]
 And so 
$
y_h=\sum_{i=1}^{d_s'} a_i p^{n_i'}\fbf_i,
$ 
for some $a_i\in \bbz_p$, which implies
\be\label{e:ComputationOfLog}
p^{-k}\log h\in \Wfr_{3,s}+p^{\min\{n_{i_0},n_{i_0'+1}'\}}\gfr+\bbv(\bbq_p).
\ee
By (\ref{e:Jump}), we have 
\be\label{e:LogsErrorTerm}
\min\{n_{i_0},n_{i_0'+1}'\}\ge \frac{n_{i_0}-n_{i_0-1}-d_1}{2d_s}-d_1\ge 2C_0m-d_1. 
\ee
 Hence by (\ref{e:ComputationOfLog}) and (\ref{e:LogsErrorTerm}), we have
\be\label{e:LogFinalForm}
p^{-k}\log h\in \Wfr_{3,s}+p^{2C_0m-d_1}\gfr+\bbv(\bbq_p).
\ee 
By (\ref{e:AlmostToExact}) and (\ref{e:LogFinalForm}), we have
\be\label{e:LogInLieIdeal}
p^{c_4-k}\log h\in \gfr_{W_{3,s}}+p^{\min\{2C_0m-d_1,m-c_4\}}\gfr+\bbv(\bbq_p).
\ee
By (\ref{e:LogInLieIdeal}) and the fact that $[\gfr_{W_{3,s}},x_0]=\{0\}$, we have
\be\label{e:SmallCommutator}
\ad(\log h)(x_0)\in p^{\min\{2C_0m-d_1,m-c_4\}-c_4}\gfr+\bbv(\bbq_p).
\ee
Thus by $m=\lfloor(n_{i_0}-n_{i_0-1}-d_1)/(4d_sC_0)\rfloor$, (\ref{e:LargeJump}), and (\ref{e:SmallCommutator}), for $l(\wt{H})\gg_{\gcal,C_0} 1$ we have
\be\label{e:ExtraInvariance}
\Ad(h)(x_0)=\exp(\ad(\log h))(x_0)\in x_0+p^{\min\{2C_0m-d_1,m-c_4\}-c_4}\gfr+ \bbv(\bbq_p).
\ee
So by (\ref{e:PrimitiveSemisimpleInvariant}) and (\ref{e:ExtraInvariance}), we have
\be\label{e:LargerPrimitiveInvariance}
\Ad(\wt{H})(\Wfr_{3,s}+\bbz_p x_0)\subseteq \Wfr_{3,s}+\bbz_p x_0+ p^{\min\{2C_0m-d_1,m-c_4\}-c_4}\gfr+\bbv(\bbq_p).
\ee
On the other hand, looking at the action of $\gcal(\bbz_S)$ on $\gfr_{i_0'}$, we can choose $x_0\in \gfr_{i_0'}$, $L_0\in \Hom_{\bbz_p}(\gfr_{i_0'},\bbz_p)$, and a positive integer $c_5$ which depends only on $\gcal$ such that for some $\lambda\in \Omega'$ we have
\be\label{e:BeforeSeparation}
|L_0((\Ad(\lambda)(x_0))_s)|_p\ge p^{-c_5},\text{ and }\h L_0(x_0)=0.
\ee
 Now let $\Lie(\bbg_s)(\bbq_p)=\oplus_i \Lie(\bbg_{s,i})(\bbq_p)$ be the decomposition of $\Lie(\bbg_s)(\bbq_p)$ to its simple factors; and for any $x\in \Lie(\bbg)(\bbq_p)$ let $x=x_v+\sum_i x_{s,i}$ be the such that $x_v\in \Lie(\bbv)(\bbq_p)$ and $x_{s,i}\in \Lie(\bbg_{s,i})(\bbq_p)$. Then for some positive integer $c_6$ which depends only on $\gcal$ we have
\[
p^{c_6}\|x\|_p\ge \max\{\|x_{s,i}\|_p\}_i.
\]
So there is a well-defined $\bbz_p$-linear map $L:\gfr\rightarrow \bbz_p, L(x):=L_0(p^{c_6}x_{s,i_0'})$. We notice that 
\be\label{e:LargeMarginCase3}
|L(\Ad(\lambda)(x_0))|_p\ge p^{-c_5},\text{ and }\h L(\Wfr_{3,s}+\bbz_p x_0)=0.
\ee
And so by (\ref{e:LargerPrimitiveInvariance}), for any $h\in \wt{H}$, we have
\be\label{e:SeparationCase3}
|L(\Ad(h)(x_0))|_p\le p^{-\min\{2C_0m-d_1,m-c_4\}-c_4}.
\ee
So by (\ref{e:LargeJump}), for $l(\wt{H})\gg_{\gcal,C_0} 1$, we get that (\ref{e:LargeMarginCase3}) and (\ref{e:SeparationCase3}) imply $w:=x_0$ and $L$ satisfy the conditions of Claim 1.
\end{proof}
Next we will show elements with {\em small height} in $\wt{H}$ are in a proper quadratic subvariety. Let $\ugfr:=\Lie \gcal$. So there are finitely many $\bbz_S$-linear maps $f_i$ (for $1\le i\le s_0$) viewed as regular functions on the affine scheme $(\gl_{N_0})_{\bbz_S}$ of $N_0$-by-$N_0$ matrices such that for any (unital commutative) $\bbz_S$-algebra $R$, we have
\be\label{e:LieAlgebraScheme}
\ugfr(R)=\{x\in \gl_{N_0}(R)|\h \forall i, f_i(x)=0\}.
\ee
Let us view $\gl_{N_0}(R)$ as $R^{N_0^2}$, and write its elements as vectors. This way we write $f_i(x_1,\ldots,x_{N_0^2}):=\sum_j a_{ij} x_j$, where $a_{ij}\in \bbz_S$. And consider the matrix $A:=[a_{ij}]\in M_{s_0\times N_0^2}(\bbz_S)$. Since $\bbz_S$ is a PID, there are $\gamma_1\in \GL_{s_0}(\bbz_S)$, $\gamma_2\in\GL_{N_0^2}(\bbz_S)$, and positive integers $r_1|r_2|\cdots|r_{N_0^2-d}$ such that 
\be\label{e:SmithNormalForm}
A=\gamma_1 \left[\begin{array}{cc}
0&0\\
0&\diag(r_1,\ldots,r_{N_0^2-d})
\end{array}
\right] \gamma_2,
\ee
where $\diag(r_1,\ldots,r_{N_0^2-d})$ is the diagonal matrix with diagonal entries $r_i$'s (this is known as the Smith normal form when the diagonal block is at the top left corner.). Let us view the rows of $\gamma_2$ as a $\bbz_S$-basis $\{\ebf_1^{\ast},\ldots,\ebf_{N_0^2}^{\ast}\}$ of the dual of $\bbz_S^{N_0^2}$. And let $\{\ebf_1,\ldots,\ebf_{N_0^2}\}$ be the dual $\bbz_S$-basis of $\bbz_S^{N_0^2}$. So by (\ref{e:LieAlgebraScheme}) and (\ref{e:SmithNormalForm}), for any $\bbz_S$-algebra $R$ with {\em no additive torsion} element, we have 
\be\label{e:NewPresentationOfLieAlgebra}
x\in \ugfr(R) \Longleftrightarrow \ebf_{d+1}^{\ast}(x)=\cdots=\ebf^{\ast}_{N_0^2}(x)=0 \Longleftrightarrow x\in \bigoplus_{i=1}^{d}R \ebf_i.
\ee
And so in this case we have
\be\label{e:Dual}
\ugfr^{\ast}(R):=\Hom_R(\ugfr(R),R)=\bigoplus_{i=1}^{d}R \ebf_i^{\ast}.
\ee

{\bf Claim 2.} Suppose $\{\ebf_1,\ldots,\ebf_{N_0^2}\}$ is a $\bbz_S$-basis of $\gl_{N_0}(\bbz_S)$ such that (\ref{e:NewPresentationOfLieAlgebra}) and (\ref{e:Dual}) hold for any $\bbz_S$-algebra $R$ with no additive torsion element. For $(w_1,\ldots,w_d)\in R^d$, $(l_1,\ldots,l_d)\in R^d$, and $g\in \gcal(R)$, let 
\[
\eta_g(w_1,\ldots,w_d;l_1,\ldots,l_d):=\sum_{i,j}\ebf_i^{\ast}(\Ad(g)(\ebf_j)) l_iw_j.
\]
 If $\delta$ is small enough, there is $(w_1,\ldots,w_d,l_1,\ldots,l_d)\in \overline{\bbq}^{2d}$ ($\overline{\bbq}$ is the algebraic closure of $\bbq$) such that
\begin{enumerate}
	\item For any $h\in \lcal_{\delta}(\wt{H})$ we have $\eta_h(w_1,\ldots,w_d;l_1,\ldots,l_d)=0$,
	\item For some $\gamma\in \Omega'$ we have $\eta_{\gamma}(w_1,\ldots,w_d;l_1,\ldots,l_d)\neq 0$. 
\end{enumerate}
\begin{proof}[Proof of Claim 2]
By Claim 1, there are $w\in\ugfr(\bbz_p)$ and $L\in \ugfr(\bbz_p)^{\ast}$ such that for any $h\in \wt{H}$ and some $\gamma\in\Omega'$ we have
\be\label{e:Claim1Conclusion}
|L(\Ad(\gamma)(w))|_p\ge [\gcal(\bbz_p):\wt{H}]^{-c/C_0},
\text{ and }
|L(\Ad(h)(w))|_p\le [\gcal(\bbz_p):\wt{H}]^{-c},
\ee 	
for some $c$ which depends only on $\gcal$ and some large positive number $C_0$ which will be specified later.

Suppose $L(\Ad(\lambda)(w))^{-1}w=\sum_{i=1}^d x_i \ebf_i$ and $L=\sum_{i=1}^{d}y_i\ebf_i^{\ast}$; and let $\xbf_0:=(x_1,\ldots,x_d)$ and $\ybf_0:=(y_1,\ldots,y_d)$. Since $\ugfr(\bbz_p)=\gfr=\Lie(\bbg)(\bbq_p)\cap \gl_{N_0}(\bbz_p)$, it is a primitive $\bbz_p$-submodule of $\gl_{N_0}(\bbz_p)$; and so 
\be\label{e:UpperPAdicBoundForX}
\|\xbf_0\|_p=|L(\Ad(h)(w))|_p^{-1}\|w\|_p\le [\gcal(\bbz_p):\wt{H}]^{c/C_0}.
\ee
And clearly $\|\ybf_0\|_p\le 1$. By (\ref{e:Claim1Conclusion}) and the way $\xbf_0$ and $\ybf_0$ are defined, we get that for any $h\in \wt{H}$ and some $\gamma\in\Omega'$ we have 
\be\label{e:ConclusionOfClaim1InNewCoordinates}
\eta_{\gamma}(\xbf_0,\ybf_0)=1,\text{ and }
|\eta_h(\xbf_0,\ybf_0)|\le [\gcal(\bbz_p):\wt{H}]^{-c(1-1/C_0)}.
\ee
	
To prove the claim, we proceed by contradiction and use an effective version of Nullstellensatz theorem~\cite{BY} (or \cite[Theorem IV]{MW}). 

Suppose the following has no solution over $\overline{\bbq}$:
\[
\forall\h h\in \lcal_{\delta}(\wt{H}),\h \eta_h(\xbf,\ybf)=0,\h \eta_{\gamma}(\xbf,\ybf)=1,
\]
where $\gamma\in \Omega'$ is the one given in (\ref{e:ConclusionOfClaim1InNewCoordinates}). Notice that the total degree of $\eta_g$ is two for any $g\in\gcal(\bbz_p)$; and the $S$-norm of the coefficients of $\eta_h$ for $h\in \lcal_{\delta}(\wt{H})$ is at most $[\gcal(\bbz_p):\wt{H}]^{\Theta_{\Omega}(\delta)}$. So by the effective Nullstellensatz, there are polynomials $q_{\gamma}, q_h\in \bbz_S[T_1,\ldots,T_{2d}]$ and $D_0\in \bbz_S$ such that 
\be\label{e:Nullstellensatz}
 q_{\gamma}(\xbf,\ybf)(\eta_{\gamma}(\xbf,\ybf)-1)+\sum_{h\in \lcal_{\delta}(\wt{H})} q_h(\xbf,\ybf) \eta_h(\xbf,\ybf)=D_0, \text{ and}
\ee
\be\label{e:EffectivePart}
\deg q_{\gamma}, \deg q_h \le N(d), \text{ and } \|D_0\|_S\le  [\gcal(\bbz_p):\wt{H}]^{\Theta_{\Omega}(\delta)}, 
\ee
where $h\in \lcal_{\delta}(\wt{H})$ and $N(d)$ is a positive integer which depends only on $d:=\dim \bbg$. So by (\ref{e:EffectivePart}), (\ref{e:UpperPAdicBoundForX}), and the fact that $\ybf_0\in \bbz_p^d$, we have 
\be\label{e:ValueAtSpecialPoints} 
|q_h(\xbf_0,\ybf_0)|_p\le [\gcal(\bbz_p):\wt{H}]^{N(d)c/C_0}.
\ee
Hence by (\ref{e:ValueAtSpecialPoints}) and (\ref{e:ConclusionOfClaim1InNewCoordinates}), for any $h\in \lcal_{\delta}(\wt{H})$ we have
\be\label{e:ValueAtSpecialPoints2}
|q_h(\xbf_0,\ybf_0)\eta_h(\xbf_0,\ybf_0)|_p\le [\gcal(\bbz_p):\wt{H}]^{-c\left(1-\frac{N(d)+1}{C_0}\right)}.
\ee
Now suppose $C_0=2N(d)+2$; so by (\ref{e:ConclusionOfClaim1InNewCoordinates}), (\ref{e:Nullstellensatz}), and (\ref{e:ValueAtSpecialPoints2}), we have
\be\label{e:PAdicNormOfLHS}
|D_0|_p=\left | q_{\gamma}(\xbf_0,\ybf_0)(\eta_{\gamma}(\xbf_0,\ybf_0)-1)+\sum_{h\in \lcal_{\delta}(\wt{H})} q_h(\xbf_0,\ybf_0) \eta_h(\xbf_0,\ybf_0)\right|_p\le [\gcal(\bbz_p):\wt{H}]^{-c/2}.
\ee
Hence by (\ref{e:EffectivePart}) and (\ref{e:PAdicNormOfLHS}) we get
\[
[\gcal(\bbz_p):\wt{H}]^{c/2}\le \|D_0\|_S\le  [\gcal(\bbz_p):\wt{H}]^{\Theta_{\Omega}(\delta)},
\]
which gives us a contradiction for $\delta\ll_{\Omega} 1$.
\end{proof}
({\em Going back to the proof of Lemma~\ref{l:SmallLifts}.}) For a given $(\wbf,\lbf)\in \overline{\bbq}^{d}\times  \overline{\bbq}^{d}$, let 
\[
V_{\wbf,\lbf}(\overline{\bbq}):=\{g\in \bbg(\overline{\bbq})|\h \eta_g(\wbf,\lbf)=0\}.
\]
Then $V_{\wbf,\lbf}$ gives us a closed subvariety of $\bbg$. And since it is coming from an intersection of a hyperplane in $\gl_{N_0}(\overline{\bbq})$ and $\bbg(\overline{\bbq})$, both its number of irreducible components and their degrees are bounded by a positive integer $N_1:=N_1(\bbg)$ which depends only on $\bbg$. 

Hence by \cite[Proposition 3.2]{EMO} and its proof, there is a positive integer $N_2:=N_2(\bbg)$ such that for any subset $A$ of $\bbg(\overline{\bbq})$ which generates a Zariski-dense subgroup of $\bbg(\overline{\bbq})$ we have 
\be\label{e:EscapingZariskiClosedSets}
\textstyle \prod_{N_2} A \not\subseteq V_{\wbf,\lbf}(\overline{\bbq}),
\ee
for any $(\wbf,\lbf)\in \overline{\bbq}^d\times \overline{\bbq}^d$. By Claim 2, if $\delta\ll_{\gcal} 1$, there is $(\wbf,\lbf)\in \overline{\bbq}^{d}\times  \overline{\bbq}^{d}$ such that $\lcal_{N_2\delta}(\wt{H})\subseteq V_{\wbf,\lbf}(\overline{\bbq})$. 
And so
\be\label{e:Proper}
\textstyle \prod_{N_2}\lcal_{\delta}(\wt{H})\subseteq \lcal_{N_2\delta}(\wt{H})\subseteq V_{\wbf,\lbf}(\overline{\bbq}).
\ee
Therefore by (\ref{e:EscapingZariskiClosedSets}) and (\ref{e:Proper}) we have that $\lcal_{\delta}(\wt{H})$ does not generate a Zariski-dense subgroup of $\bbg(\overline{\bbq})$; proving Lemma.
\end{proof}
\begin{proof}[Proof of Proposition~\ref{p:EscapeModQ}]
By Nori's strong approximation~\cite[Theorem 5.4]{Nor} we know that $[\gcal(\bbz_p):\overline{\Gamma}]\ll_{\Gamma} 1$, where $\overline{\Gamma}$ is the closure of $\Gamma$ in $\gcal(\bbz_p)$. For a proper subgroup $H$ of $\pi_{p^n}(\Gamma)$, let $\wt{H}:=\{h\in \gcal(\bbz_p)|\h \pi_{p^n}(h)\in H\}$. So by Lemma~\ref{l:SmallLifts} we have that $\lcal_{\delta_1}(\wt{H})$ is in a proper algebraic subgroup $\bbh$ of $\bbg$ for small enough $\delta_1$. So by Proposition~\ref{p:EscapeSubgroups} we have that
\be\label{e:EscapeModQ}
\pcal^{(l)}(\{\gamma\in \Gamma|\h \pi_{p^n}(\gamma)\in H\})\le e^{-\delta_0 l}
\ee
if $1\ll l\ll \delta_1\log [\pi_{p^n}(\Gamma):H]$ where $\delta_0$ is given by Proposition~\ref{p:EscapeSubgroups} and $\delta_1$ is given by Lemma~\ref{l:SmallLifts}.

If for some $l\ge \frac{n}{\delta}\log p$ we have $\pcal_{\pi_{p^n}(\Omega)}^{(l)}(H)\ge [\pi_{p^n}(\Gamma):H]^{-\delta}$, then  by~\cite[Remark, page 1060]{BG3} (see \cite[Lemma 8]{SGsemisimple}) for any $l'\le \frac{n}{\delta}\log p$ we have  $\pcal_{\pi_{p^n}(\Omega)}^{(2l')}(H)\ge [\pi_{p^n}(\Gamma):H]^{-2\delta}$. This contradicts (\ref{e:EscapeModQ}) if $\delta$ is small enough.
\end{proof}
\subsection{Getting a large ideal by adding/subtracting a congruence subgroup boundedly many times.}~\label{ss:BoundedGenerationOfModules}
The main goal of this section is to prove Proposition~\ref{p:BoundedGenerationOfModules}. At the end we also provide a few lemmas that are needed for using Proposition~\ref{p:BoundedGenerationOfModules} in the context of modules. The results of this section rely on properties of $p$-adic analytic maps that are proved in the appendix. 
\begin{prop}\label{p:BoundedGenerationOfModules}
	Let $\bbg_s\subseteq (\GL_{N_0})_{\bbq}$ be a semisimple $\bbq$-subgroup. Let $\rho:\bbg_s\rightarrow \mathbb{GL}(\bbv)$ be a $\bbq$-homomorphism, where $\bbv$ is a $\bbq$-vector group. Suppose no non-zero vector of $\bbv(\overline{\bbq})$ is $\rho(\bbg_s(\overline{\bbq}))$-invariant. 
	For a prime $p$, let $P[p^l]:=\{g\in \bbg_s(\bbq_p)\cap \GL_{N_0}(\bbz_p)|\h \pi_{p^l}(g)=1\}$.  Then for $l\gg_{\bbg_s,\rho} 1$ we have
	\[\textstyle
	\sum_{\Theta_{\bbg_s,\rho}(1)}\rho(P[p^l])-\sum_{\Theta_{\bbg_s,\rho}(1)}\rho(P[p^l])\supseteq p^{\Theta_{\bbg_s,\rho}(l)}\bbz_p[\rho(P[1])].
	\]
	(Notice that the implied constants are independent of $p$.)
\end{prop}
Let us point out that the same proof gives us also the local version of Proposition~\ref{p:BoundedGenerationOfModules}. 
\begin{prop'}\label{p:BoundedGenerationOfModulesLocal}
Let $\bbg_s\subseteq (\GL_{N_0})_{\bbq}$ be a semisimple $\bbq_p$-subgroup. Let $\rho:\bbg_s\rightarrow \mathbb{GL}(\bbv)$ be a $\bbq_p$-homomorphism, where $\bbv$ is a $\bbq_p$-vector group. Let $P[p^l]:=\{g\in \bbg_s(\bbq_p)\cap \GL_{N_0}(\bbz_p)|\h \pi_{p^l}(g)=1\}$. 
Suppose that $\bbv(\bbq_p)$ has no non-zero $\bbg_s(\bbq_p)$ fixed vector. Then for $l\gg_{\bbg_s,\rho} 1$ we have
\[\textstyle
\sum_{\Theta_{\bbg_s,\rho}(1)}\rho(P[p^l])-\sum_{\Theta_{\bbg_s,\rho}(1)}\rho(P[p^l])\supseteq p^{\Theta_{\bbg_s,\rho}(l)}\bbz_p[\rho(P))].
\]
\end{prop'}
\begin{lem}~\label{l:LinearIndependence}
Let $K$ be a field (only in this Lemma), and $G$ be a subgroup of $\GL_{N_0}(K)$. Suppose $V:=K^{N_0}$ is a completely reducible $G$-module, and $V$ has no non-zero $G$-fixed vector. Then there is no $a\in M_{N_0}(K)$ such that ${\rm Tr}(ag)=1$ for any $g\in G$.
\end{lem} 
\begin{proof}
Suppose to the contrary that there is such $a\in M_{N_0}(K)$. Let $A:=K[G]$ be the $K$-span of $G$. Since $A$ has a faithful semisimple finite dimensional $A$-module, $A$ is a semisimple $K$-algebra. Let 
\[
\afr:=\langle g-1|\h g\in G\rangle
\]
be the ideal generated by $G-1$ in $A$. Since ${\rm Tr}(ag)=1$ for any $g\in G$, we have that ${\rm Tr}(ax)=0$ for any $x\in \afr$. Thus $\afr$ is a proper ideal. On the other hand, since $A$ is semisimple, by the idempotent decomposition there are $\alpha_1,\alpha_2\in A$ such that 
\begin{enumerate}
	\item $\alpha_i^2=\alpha_i$,
	\item $\alpha_1+\alpha_2=1$,
	\item $\afr=A\alpha_1$.
\end{enumerate}
Since $\afr$ is a proper ideal, $\alpha_1\neq 1$. And so there is a non-zero vector $v\in V$ in the kernel of $\alpha_1$, which implies that $\afr v=A \alpha_1 v=0$. Therefore for any $g\in G$ we have $gv=v$ which contradicts the assumption that $V$ does not have a non-zero $G$-fixed point. 
\end{proof}
\begin{proof}[Proof of Proposition~\ref{p:BoundedGenerationOfModules}]
Since $\bbg_s$ is a semisimple group,  $\bbv(\overline{\bbq}_p)$ is a completely reducible $\overline{\bbq}_p[\rho(\bbg_s(\overline{\bbq}_p))]$-module, for any prime $p$. As $\bbg_s(\bbq_p)$ is Zariski-dense in $\bbg_s(\overline{\bbq}_p)$, we have that
$\overline{\bbq}_p[\rho(\bbg_s(\overline{\bbq}_p))]=\overline{\bbq}_p[\rho(\bbg_s(\bbq_p))]$. Therefore $\overline{\bbq}_p[\rho(\bbg_s(\bbq_p))]$ has a faithful completely reducible module, which implies that $\overline{\bbq}_p[\rho(\bbg_s(\bbq_p))]$ is a semisimple algebra. Hence the nilradical of $\bbq_p[\rho(\bbg_s(\bbq_p))]$ is zero, too. Thus $\bbq_p[\rho(\bbg_s(\bbq_p))]$ is a semisimple algebra and $\bbv(\bbq_p)$ is a completely reducible $\bbq_p[\rho(\bbg_s(\bbq_p))]$-module. 

Next we show that $\bbv(\bbq_p)$ does not have a non-zero $\bbg_s(\bbq_p)$-invariant vector. If not, $\bbv(\bbq_p)^{\bbg_s(\bbq)}:=\{v\in \bbv(\bbq_p)|\h \rho(\bbg_s(\bbq))(v)=v\}$ is non-zero.  For $v_0\in\bbv(\bbq_p)^{\bbg_s(\bbq)}$ we have $v_0=\sum_i a_i v_i$ where $a_i\in \bbq_p$ are $\bbq$-linearly independent and $v_i\in \bbv(\bbq)$. Since $\rho$ is defined over $\bbq$, for $g\in \bbg_s(\bbq)$ we have $\rho(g)(v_i)\in \bbv(\bbq)$ for any $i$. Hence $\rho(g)(v_0)=v_0$ implies that for any $i$ we have $\rho(g)(v_i)=v_i$. Since $\bbg_s(\bbq)$ is Zariski-dense in $\bbg_s$, we get that $\rho(\bbg_s(\overline{\bbq}))(v_i)$, which is a contradiction.

Hence Lemma~\ref{l:LinearIndependence} implies that the constant function $\one$ does not belong to the linear span of the analytic functions $\rho_{ij}$ where $\rho_{ij}$ are the entries of $\rho:\bbg_s(\bbq_p)\rightarrow \GL(\bbv(\bbq_p))$ with respect to a $\bbq_p$-basis of $\bbv(\bbq_p)$. So we get the desired result by Corollary~\ref{c:UnifOpenImage}. 
\end{proof}

Next we prove a corollary of Lemma~\ref{l:SimpleModulesOrder}.
\begin{cor}\label{c:NoTrivialFactor}
Let $R\subseteq M_{n_0}(\bbz_p)$ be a $\bbz_p$-subalgebra. Let
\[
0=W_{k_0+1}\subseteq W_{k_0} \subseteq \ldots \subseteq W_1=\bbq_p^{n_0}
\] 
be a composition series of $\bbq_p^{n_0}$ as an $\bbq_p[R]$-module, where $\bbq_p[R]$ is the $\bbq_p$-span of $R$. Suppose $v_i\in M_i:=W_i\cap \bbz_p^{n_0}$ and $\|\overline{v_i}\|:=\inf \{\|v_i+w\|_p|\h w\in W_{i+1}\}\ge q^{-1}$, where $q$ is a power of $p$. Then   
\[
\sum_{i=1}^{k_0} R v_i\supseteq q^{\Theta_{R,\{W_j\}}(1)} \bbz_p^{n_0}.
\]
\end{cor}
\begin{proof}
For any $1\le i\le k_0$, we get a ring homomorphism from $R$ to ${\rm End}_{\bbq_p}(W_i/W_{i+1})$. Let $R_i$ be its image. Notice that $M_i/M_{i+1}$ can be embedded into $W_i/W_{i+1}$ and for any $v\in M_i$ we have $R_i(v+M_{i+1})=Rv+M_{i+1}$. So by Lemma~\ref{l:SimpleModulesOrder} for the ring $R_i$ we have that,
\[
q^{\Theta_{R,W_i}(1)} M_i \subseteq Rv_i+M_{i+1} \subseteq M_i.
\]
By induction on $i$ one can easily see that, for any $1\le i\le k_0$, we have
\[
q^{\Theta_{R,\{W_j\}}(1)} \bbz_p^{n_0}=q^{\Theta_{R,\{W_j\}}(1)} M_1\subseteq \sum_{j=1}^i Rv_j+M_{i+1}.
\]
\end{proof}

\begin{lem}\label{l:OneFactorLargePrime}
Let $R\subseteq M_{n_0}(\bbz_S)$ be a $\bbz_S$-subalgebra. Assume the $\bbq$-span $\bbq[R]$ of $R$ is a semisimple algebra. Suppose $p$ is a large prime number (depending on $R$), and $W$ is a composition factor of $\bbq_p^{n_0}$; that means there are two $\bbq_p[R]$-submodules $W_1\subseteq W_2$ such that $W= W_2/W_1$ is a simple $\bbq_p[R]$-module. Then for any $v\in W_{\bbz_p}:=(W_2\cap \bbz_p^{n_0})+W_1/W_1$ we have
$
\bbz_p[R]v=W_{\bbz_p}
$
if $v\not\in p W_{\bbz_p}$.
\end{lem}
\begin{proof}
By Wedderburn theorem and Artin-Brauer-Noether theorem, for large enough $p$ we have that 
\begin{enumerate}
	\item $\bbq_p[R]\simeq \bigoplus_i M_{n_i}(K_i)$ where $K_i$ are finite extensions of $\bbq_p$,
	\item under the above identification $\bbz_p[R]$ gets identified with $\bigoplus_i M_{n_i}(\ocal_{K_i})$,
	\item under the above identification $\bbz_p^{n_0}$ naturally gets identified with $\bigoplus_i \ocal_{K_i}^{n_i}$.   
\end{enumerate} 
Hence $W_{\bbz_p}$ is isomorphic to $\ocal_{K_i}^{n_i}$ where $\bbz_p[R]$ acts via $M_{n_i}(\ocal_{K_i})$. It is clear that for any unit vector $v\in \ocal_{K_i}^{n_i}$ we have $M_{n_i}(\ocal_{K_i})v=\ocal_{K_i}^{n_i}$.
\end{proof}

\begin{cor}\label{c:ManyFactorsLargePrime}
Let $R\subseteq M_{n_0}(\bbz_S)$ be a $\bbz_S$-subalgebra. Assume the $\bbq$-span $\bbq[R]$ of $R$ is a semisimple algebra. Let
\[
0=W_{k_0+1}\subseteq W_{k_0} \subseteq \ldots \subseteq W_1=\bbq_p^{n_0}
\] 
be a composition series of $\bbq_p^{n_0}$ as an $\bbq_p[R]$-module, where $\bbq_p[R]$ is the $\bbq_p$-span of $R$. Suppose $v_i\in M_i:=W_i\cap \bbz_p^{n_0}$ and $\|\overline{v_i}\|:=\inf \{\|v_i+w\|_p|\h w\in W_{i+1}\}=1$. Then, for large enough $p$, we have
\[
\bbz_p^{n_0}=\sum_{i=1}^{k_0}Rv_i.
\]
\end{cor}
\begin{proof}
It is a direct corollary of Lemma~\ref{l:OneFactorLargePrime}.
\end{proof}

\subsection{Getting a $p$-adically large vector in a submodule in boundedly many steps.}\label{ss:p-adicLargeBasis}
The main goal of this short section is proving a key lemma (Lemma~\ref{l:LargePAdic}). Using this Lemma, we will be able to get a $\bbq_p$-basis of $\bbv(\bbq_p)$ consisting of {\em large} vectors in {\em boundedly many steps}.  

\begin{lem}\label{l:LargePAdic}
Let $\Omega$, $\Gamma$ and $\bbg$ be as above. In particular, $\bbg=\bbg_s\ltimes \bbv$ where $\bbg_s$ is a semisimple $\bbq_p$-group and $\bbv$ is a $\bbq_p$-vector group; and for some non-negative integer $k_p\ll_{\Omega} 1$ (which is zero for large enough $p$) $Q_p[p^{k_p}]=K_p\ltimes V_p$ where $Q_p$ is the closure of $\Gamma$ in $\bbg(\bbq_p)$, $Q_p[p^{k_p}]:=Q_p\cap \GL_{N_0}(\bbz_p)[p^{k_p}]$, $K_p:=\bbg_s(\bbq_p)\cap \GL_{N_0}(\bbz_p)[p^{k_p}]$, and $V_p:=\bbv(\bbq_p)\cap \GL_{N_0}(\bbz_p)[p^{k_p}]$ (see Section~\ref{ss:InitialReductions}).

Let $\bbv'\subseteq \bbv$ be a non-zero $\bbq_p$-subgroup which is $\bbg_s$-invariant. For any $0<\vare\ll_{\Omega,\bbv'} 1$, there are $\delta>0$ and a positive integer $C$ such that the following holds: 

Suppose $p$ is a prime number, $Q=p^n$, and $n\vare\gg_{\Omega,\bbv'} 1$.

If $\Pfr_{Q}(\delta,A,l)$ holds and $\pi_{Q}(\Gamma[q_2])\subseteq \pi_Q(\prod_C A)\pi_Q(V_p\cap \bbv'(\bbq_p))$ where $q_2=p^{n_2}$ and $n_2\le \vare^2 n$, then 
\be\label{e:LargePAdic}
\textstyle\{v\in \bbv'(\bbq_p)|\h \|v\|_p> Q^{\vare},\h \pi_Q(v)\in \pi_Q(\prod_{4C} A)\}
\ee
is non-empty.
\end{lem}
\begin{proof}
Since $\bbv'$ is $\bbg_s$-invariant and it commutes with $\bbv$, $\bbv'$ is a normal subgroup of $\bbg$. On the other hand, $\GL_{N_0}(\bbz_p)[p^{k_p}]$ is a normal subgroup of $\GL_{N_0}(\bbz_p)$ and $Q_p\subseteq \GL_{N_0}(\bbz_p)$; and so $V_p\cap \bbv'(\bbq_p)=\bbv'(\bbq_p)\cap \GL_{N_0}(\bbz_p)[p^{k_p}]$ is a normal subgroup of $Q_p$. Let $\eta_{q_1}:\pi_{q_1}(Q_p)\rightarrow \pi_{q_1}(Q_p)/\pi_{q_1}(V_p\cap \bbv'(\bbq_p))$ be the projection map where $q_1=p^{n_1}$ and $n_1=\lfloor\vare n\rfloor$. So by the assumption there is a section 
\[
s:\eta_{q_1}(\pi_Q(Q_p[q_2]))\rightarrow \pi_{q_1}(Q_p)
\]
such that the image of $s$ is a subset of $\pi_{q_1}(\prod_C A)$ (by section we mean $\eta_{q_1}(s(x))=x$). Notice that 
\[
\eta_{q_1}(s(x_1^{-1}x_2)^{-1}s(x_1)s(x_2^{-1}))=\eta_{q_1}(s(x_1^{-1}x_2))^{-1}\eta_{q_1}(s(x_1))\eta_{q_1}(s(x_2^{-1}))=(x_1^{-1}x_2)^{-1}(x_1)(x_2^{-1})=1.
\]
Hence $s(x_1^{-1}x_2)^{-1}s(x_1)s(x_2^{-1})\in \pi_{q_1}(\prod_{3C}A)\cap \pi_{q_1}(V_p\cap \bbv'(\bbq_p)).$ Now if we assume to the contrary that the set in (\ref{e:LargePAdic}) is empty, then we get that $\pi_{q_1}(\prod_{3C}A)\cap \pi_{q_1}(V_p\cap \bbv'(\bbq_p))=\{1\}$. And so $s$ is a group homomorphism. So $H:={\rm Im}(s)$ is a subgroup of $\pi_{q_1}(Q_p)$, and $[\pi_{q_1}(Q_p):H]\ge |\pi_{q_1}(V_p\cap \bbv'(\bbq_p))|\ge p^{\Theta_{\bbv'}(n\vare)}$. By Proposition~\ref{p:EscapeModQ} we have that
\be\label{e:EscapeQ}
\pcal^{(2l)}_{\pi_{q_1}(\Omega)}(H)\le [\pi_{q_1}(\Gamma):H]^{-\Theta_{\Omega}(1)}\le p^{-\Theta_{\Omega,\bbv'}(n\vare)}.
\ee
On the other hand, if $x\in \pi_{q_1}(\prod_{2C} A)\cap \pi_{q_1}(\Gamma[q_2])$, then 
\[\textstyle
x^{-1}s(\eta_{q_1}(x)))\in \pi_{q_1}(\prod_{4C} A)\cap \pi_{q_1}(V_p\cap\bbv'(\bbq_p)).
\]
So by the contrary assumption we have that $x\in H$, which means
\be\label{e:InTheSubgroupH}
\textstyle
\pi_{q_1}(\prod_{2C} A)\cap \pi_{q_1}(\Gamma[q_2])\subseteq H.
\ee
By the assumption we have $\pcal^{(l)}_{\Omega}(A)>Q^{-\delta}$ (where $\delta$ is small positive number to be determined later). So 
\[
\pcal^{(l)}_{\Omega}(A\cap \pi_{q_2}^{-1}(x'))> Q^{-\delta} |\pi_{q_2}(\Gamma)|^{-1}= p^{-\Theta_{\Omega}(n\vare^2)}=p^{-n(\delta+\Theta_{\Omega}(\vare^2))},
\]
 for some $x'\in\pi_{q_2}(\Gamma)$. Hence $\pcal^{(2l)}(A\cdot A\cap \Gamma[q_2])> p^{-n\Theta_{\Omega}(\vare^2)}$ for small enough $\delta$. Therefore by (\ref{e:InTheSubgroupH}) and (\ref{e:EscapeQ}) we have
\[
-n\Theta_{\Omega}(\vare^2)<-n\Theta_{\Omega,\bbv'}(\vare),
\]  
which is a contradiction if $0<\vare\ll_{\Omega,\bbv'} 1$.
\end{proof}

\subsection{Proof of super-approximation: the $p$-adic case.}\label{ss:AbelianU}
Let us recall that $\bbg=\bbg_s\ltimes \bbv$ where $\bbg_s$ is a semisimple $\bbq_p$-group and $\bbv$ is a $\bbq_p$-vector group; let $\rho:\bbg_s\rightarrow \mathbb{GL}(\bbv)$ be the $\bbq_p$-representation which gives us the action of $\bbg_s$ on $\bbv$. Let us recall that, since $\bbg$ is perfect, no non-zero vector of $\bbv(\bbq_p)$ is $\rho(\bbg_s(\bbq_p))$-invariant. Let $\bbv_i$ be $\bbq_p$-subgroups of $\bbv$ such that 
\begin{enumerate}
	\item $\bbv_i(\bbq_p)$ are irreducible $\bbg_s(\bbq_p)$-modules. 
	\item $\bbv=\oplus_{i=1}^{k_0} \bbv_i$. 
\end{enumerate}
Let $P[p^k]:=\{g\in\bbg_s(\bbq_p)|\h \|g-1\|_p\le p^{-k}\}$ for any non-negative integer $k$. 

Inductively we do the following simultaneously:
\begin{enumerate}
	\item construct a permutation $\sigma:\{1,\ldots,k_0\}\rightarrow \{1,\ldots,k_0\}$,
	\item find vectors $v_j\in W_j:=\bigoplus_{i\in\{1,\ldots, k_0\}\setminus \{\sigma(1),\ldots,\sigma(j-1)\}}\bbv_{\sigma(i)}(\bbq_p)$,
\end{enumerate}
such that the following holds:
\begin{enumerate}
	\item for small enough $\delta$, large enough $C$ (depending on $\vare$, $\Omega$, and $\bbv_i$) and any $1\le j\le k_0$ we have
\be\label{e:BG}\textstyle
\pi_Q\left(\Gamma[p^{n (\vare^{\Theta_{\bbg}(1)})}]\right)\subseteq \pi_Q(\prod_{\Theta(C)} A)\pi_Q(V_p\cap W_j);
\ee
	\item $v_j \in W_j$;
	\item $\overline{v}_1,\ldots,\overline{v}_j$ generate $W_1/W_{j+1}$ as a $\bbg_s(\bbq_p)$-module, where $\overline{v}_i$ is the projection of $v_i$.
	\item $\log_p\|\overline{v}_j\|_p\gg n \vare^{\Theta_{\bbg}(1)}$;
	\item $\sum_{i=1}^j\left(\sum_{\Theta_{\bbg}(1)}\rho(P[p^{n (\vare^{\Theta_{\bbg}(1)})}])(v_i)-\sum_{\Theta_{\bbg}(1)}\rho(P[p^{n (\vare^{\Theta_{\bbg}(1)})}])(v_i)\right)+W_{j+1}\supseteq p^{n (\vare^{\Theta_{\bbg}(1)})} V_p$. 
\end{enumerate}
By \cite[Theorem 36]{SGsemisimple}, if $\delta$ is small enough and $C$ is large enough, (\ref{e:BG}) holds for $j=1$. Now we suppose $(\ref{e:BG})$ holds for $1\le j\le j_0$ and we have already found $v_1,\ldots,v_{j_0-1}$ as desired. We will define $\sigma(j_0)$ and find $v_{j_0}$ so that the above properties hold. 

Since (\ref{e:BG}) holds for $j=j_0$, by Lemma~\ref{l:LargePAdic} there is $v_{j_0}\in W_{j_0}$ such that
\be\label{e:LargeVectors}\textstyle
\log_p\|v_{j_0}\|_p\gg n\vare^{\Theta_{\bbg}(1)},\h{\rm and}\h \pi_Q(v_{j_0})\in \pi_Q(\prod_{\Theta(C)}A).
\ee
Since $\log_p\|v_{j_0}\|_p\gg n\vare^{\Theta_{\bbg}(1)}$, the projection $\overline{v}_{j_0}$ of $v_{j_0}$ to $\bbv_{\sigma(j_0)}$ for some 
\[
\sigma(j_0)\in \{1,\ldots, k_0\}\setminus \{\sigma(1),\ldots,\sigma(j_0-1)\}
\]
 has length at least $\Theta_{\bbv_i}(p^{\Theta(n\vare^{\Theta_{\bbg}(1)})})$. 
 Now by Proposition~\ref{p:BoundedGenerationOfModules}, and Corollaries~\ref{c:NoTrivialFactor} and~\ref{c:ManyFactorsLargePrime} one gets all the mentioned properties except the first one. 
 
By \cite[Theorem 36]{SGsemisimple} we have that for small enough $\delta$ and large enough $C$
\be\label{e:SSpart}
\textstyle
\pi_Q(P[p^{n\vare^{\Theta(1)}}]) \subseteq \pi_Q(\prod_C A) \pi_Q(V_p).
\ee 
To simplify our notation let us drop the constant power of $\vare$ in the rest of the argument. Hence by (\ref{e:LargeVectors}) we have
\be\label{e:UnipotentPart}
\textstyle
\pi_Q(p^{n\vare}V_p)\subseteq \pi_Q(\prod_{\Theta(C)} A) \pi_Q(V_p \cap W_{j}).
\ee
For any $g_1,g_2\in P[p^{n\vare}]$ and $v_1,v_2\in V_p$, we have 
\[
[(g_1,v_1),(g_2,v_2)]\in P[p^{2n\vare}]\ltimes (p^{n\vare} V_p). 
\]
Hence by (\ref{e:SSpart}) and (\ref{e:UnipotentPart}) we have
\be\label{e:Commutators}
\textstyle
\pi_Q(\{[g_1,g_2]|\h g_1,g_2\in P[p^{n\vare}]\})\subseteq \pi_Q(\prod_{\Theta(C)} A)\pi_Q(V_p\cap W_j).
\ee
Let $X:=\{[g_1,g_2]|\h g_1,g_2\in P[p^{n\vare}]\}$. Using properties of the finite logarithmic maps (see~\cite[Lemma 34]{SGsemisimple}), we have 
\[
\textstyle
(\prod_{\dim \bbg_s} X \cap P[p^{kn\vare+\Theta(1)}]) P[p^{2kn\vare}]=P[p^{kn\vare+\Theta(1)}],
\]
for any integer $k\ge 2$. Therefore $\pi_Q(\prod_{\Theta(-\log \vare)} X)\supseteq\pi_Q(P[p^{n\Theta(\vare)}])$. Hence
\be\label{e:CongruenceSSpart}
\textstyle
\pi_Q(P[p^{n\Theta(\vare)}])\subseteq \pi_Q(\prod_{\Theta(C)} A) \pi_Q(V_p\cap W_j).
\ee
And so by (\ref{e:UnipotentPart}) and (\ref{e:CongruenceSSpart}) we have
\[
\textstyle
\pi_Q(\Gamma[p^{n\Theta(\vare)}]) \subseteq \pi_Q(\prod_{\Theta(C)} A) \pi_Q(V_p\cap W_j),
\]
which finishes the proof.

\section{Appendix A: quantitative open image for $p$-adic analytic maps.}
The main goal of this appendix is to prove that one can get a {\em large open set} by adding the image of an analytic function in controlled number of times (see Proposition~\ref{p:QuanOpenImage}).   
\begin{prop}\label{p:QuanOpenImage}
Let $K$ be a characteristic zero non-Archimedean local field. Let $\ocal:=\ocal_K$ be its ring of integers, and $\pfr$ be a uniformizing element. Let $U\subseteq K^{n_0}$ be a neighborhood of the origin, and $F:=(f_1,\ldots,f_{d_0}):U\rightarrow K^{d_0}$ be an analytic function. Suppose the constant function $\one$ is not in the $K$-span of $f_i$. Then for any $l\gg_F 1$ we have
\[\textstyle
\sum_{\Theta_F(1)} F(\pfr^l \ocal^{n_0})-\sum_{\Theta_F(1)} F(\pfr^l \ocal^{n_0})\supseteq (K\text{-span of}\h F(U))\cap \pfr^{\Theta_F(l)}\ocal^{d_0}.
\] 
\end{prop} 

In fact, we prove the following refinement of Proposition~\ref{p:QuanOpenImage}. In Proposition~\ref{p:QuanOpenImage'}, we pinpoint how the implied constants depend on the given analytic functions (with a few extra technical assumptions). This type of control helps us to prove  Corollary~\ref{c:UnifOpenImage}. Corollary~\ref{c:UnifOpenImage} is the only result in the appendix that is needed in the main part of the article. In Corollary~\ref{c:UnifOpenImage}, we deal with polynomial maps that are defined over a number field, and prove that their images in all the non-Archimedean completions are {\em uniformly large}.  

\begin{prop'}\label{p:QuanOpenImage'}
Let $K$ be a characteristic zero non-Archimedean local field. Let $\ocal:=\ocal_K$ be its ring of integers, and $\pfr$ be a uniformizing element. Let $F:=(f_1,\ldots,f_{d_0}):\ocal^{n_0}\rightarrow \ocal^{d_0}$ be an analytic function of the form
\[
F(\xbf)=\sum_{\ibf}(c_{\ibf,1}\xbf^{\ibf},\ldots,c_{\ibf,d_0}\xbf^{\ibf}),
\]
where $\ibf=(i_1,\ldots,i_{n_0})$ ranges over multi-indexes non-negative integers and $\xbf^{\ibf}=x_1^{i_1}\cdots x_{n_0}^{i_{n_0}}$. Suppose 
	\begin{enumerate}
	\item $|c_{\ibf,j}|_{\pfr}\le  1$ for any $\ibf=(i_1,\ldots,i_m)$ and any $1\le j\le d_0$, 
	\item $L:K^{d_0'}\rightarrow K^{d_0}$ is a linear embedding such that $F(\xbf)=L(f_{j_1}(\xbf),\ldots,f_{j_{d'_0}}(\xbf))$ for some indexes $1\le j_u\le d_0$,
	\item $|\det(c_{\ibf_i,j_u})_{1\le i,u\le d_0'}|_{\pfr}=|\pfr^{k_0}|_{\pfr}$ for some multi-indexes $\ibf_1,\ldots,\ibf_{d'_0}$ whose coordinates add up to a number at most $m_0$ and some indexes $j_1,\ldots,j_{d_0'}$. 
	\end{enumerate}

  Then for any $l\gg_{d'_0,n_0,m_0,k_0,\|L\|_{\pfr}} 1$ we have
\[\textstyle 
\sum_{\Theta(1)} F(\pfr^l \ocal^{n_0})-\sum_{\Theta(1)} F(\pfr^l \ocal^{n_0})\supseteq (K\text{-span of}\h F(U))\cap \pfr^{\Theta(l)}\ocal^{d_0},
\] 
where all the implied constants depend on $d_0,n_0,m_0$ and $k_0$.
\end{prop'}

\begin{cor}\label{c:UnifOpenImage}
	Let $\kappa$ be a number field. Let $f_1,\ldots,f_{d_0}\in \kappa[x_1,\ldots,x_{n_0}]$. Suppose $\one, f_1, \ldots, f_{d_0}$ are linearly independent where $\one$ is the constant polynomial one. Let $F:=(f_1,\ldots,f_{d_0})$. Then for any $l\gg_{F} 1$ and any $\pfr\in V_f(\kappa)$ we have 
	\[\textstyle
	\sum_{\Theta_{F}(1)} F(\pfr^{l}\ocal_{\pfr}^{n_0})-\sum_{\Theta_{F}(1)} F(\pfr^{l}\ocal_{\pfr}^{n_0})\supseteq (\kappa_{\pfr}\text{-span of}\h F(\ocal_{\pfr}))\cap \pfr^{\Theta_F(l)}\ocal_{\pfr}^{d_0},
	\]
	where $\kappa_{\pfr}$ is the completion of $\kappa$ with respect to the finite place $\pfr\in V_f(\kappa)$, $\ocal_{\pfr}$ is the ring of integers of $\kappa_{\pfr}$, and $\pfr$ also shows a uniformizing element of $\ocal_{\pfr}$. 
\end{cor}

Let us recall the needed notation from $\pfr$-adic analysis. If $U$ is an open subset of $K$, then $\nabla^k U:=\{(x_1,\ldots,x_k)\in U\times \cdots \times U|\h x_i\neq x_j\h\text{if}\h i\neq j\}$. If $f:U\rightarrow K$, then $\Phi^k f:\nabla^{k+1} U\rightarrow K$ is defined recursively 
\[
\Phi^k f(x_1,\ldots,x_k):=\frac{\Phi^{k-1} f(x_1,x_3,\ldots,x_k)-\Phi^{k-1} f(x_2,\ldots,x_k)}{x_1-x_2},
\]
and $\Phi^0f=f$. If $f$ is an analytic function, then $\Phi^n f$ can be uniquely extended to a continuous function $\oPhi^n f:U^{n+1}\rightarrow K$ for any $n$ and $f^{(i)}(a)=i! \oPhi^{i}f(a,\ldots,a)$. Similarly for open subsets $U_i\subseteq K$ and a (multi-variable) function $f:U_1\times \cdots \times U_m \rightarrow K$ we can define the $k$th order difference quotient $\Phi_i^k f$ of $f$ with respect to the $i$th variable. And then for a multi-index $\ibf:=(i_1,\ldots,i_m)$ we set
\[
\Phi_{\ibf}f:\nabla^{i_1+1}U_1\times \cdots \times \nabla^{i_m+1}U_m\rightarrow K,\hspace{1cm} \Phi_{\ibf}f:=\Phi_1^{i_1}\circ \Phi_2^{i_2} \circ \cdots \circ \Phi_m^{i_m}f.
\]
If $f$ is analytic, then $\Phi_{\ibf} f$ can be uniquely extended to a continuous function $\oPhi_{\ibf}f:U_1^{i_1+1}\times \cdots \times U_m^{i_m+1}\rightarrow K$, and 
\[
\partial_{\ibf} f(a_1,\ldots,a_m)=\ibf! \oPhi_{\ibf}f(\theta_\ibf(a_1,\ldots,a_m)),
\]
where $\ibf!:=i_1!i_2!\cdots i_m!$ and
\[
\theta_\ibf(a_1,\ldots,a_m):=(\underbrace{a_1,\ldots,a_1}_{i_1+1-\text{times}},\underbrace{a_2,\ldots,a_2}_{i_2+1-\text{times}},\ldots, \underbrace{a_m,\ldots,a_m}_{i_m+1-\text{times}}). 
\]
Let $U:=U_1\times \cdots \times U_m\subseteq K^m$ be an open subset, $f_i:U\rightarrow K$ be analytic functions, and $F:=(f_1,\ldots,f_d):U\rightarrow K^d$ (in computations, it is considered as a $d$-by-$1$ column matrix). Then $dF(\xbf)$ is a $d$-by-$m$ matrix
\[
dF(\xbf):=[\partial_{j}f_i(\xbf)].
\]
For any $1\le i,j\le m$, let $\oPhi_{\ebf_i+\ebf_j}F(\bullet)$ be the column vector whose $k$th entry is $\oPhi_{\ebf_i+\ebf_j}f_k(\bullet)$. In the above setting, by Taylor expansion we have
\be\label{e:Taylor}
F(\xbf_0+\xbf)=F(\xbf_0)+dF(\xbf_0)\xbf+R_2F(\xbf_0+\xbf,\xbf_0), 
\ee
where $R_2F(\xbf_0+\xbf,\xbf_0)$ is of the form $\sum_{i,j} x_i x_j \oPhi_{\ebf_i+\ebf_j}F(\bullet)$ (the entries are the entries of either $\xbf_0+\xbf$ or $\xbf_0$) and $\xbf=(x_1,\ldots,x_m)$.  

For a $d$-by-$m$ matrix $X=[\vbf_1\cdots \vbf_m]$ with entries in $K$, let 
\[
N(X):=\max_{1\le i_1\le \cdots\le i_d\le m} |\det[\vbf_{i_1}\cdots \vbf_{i_d}]|. 
\]
Let $\{\ebf_i\}_i$ be the standard basis of $K^n$; and for any $I:=\{i_1<i_2<\cdots<i_m\}\subseteq \{1,\ldots,n\}$, let $\ebf_I:=\ebf_{i_1}\wedge \ldots \wedge \ebf_{i_m}$. It is well-known that $\{\ebf_I\}_{I\subseteq \{1,\ldots,n\},\h |I|=m}$ is a basis of $\bigwedge^m (K^n)$. For any 
\[
\xbf=\sum_{I\subseteq \{1,\ldots,n\},\h |I|=m} x_I \ebf_I\in {\textstyle\bigwedge^m}  (K^n),
\text{ we let }
\|\xbf\|:=\max\{|x_I| |\h I\subseteq \{1,\ldots,n\},\h |I|=m\}.
\]

\begin{lem}\label{l:LinearPart}
\begin{enumerate}
	\item If $\wbf_1,\ldots,\wbf_d$ are the rows of $X\in M_{d,m}(K)$, then $N(X)=\|\wbf_1\wedge \cdots\wedge \wbf_d\|$; in particular $N(X)=0$ if $\rk(X)<d$.
	\item If $X\in M_{d,m}(\ocal)$ and $N(X)\ge |\pfr^{k_0}|$, then for any positive integer $l$ and $\ybf \in \pfr^{l+k_0}\ocal^d$ there is $\xbf \in \pfr^l\ocal^m$ such that $X\xbf=\ybf$.
	\item For any $1\le i\le m$ and vectors $\vbf_j$ we have 
	\[
	N[\vbf_1 \cdots \vbf_i (\vbf_{i+1}-\vbf_i) \cdots (\vbf_m-\vbf_i)]\le (d+1) N[\vbf_1\cdots \vbf_m].
	\]
\end{enumerate} 
\end{lem}
\begin{proof}
The first part is clear. Now suppose that  $N(X)\ge |\pfr^{k_0}|$; then there are $d$ columns $\vbf_{i_1},\ldots, \vbf_{i_d}$ such that $|\det[\vbf_{i_1}\cdots \vbf_{i_d}]|\ge |\pfr^{k_0}|$. Let $X_I=[\vbf_{i_1}\cdots \vbf_{i_d}]$. Then $X_I{\rm adj}(X_I)\ybf=\det(X_I)\ybf$, and so there is $\xbf'\in \pfr^l\ocal^d$ such that $X_I\xbf'=\ybf$. Therefore there is $\xbf\in \pfr^l\ocal^m$ such that $X\xbf=\ybf$.

For the third part, we get an upper-bound for the determinant of a $d$-by-$d$ submatrix:
\begin{align*}
|\det[\vbf_{i_1} \cdots \vbf_{i_j} (\vbf_{i_{j+1}}-\vbf) \cdots (\vbf_{i_{d}}-\vbf)]|&
=\|\vbf_{i_1}\wedge \cdots\wedge \vbf_{i_j}\wedge (\vbf_{i_{j+1}}-\vbf)\wedge \cdots \wedge(\vbf_{i_{d}}-\vbf)\|\\
&
=\|\vbf_{i_1}\wedge \cdots\wedge \vbf_{i_d}-\sum_{k=j+1}^d \vbf_{i_{j+1}}\wedge \cdots\wedge \vbf_{i_{k-1}}\wedge \vbf \wedge \vbf_{i_{k+1}}\wedge \cdots\wedge \vbf_{i_d}\|\\
&\le (d+1) N([\vbf_1 \cdots \vbf_m]),
\end{align*}
where $\vbf$ is either $0$ or $\vbf_i$.
\end{proof}
In what follows, a series of Lemmas are proved in pairs. In the first ones the constants depend on $F$, and they are geared towards proving Proposition~\ref{p:QuanOpenImage}. In the second ones, we make the needed modification and make them suitable for proving Proposition~\ref{p:QuanOpenImage'}. 

\begin{lem}[Hensel's lemma]\label{l:Hensel}
Let $\xbf_0\in K^m$ and $U\subseteq K^m$ be an open neighborhood of $\xbf_0$. Let $f_i:U \rightarrow K$ be analytic functions and $F=(f_1,\ldots,f_d)$. Suppose there is a positive integer $k_0$ such that
\begin{enumerate}
  \item $\xbf_0+\pfr^{k_0}\ocal^m\subseteq U$,
	\item $\|dF(\xbf_0)\|\le 1$,
	\item for any $i,j$, if $\|\xbf'-\theta_{\ebf_i+\ebf_j}(\xbf_0)\|\le |\pfr^{k_0}|$, then $\|\oPhi_{\ebf_i+\ebf_j}F(\xbf')\|\le 1$,
	\item $N(dF(\xbf_0))\ge |\pfr^{k_0}|$.
\end{enumerate} 
Then for any $l\ge k_0$ and any $\ybf\in \ocal^d$ there is $\xbf_1\in \xbf_0+\pfr^l \ocal^m$ such that
\[
\|F(\xbf_1)-F(\xbf_0)-\pfr^{l+k_0}\ybf\|\le |\pfr^{2l}|.
\]
\end{lem}
\begin{proof}
By (\ref{e:Taylor}) we have that 
\be\label{e:LinearEstimate}
\|F(\xbf_0+\pfr^l \xbf)-F(\xbf_0)-\pfr^l dF(\xbf_0)\xbf\|=\|\sum_{i,j} (\pfr^l x_i)(\pfr^l x_j) \oPhi_{\ebf_i+\ebf_j}F(\xbf_{ij})\|,
\ee
where $\xbf=(x_1,\ldots,x_m)$, $\|\xbf\|\le 1$ and $\|\xbf_{ij}-\theta_{\ebf_i+\ebf_j}(\xbf_0)\|\le |\pfr^l|$. Now by Lemma~\ref{l:LinearPart} for any $\ybf \in \ocal^d$ there is $\xbf\in \ocal^m$ such that $dF(\xbf_0)\xbf=\pfr^{k_0}\ybf$. And so by (\ref{e:LinearEstimate}) 
\[
\|F(\xbf_0+\pfr^l \xbf)-F(\xbf_0)-\pfr^{l+k_0}\ybf\|\le |\pfr^{2l}|.
\]
\end{proof}

\begin{lem}[Quantitative open function theorem]\label{l:OpenFunction}
Let $\xbf_0\in K^m$ and $U\subseteq K^m$ be an open neighborhood of $\xbf_0$. Let $f_i:U \rightarrow K$ be analytic functions and $F=(f_1,\ldots,f_d)$. Suppose 
\begin{enumerate}
	\item $\|dF(\xbf_0)\|\le 1$,
	\item $\|\partial_{ij}F(\xbf_0)\|\le 1$, for any $i,j$,
	\item $N(dF(\xbf_0))\ge |\pfr^{k_0}|$, for some positive integer $k_0$.
\end{enumerate}
Then for any large enough integer $l$ (depending on $F$ and $U$) we have 
\[
F(\xbf_0+\pfr^l \ocal^m)\supseteq F(\xbf_0)+\pfr^{l+k_0}\ocal^d.
\]
\end{lem}
\begin{proof}
By the continuity of $dF$ and $\oPhi_{\ebf_i+\ebf_j}F$, for large enough $l$ (in particular $l>k_0$), we have that
\begin{enumerate}
	\item $\|dF(\xbf)\|\le 1$ if $\|\xbf-\xbf_0\|\le |\pfr^l|$,
	\item $\|\oPhi_{\ebf_i+\ebf_j}F(\xbf')\|\le 1$ if $\|\xbf'-\theta_{\ebf_i+\ebf_j}(\xbf_0)\|\le |\pfr^l|$,
	\item $N(dF(\xbf))\ge |\pfr^{k_0}|$ if $\|\xbf-\xbf_0\|\le |\pfr^l|$.
\end{enumerate}

By induction on $i$ we prove that for any $\ybf\in \ocal^d$ there are $\xbf_i\in \xbf_0+\pfr^l\ocal^m$ and integers $l_i\ge 2l$ such that 
\begin{enumerate}
	\item $l_1:=2l$ and $l_{i+1}:=2(l_i-k_0)$,
	\item $\|F(\xbf_i)-F(\xbf_0)-\pfr^{l+k_0}\ybf\|\le |\pfr^{l_i}|$.
\end{enumerate}
Lemma~\ref{l:Hensel} gives us the base of the induction. By induction hypothesis, there is $\xbf_i\in \xbf_0+\pfr^l\ocal^m$ such that 
\[
F(\xbf_i)-F(\xbf_0)-\pfr^{l+k_0}\ybf\in \pfr^{l_i}\ocal^d.
\] 
Hence by Lemma~\ref{l:Hensel} there is $\xbf_{i+1}\in \xbf_i+\pfr^{l_i-k_0}\ocal^m\subseteq \xbf_0+\pfr^l \ocal^m$ such that
\[
\|F(\xbf_{i+1})-F(\xbf_i)+(F(\xbf_i)-F(\xbf_0)-\pfr^{l+k_0}\ybf)\|\le |\pfr^{2(l_i-k_0)}|=|\pfr^{l_{i+1}}|.
\]
This proves the induction step. One can easily see that $\{l_i\}$ is a strictly  increasing integer sequence. So $\{\xbf_i\}$ is a Cauchy sequence. Therefore $\lim_{i\rightarrow \infty}\xbf_i=\xbf\in \xbf_0+\pfr^l \ocal^m$ by the compactness of $\xbf_0+\pfr^l \ocal^m$, and by the continuity of $F$ we have $F(\xbf)=F(\xbf_0)+\pfr^{l+k_0}\ybf$.
\end{proof}

\begin{lem'}\label{l:OpenFunction'}
Let $F:=(f_1,\ldots,f_{d_0}):\ocal^{n_0}\rightarrow \ocal^{d_0}, F(\xbf)=\sum_{\ibf}(c_{\ibf,1}\xbf^{\ibf},\ldots,c_{\ibf,d_0}\xbf^{\ibf})$ be an analytic function such that
\begin{enumerate}
\item $|c_{\ibf,j}|\le 1$ for any $\ibf$ and $j$,
\item $N(dF({\bf 0}))\ge |\pfr^{k_0}|$ for some positive integer $k_0$.
\end{enumerate}
Then for any $l\gg_{k_0} 1$ we have
\[
F(\pfr^l\ocal)\supseteq F({\bf 0})+\pfr^{l+k_0}\ocal^{d_0}.
\]
\end{lem'}
\begin{proof}
Since 
$F(\xbf):=\sum_{\ibf}(c_{\ibf,1}\xbf^{\ibf},\ldots,c_{\ibf,d_0}\xbf^{\ibf})$ 
and 
$|c_{\ibf,j}|\le 1$, we have 
$\|dF(\xbf)\|\le 1$ and 
$\|\oPhi_{\ebf_i+\ebf_j} F(\xbf')\|\le 1$ 
for any $\|\xbf\|\le 1$ and 
$\|\xbf'\|\le 1$. 

Since $|c_{\ibf,j}|\le 1$ and $N(dF({\bf 0}))\ge |\pfr^{k_0}|$, we have that $N(dF(\xbf))\ge |\pfr^{k_0}|$ for any $\xbf\in \pfr^{k_0+1}\ocal$. One can finish the argument as in the proof of Lemma~\ref{l:OpenFunction}.
\end{proof}

\begin{lem}\label{l:NondegenerateCurve}
Let $U$ be a neighborhood of $0\in K$, and $f_i:U\rightarrow K$ be analytic functions. Suppose $\one,f_1,\ldots,f_d$ are linearly independent. Then for large enough $m$ (depending on $f_i$) and large enough $l$ (depending on $f_i$ and $m$) we have that, for any $x_1,\ldots,x_m\in \pfr^{l} \ocal$,
\[
N([f_i'(x_j)])\gg \prod_{1\le i<j\le m} |(x_i-x_j)|,
\] 
where the implied constant depends on $f_i$. 
\end{lem}
\begin{proof}
Notice that we can rescale, i.e. change $f_i$ to $g_i(x):=f_i(\pfr^k x)$, and make sure that $|f_i^{(j)}(0)|\le 1$ for any $1\le i\le d$ and any positive integer $j$.  

Since $f_i$ are analytic and $\one ,f_1,\ldots,f_d$ are linearly independent, for large enough $m$ we have that $\rk([f_i^{(j)}(0)])=d$, where $1\le i\le d$ and $1\le j\le m$. So $N([f_i^{(j)}(0)])=|\pfr^{k_0}|$ for some non-negative integer $k_0$. Hence by the continuity of $\oPhi^jf_i$, for large enough $l$, we have that $N([\oPhi^jf_i(\xbf_j)])\gg 1$ if $\xbf_j\in \pfr^l \ocal^{j+1}$ for any $1\le j\le m$.

Let $F(x)$ be the column vector $[f_1(x) \cdots f_d(x)]^T$. Then by the repeated use of the third part of Lemma~\ref{l:LinearPart} we have
\begin{align*}
N([f_i'(x_j)])&=N[F'(x_1)\cdots F'(x_m)]\\
&\ge (d+1)^{-1} N[F'(x_1) (F'(x_2)-F'(x_1)) \cdots (F'(x_m)-F'(x_1))]\\
&\ge (d+1)^{-1} \left(\prod_{i=2}^m |x_i-x_1|\right) N[F'(x_1) \oPhi^1 F'(x_2,x_1) \cdots \oPhi^1 F'(x_m,x_1)]\\
&\ge \cdots \\
&\ge (d+1)^{-m} \left(\prod_{1\le i<j \le m}|x_i-x_j|\right) N[F'(x_1) \oPhi^1F'(x_2,x_1) \oPhi^2 F'(x_3,x_2,x_1) \cdots \oPhi^{m-1}F'(x_m,\ldots,x_1)]\\
&\gg \prod_{1\le i<j \le m}|x_i-x_j|. 
\end{align*}
\end{proof}
\begin{lem'}\label{l:NondegenerateCurve'}
Let $F:=(f_1,\ldots,f_{d_0}):\ocal\rightarrow \ocal^{d_0}, F(x)=\sum_{i}(c_{i,1}x^{i},\ldots,c_{i,d_0}x^{i})$ be an analytic function such that
\begin{enumerate}
\item $|c_{i,j}|\le 1$ for any $i$ and $j$,
\item $|\det(c_{i_e,j})|=|\pfr^{k_0}|$ for some indexes $i_1,\ldots,i_{d_0}\le m_0$ and some positive integer $k_0$.
\end{enumerate}
Then for any $l\gg_{k_0} 1$ we have that, for any $x_1,\ldots,x_{m_0}\in \pfr^l \ocal$,
\[
N([f_i'(x_j)])\ge |\pfr^{\Theta_{k_0,d_0,m_0}(1)}| \prod_{1\le i<j\le m_0} |x_i-x_j|.
\]
\end{lem'}
\begin{proof}
Since $|c_{i,j}|\le 1$ and $|\det(c_{i,j})|=|\pfr^{k_0}|$, we have that $N(\oPhi^j f_i(\xbf_j))\ge |\det(\oPhi^{i_e} f_j(\xbf_{i_e}))| =|\pfr^{k_0}|$ for any $\xbf_j\in \pfr^{k_0+1}\ocal^j$. Now one gets the claim as in the proof of Lemma~\ref{l:NondegenerateCurve}.
\end{proof}
\begin{lem}\label{l:SpanOfNondegenerateCurves}
Let $U$ be a neighborhood of $0\in K$, and $f_i:U\rightarrow K$ be analytic functions. Suppose $\one,f_1,\ldots,f_d$ are linearly independent. Let $F=(f_1,\ldots,f_d)$. Then for any $l\gg_F 1$ we have
\[\textstyle
\sum_{\Theta_F(1)} F(\pfr^l \ocal)-\sum_{\Theta_F(1)} F(\pfr^l \ocal)\supseteq \pfr^{\Theta_F(l)}\ocal^{d}.
\] 
\end{lem}
\begin{proof}
By Lemma~\ref{l:NondegenerateCurve} for large enough $m_0$ (depending on $F$) and large enough $l$ (depending on $F$ and $m_0$) we have 
\[
N(d\wt{F}(\xbf))\gg \prod_{1\le i<j\le m_0} |x_i-x_j|, 
\]
where $\wt{F}(\xbf)=F(x_1)+\cdots+F(x_{m_0})$ and $\xbf=(x_1,\ldots,x_{m_0})$ has norm at most $|\pfr^l|$. Let $\xbf_0=(x_1,\ldots,x_{m_0})$ be such that 
$|x_i|\le |\pfr^l|$ for any $i$ and $|x_i-x_j|\ge |\pfr^{2l}|$ for any $i\neq j$. So $N(d\wt{F}(\xbf_0))\gg \pfr^{\Theta_{F, m_0}(l)}.$
By rescaling, if needed, we can assume that $|f_i^{(j)}(0)|\le 1$ for any $1\le i\le d$ and any positive integer $j$. And so $\|d\wt{F}(\xbf_0)\|\le 1$ and $\|\partial_{ij}\wt{F}(\xbf_0)\|\le 1$. Therefore by Lemma~\ref{l:OpenFunction} we have
\[
\wt{F}(\pfr^l\ocal^{m_0})=\wt{F}(\xbf_0+\pfr^l\ocal^{m_0})\supseteq \wt{F}(\xbf_0)+\pfr^{\Theta_{m_0,F}(l)}\ocal^d,
\]
and so 
\[\textstyle
\sum_{m_0}F(\pfr^l \ocal)-\sum_{m_0}F(\pfr^l \ocal)=\wt{F}(\pfr^l\ocal^{m_0})-\wt{F}(\pfr^l\ocal^{m_0})\supseteq \pfr^{\Theta_{m_0,F}(l)}\ocal^d.
\]
\end{proof}
\begin{lem'}~\label{l:SpanOfNondegenerateCurves'}
Let $F:=(f_1,\ldots,f_{d_0}):\ocal\rightarrow \ocal^{d_0}, F(\xbf)=\sum_{i}(c_{i,1}x^{i},\ldots,c_{i,d_0}x^{i})$ be an analytic function such that
\begin{enumerate}
\item $|c_{i,j}|\le 1$ for any $i$ and $j$,
\item $|\det(c_{i_e,j})|=|\pfr^{k_0}|$ for some indexes $i_1,\ldots,i_{d_0}\le m_0$ and some positive integer $k_0$.
\end{enumerate}
Then for any $l\gg_{k_0,m_0} 1$ we have
\[
\textstyle
\sum_{m_0} F(\pfr^l\ocal)-\sum_{m_0} F(\pfr^l\ocal)\supseteq \pfr^{\Theta_{m_0,d_0,k_0}(l)}\ocal^{d_0}.
\]
\end{lem'}
\begin{proof}
As in the proof of Lemma~\ref{l:SpanOfNondegenerateCurves}, let $\widetilde{F}(\xbf)=\sum_{i=1}^{m_0}F(x_i)$. Hence by Lemma~\ref{l:NondegenerateCurve'} for $l\gg_{k_0} 1$ we have
\[
N(d\widetilde{F}(\xbf))\ge |\pfr^{\Theta_{k_0,d_0,m_0}(1)}|\prod_{1\le i<j\le m_0}|x_i-x_j|,
\]
for any $\xbf\in \pfr^l\ocal^{m_0}$. Let $\xbf_0\in \pfr^l \ocal^{m_0}$ be such that $|x_i-x_j|\ge |\pfr^{2l}|$ for any $i\neq j$. Notice that, since $l\gg_{m_0} 1$, there is such $\xbf_0$. Hence 
\[
N(d\widetilde{F}(\xbf_0))\ge |\pfr^{\Theta_{m_0,d_0,k_0}(l)}|.
\]
Let $F_{\xbf_0}(\xbf):=\widetilde{F}(\xbf+\xbf_0)$. Hence $F_{\xbf_0}$ has a Taylor series expansion and its coefficients have norm at most one. Moreover $N(dF_{\xbf_0}({\bf 0}))\ge |\pfr^{\Theta_{m_0,d_0,k_0}(l)}|$. Therefore by Lemma~\ref{l:OpenFunction'} we have
\[
F_{\xbf_0}(\pfr^l\ocal)\supseteq F_{\xbf_0}({\bf 0})+\pfr^{\Theta_{m_0,d_0,k_0}(l)}\ocal^{d_0}.
\]
One can finish the proof as above.
\end{proof}
\begin{lem}~\label{l:ManifoldToCurve}
Let $U\subseteq K^{n_0}$ be a non-empty open subset. Let $f_i:U\rightarrow K$ be analytic functions such that $\one,f_1,\ldots,f_d$ are linearly independent. Then there is a polynomial curve $r:K\rightarrow K^{n_0}$ such that $\one, f_1\circ r, \ldots, f_{d}\circ r$ are linearly independent, and defined on a neighborhood of $0\in K$.  
\end{lem}
\begin{proof}
Let $\xbf_0$ be a point in $U$. After rescaling, if needed, we can assume that $|\partial_{\ibf} f_j(\xbf_0)/\ibf!|\le 1$ for any $\ibf$ which is not zero. Since $\one, f_1,\ldots, f_d$ are linearly independent, for large enough $m$ we have that 
\[
\rk[\partial_{\ibf} f_j(\xbf_0)]_{1\le \|\ibf\|_1\le m,\h 1\le j\le d}=d.
\]  
Thus 
\be\label{e:Nondegenerate}
N([\partial_{\ibf} f_j(\xbf_0)/\ibf!])=|\pfr^{k_0}|,
\ee
where $k_0$ is a non-negative integer. Now let
\[
r(t):=\xbf_0+\pfr (t,t^s,t^{s^2},\ldots, t^{s^{n_0-1}})
\]
where $s$ is sufficiently large (to be determined later). By the Taylor expansion of $f_j$ we have
\[
f_j(r(t))=\sum_{\ibf} a^{(j)}_{\ibf} \pfr^{\|\ibf\|_1} t^{\sum_{k=1}^{n_0} i_ks^{k-1}},
\] 
where $\ibf=(i_1,\ldots,i_{n_0})$ and $a^{(j)}_{\ibf}=\partial_{\ibf}f_j/\ibf!$. And so
\be\label{e:TaylorExpansion}
f_j(r(t))=\sum_{n=0}^{\infty} \left(\sum_{\{\ibf|\sum_k i_ks^{k-1}=n\}} a^{(j)}_{\ibf}\pfr^{\|\ibf\|_1}\right) t^n. 
\ee
We make the following two observations:
\begin{enumerate}
	\item $\sum_{\{\ibf|\sum_k i_ks^{k-1}=n\}} a^{(j)}_{\ibf}\pfr^{\|\ibf\|_1}\equiv \sum_{\{\ibf|\|\ibf\|_{\infty}<s,\h\sum_k i_ks^{k-1}=n\}} a^{(j)}_{\ibf}\pfr^{\|\ibf\|_1} \pmod {\pfr^s}$,
	\item If $\|\ibf\|_{\infty}, \|\ibf'\|_{\infty}<s$ and $\sum_k i_ks^{k-1}=\sum_k i'_ks^{k-1}$, then $\ibf=\ibf'$.
\end{enumerate}
In particular we have 
\[
\sum_{\{\ibf|\sum_k i_ks^{k-1}=n\}} a^{(j)}_{\ibf}\pfr^{\|\ibf\|_1}\equiv 
\begin{cases}
0&\text{if }n\ge s^{n_0}\\
a^{(j)}_{\ibf(n)} \pfr^{\|\ibf(n)\|_1}&0\le n\le s^{n_0}-1
\end{cases}
\pmod{\pfr^s}
\]
where $\ibf(n):=(i_1,\ldots,i_{n_0})$ gives us the $s$-adic digits of $n$; that means $n=\sum_k i_ks^{k-1}$. Therefore we have
\be\label{e:Estimate1}
N\left(\left[\sum_{\{\ibf|\sum_k i_ks^{k-1}=n\}} a^{(j)}_{\ibf}\pfr^{\|\ibf\|_1}\right]_{1\le n\le s^{n_0}-1,\h 1\le j\le d}\right) \equiv N([a^{(j)}_{\ibf(n)} \pfr^{\|\ibf(n)\|_1}]) \pmod{\pfr^s}.
\ee
Now suppose $s>dm+k_0$; in particular, if $\|\ibf\|_1\le m$, then $\ibf=\ibf(n)$ for some $n<s^{n_0}$. And so we have
\be\label{e:Estimate2}
N([a^{(j)}_{\ibf(n)} \pfr^{\|\ibf(n)\|_1}])\ge |\pfr^{md+k_0}|. 
\ee
Hence by (\ref{e:Estimate1}), (\ref{e:Estimate2}) and $s>dm+k_0$ we have
\[
N\left(\left[\sum_{\{\ibf|\sum_k i_ks^{k-1}=n\}} a^{(j)}_{\ibf}\pfr^{\|\ibf\|_1}\right]_{1\le n\le s^{n_0}-1,\h 1\le j\le d}\right) \not\equiv 0 \pmod{\pfr^s},
\]
which implies that 
\[
\rk\left(\left[\sum_{\{\ibf|\sum_k i_ks^{k-1}=n\}} a^{(j)}_{\ibf}\pfr^{\|\ibf\|_1}\right]_{1\le n\le s^{n_0}-1,\h 1\le j\le d}\right)=d,
\]
and $\one, f_1\circ r, \ldots, f_d\circ r$ are linearly independent.  
\end{proof}

\begin{lem'}\label{l:ManifoldToCurve'}
Let $F:=(f_1,\ldots,f_{d_0}):\ocal^{n_0}\rightarrow \ocal^{d_0}$ be an analytic function of the form
\[
F(\xbf)=\sum_{\ibf}(c_{\ibf,1}\xbf^{\ibf},\ldots,c_{\ibf,d_0}\xbf^{\ibf}),
\]
where $\ibf=(i_1,\ldots,i_{n_0})$ ranges over multi-indexes non-negative integers and $\xbf^{\ibf}=x_1^{i_1}\cdots x_{n_0}^{i_{n_0}}$. Suppose 
	\begin{enumerate}
	\item $|c_{\ibf,j}|\le  1$ for any $\ibf=(i_1,\ldots,i_m)$ and any $1\le j\le d_0$,
	\item $|\det(c_{\ibf_i,j})|=|\pfr^{k_0}|$ for some multi-indexes $\ibf_1,\ldots,\ibf_{d_0}$ whose coordinates add up to a number at most $m_0$ and some positive integer $k_0$. 
	\end{enumerate}
Then there is a polynomial curve $r:\ocal\rightarrow \ocal^{n_0}$ such that $|\det(c'_{i_e,j})|\ge|\pfr^{\Theta_{m_0,d_0,k_0}(1)}|$ for some indexes $i_e\ll_{m_0,n_0,d_0,k_0} 1$ where
$
f_j\circ r(x)=\sum_i c'_{i,j}x^i.
$
\end{lem'}
\begin{proof}
A close examination of proof of Lemma~\ref{l:ManifoldToCurve} yields that (\ref{e:Estimate1}), (\ref{e:Estimate2}) imply 
\[
N([c'_{i,j}]_{1\le i\le (d_0m_0+k_0+1)^{n_0}-1, 1\le j\le d})\ge |\pfr^{m_0d_0+k_0}|,
\]
which gives us the claim.
\end{proof}
\begin{cor}\label{c:NondegenerateManifoldCase}
Let $U\subseteq K^{m_0}$ be a non-empty open set. Let $f_j:U\rightarrow K$ be analytic functions, and $F:=(f_1,\ldots,f_d)$. Suppose $\one,f_1,\ldots, f_d$ are linearly independent. Then for any $l\gg_F 1$ we have
\[\textstyle
\sum_{\Theta_F(1)} F(\pfr^l \ocal)-\sum_{\Theta_F(1)} F(\pfr^l \ocal)\supseteq \pfr^{\Theta_F(l)}\ocal^{d}.
\]  
\end{cor}
\begin{proof}
By Lemma~\ref{l:ManifoldToCurve} there is a polynomial curve $r$ such that $\one, f_1\circ r, \ldots, f_d\circ r$ are linearly independent and defined on a neighborhood of $0\in K$. Hence by Lemma~\ref{l:NondegenerateCurve} we are done. 
\end{proof}
\begin{cor'}\label{c:NondegenerateManifoldCase'}
Let $F:=(f_1,\ldots,f_{d_0}):\ocal^{n_0}\rightarrow \ocal^{d_0}$ be an analytic function of the form
\[
F(\xbf)=\sum_{\ibf}(c_{\ibf,1}\xbf^{\ibf},\ldots,c_{\ibf,d_0}\xbf^{\ibf}),
\]
where $\ibf=(i_1,\ldots,i_{n_0})$ ranges over multi-indexes non-negative integers. Suppose 
	\begin{enumerate}
	\item $|c_{\ibf,j}|\le  1$ for any $\ibf=(i_1,\ldots,i_m)$ and any $1\le j\le d_0$,
	\item $|\det(c_{\ibf_i,j})|=|\pfr^{k_0}|$ for some multi-indexes $\ibf_1,\ldots,\ibf_{d_0}$ whose coordinates add up to a number at most $m_0$ and some positive integer $k_0$. 
	\end{enumerate}
	
Then for $l\gg_{m_0,n_0,d_0,k_0} 1$ we have 
\[\textstyle
\sum_{\Theta(1)}F(\pfr^l\ocal)-\sum_{\Theta(1)}F(\pfr^l\ocal)\supseteq 
\pfr^{\Theta(l)}\ocal^{d_0},
\]
where all the implied constants depend on $m_0,n_0,d_0,k_0$.
\end{cor'}
\begin{proof}
This is direct corollary of Lemma~\ref{l:SpanOfNondegenerateCurves'} and Lemma~\ref{l:ManifoldToCurve'}.
\end{proof}

\begin{proof}[Proof of Proposition~\ref{p:QuanOpenImage}]
Suppose $\{f_{i_1},\ldots,f_{i_d}\}$ is a basis of $\sum_i Kf_i$. So $\one, f_{i_1},\ldots,f_{i_d}$ are linearly independent. And moreover there is an injective linear map $L:K^{d}\rightarrow K^{d_0}$ such that $F=L\circ \overline{F}$, where $\overline{F}(\xbf)=(f_{i_1}(\xbf),\ldots,f_{i_d}(\xbf))$.
By Corollary~\ref{c:NondegenerateManifoldCase} we have that 
\[\textstyle
\sum_{\Theta_{\overline{F}}(1)} {\overline{F}}(\pfr^l \ocal)-\sum_{\Theta_{\overline{F}}(1)} {\overline{F}}(\pfr^l \ocal)\supseteq \pfr^{\Theta_{\overline{F}}(l)}\ocal^{d}.
\] 
for any $l\gg_{\overline{F}} 1$. Applying $L$ to the both sides and using the linearity of $L$ we have
\[\textstyle
\sum_{\Theta_F(1)} F(\pfr^l \ocal)-\sum_{\Theta_F(1)} F(\pfr^l \ocal)\supseteq \pfr^{\Theta_F(l)}L(\ocal^{d})\supseteq \pfr^{\Theta_F(l)} \ocal^{d_0} \cap {\rm Im}(L).
\]
\end{proof}
\begin{proof}[Proof of Proposition~\ref{p:QuanOpenImage'}]
By a similar argument as in the proof of Proposition~\ref{p:QuanOpenImage}, Corollary~\ref{c:NondegenerateManifoldCase'} implies the desired result. 
\end{proof}
\begin{proof}[Proof of Corollary \ref{c:UnifOpenImage}]
Let $\ocal_{\kappa}$ be the ring of integers of $\kappa$. There is $a\in \ocal_{\kappa}$ such that for any $i$
\[
\overline{f}_i(\xbf):=f_i(a\xbf)\in \ocal_{\kappa}[x_1,\ldots,x_{n_0}].
\]
Suppose $\overline{f}_{j_1},\ldots,\overline{f}_{j_{d_0}'}$ is a basis of $\sum_{i=1}^{d_0}\kappa \overline{f}_i$. So there is a linear embedding $L:\kappa^{d_0'}\rightarrow \kappa^{d_0}$ such that 
\[
(\overline{f}_1(\xbf),\ldots,\overline{f}_{d_0}(\xbf))=L(\overline{f}_{j_1},\ldots,\overline{f}_{j_{d_0'}}).
\] 
So the operator norm of the induced linear embedding $L_{\pfr}:\kappa_{\pfr}^{d_0'}\rightarrow \kappa_{\pfr}^{d_0}$ are uniformly bounded from infinity and zero. 

Since $\overline{f}_{j_1},\ldots,\overline{f}_{j_{d_0'}}$ are linearly independent, there are multi-indexes $\ibf_1,\ldots,\ibf_{d_0'}$ such that 
\[
0\neq \det(c_{\ibf_i,j_u})\in \ocal_{\kappa},
\]
where $\overline{f}_j(\xbf)=\sum_{\ibf} c_{\ibf,j} \xbf^{\ibf}$. So $|\det(c_{\ibf_i,j_u})|_{\pfr}=1$ for almost all $\pfr\in V_f(\kappa)$. 

By the above discussion, one can easily finish the proof using Proposition~\ref{p:QuanOpenImage'}. 
\end{proof}

\end{document}